\documentclass[10pt,reqno]{article}

\usepackage[b5paper,top=1.2in,left=0.9in]{geometry}
\usepackage{verbatim}
\usepackage{amssymb}
\usepackage{amsmath}
\usepackage{graphicx}
\usepackage{appendix}
\usepackage{color}
\usepackage{amsthm}
\usepackage{tikz}
\usepackage{float}
\usepackage{subfloat}
\usepackage{pdflscape}\usepackage{ifthen}
\usetikzlibrary{arrows,positioning,decorations.pathmorphing,decorations.markings, calc} 
\tikzset{>=latex}
\tikzstyle{vthick}=[line width=2.5pt]
\usetikzlibrary{external}
\renewcommand{\d}{\mathrm{d}}
\newcommand{\D}{\mathrm{D}}
\newcommand{\e}{\mathrm{e}}

\newtheorem{Thm}{Theorem}[section]
\newtheorem{Lem}[Thm]{Lemma}
\newtheorem{Prop}[Thm]{Proposition}
\newtheorem{Cor}[Thm]{Corollary}
\newtheorem{Rem}[Thm]{Remark}
\newtheorem{Def}[Thm]{Definition}
\newtheorem{Con}[Thm]{Conjecture}
\newtheorem{Ex}[Thm]{Example}

\newtheoremstyle{named}{}{}{\itshape}{}{\bfseries}{.}{.5em}{#1 #3}
\theoremstyle{named}

\def\R{\mathbb{R}}
\def\Q{\mathbb{Q}}
\def\N{\mathbb{N}}
\def\C{\mathbb{C}}
\def\Z{\mathbb{Z}}

\def\fD{\mathfrak{D}}

\def\fR{\mathfrak{R}}

\def\fn{\mathfrak{n}}

\def\fb{\mathfrak{b}}
\def\g{\mathfrak{g}}
\def\fh{\mathfrak{h}}

\def\sl{\mathfrak{sl}}

\def\cA{\mathcal{A}}

\def\cD{\mathcal{D}}
\def\cE{\mathcal{E}}
\def\cF{\mathcal{F}}

\def\cK{\mathcal{K}}

\def\cM{\mathcal{M}}

\def\cO{\mathcal{O}}
\def\cP{\mathcal{P}}

\def\cR{\mathcal{R}}
\def\cS{\mathcal{S}}
\def\cT{\mathcal{T}}
\def\cU{\mathcal{U}}

\def\cX{\mathcal{X}}

\def\cZ{\mathcal{Z}}

\def\a{\alpha}
\def\b{\beta}
\def\c{\gamma}

\def\D{\Delta}

\def\d{\delta}
\def\e{\epsilon}
\def\ze{\zeta}

\def\h{\theta}

\def\l{\lambda}
\def\L{\Lambda}

\def\s{\sigma}

\def\be{\mathbf{e}}

\def\bf{\mathbf{f}}

\def\bi{\mathbf{i}}

\def\bo{\mathbf{o}}

\def\bT{\mathbf{T}}

\def\ex{\mathbf{ex}}

\def\=>{\Longrightarrow}
\def\inj{\hookrightarrow}
\def\corr{\longleftrightarrow}

\def\to{\longrightarrow}

\def\ox{\otimes}
\def\o+{\oplus}
\def\bo+{\bigoplus}
\def\x{\times}
\def\<{\langle}
\def\>{\rangle}
\def\({\left(}
\def\){\right)}
\def\oo{\infty}

\def\^{\wedge}
\def\+{\dagger}

\def\inv{^{-1}}
\def\half{\frac{1}{2}}
\def\dis{\displaystyle}
\def\dd[#1,#2]{\frac{d#1}{d#2}}
\def\del[#1,#2]{\frac{\partial #1}{\partial #2}}

\def\v[#1,#2]{(\substack{#1\\#2})}
\def\vv[#1,#2,#3]{(\substack{#1\\#2\\#3})}
\def\tab{\;\;\;\;\;\;}

\newcommand{\til}[1]{\widetilde{#1}}
\newcommand{\what}[1]{\widehat{#1}}
\newcommand{\xto}[1]{\xrightarrow{#1}}
\newcommand{\ov}[1]{\overline{#1}}
\newcommand{\ve}[1]{\overrightarrow{#1}}

\newcommand{\case}[2][llllllllllllllllllllllllllllllllllllllllllll]{\left\{\begin{array}{#1}#2 \\ \end{array}\right.}
\newcommand{\Eq}[1]{\begin{align}#1\end{align}}
\newcommand{\Eqn}[1]{\begin{align*}#1\end{align*}}

\newcommand\drawpath[2]{%
  \foreach \too [count=\c from 1] in {#1}
  {
  \ifthenelse{\c=1}
  {\xdef\from{\too}}
  {\path (\from) edge [->, #2] (\too);
    \xdef\from{\too}}
  };
}

\begin{document}
\title{Cluster Realization of $\cU_q(\g)$ and \\Factorizations of the Universal $R$-Matrix}

\author{  Ivan C.H. Ip\footnote{
         	   Center for the Promotion of Interdisciplinary Education and Research/\newline
     	  Department of Mathematics, Graduate School of Science, Kyoto University, Japan
		\newline
		Email: ivan.ip@math.kyoto-u.ac.jp
          }
}

\date{\today}

\numberwithin{equation}{section}
\maketitle
\begin{abstract}
For each simple Lie algebra $\g$, we construct an algebra embedding of the quantum group $\cU_q(\g)$ into certain quantum torus algebra $\cD_\g$ via the positive representations of split real quantum group. The quivers corresponding to $\cD_\g$ is obtained from amalgamation of two basic quivers, where each of them is mutation equivalent to the cluster structure of the moduli space of framed $G$-local system on a disk with 3 marked points when $G$ is of classical type. We derive a factorization of the universal $R$-matrix into quantum dilogarithms of cluster monomials, and show that conjugation by the $R$-matrix corresponds to a sequence of quiver mutations which produces the half-Dehn twist rotating one puncture about the other in a twice punctured disk.
\end{abstract}

{\small {\textbf{Keywords.} Quantum groups, positive representations, cluster algebra, $R$-matrix }

\vspace{5mm}


\tableofcontents

%==============================================================================
\section{Introduction}\label{sec:intro}
For any finite dimensional complex simple Lie algebra $\g$, Drinfeld \cite{D} and Jimbo \cite{J} associated to it a remarkable Hopf algebra $\cU_q(\g)$ known as \emph{quantum group}, which is certain deformation of the universal enveloping algebra. To better understand the structure of $\cU_q(\g)$, a very natural problem is to find certain embeddings into simpler algebras. In \cite{GKL, GKLO}, through the generalization of Gelfand-Tsetlin representations, embeddings of the whole quantum group $\cU_q(\g)$ into certain \emph{rational functions} $\C(\bT_q)$ of quantum torus have been found. Another well-known result is provided by \emph{Feigin's homomorphism} \cite{Ber, Rup} which embeds the positive Borel part $\cU_q(\fb_+)$ of $\cU_q(\g)$ directly into a \emph{quantum torus algebra} $\C[\bT_q]$. However, the explicit extension of Feigin's map to the whole quantum group, i.e. given by \emph{polynomial} embeddings of $\cU_q(\g)$ into certain quantum torus algebra, appears to be much more subtle. While the case for $\cU_q(\sl_n)$ is known previously \cite{KV}, the case for general types has only been solved recently with the introduction of \emph{positive representations} of split real quantum groups.

\subsection{Quantum group embeddings via positive representations}

The notion of \emph{positive representations} was introduced in \cite{FI} as a new research program devoted to the representation theory of split real quantum groups $\cU_{q}(\g_\R)$ and its modular double $\cU_{q\til{q}}(\g_\R)$ introduced in \cite{Fa1, Fa2}, in the regime where $|q|=1$. It is motivated by the simplest case $\cU_{q\til{q}}(\sl(2,\R))$ which has been studied extensively by Teschner \textit{et al.} \cite{BT, PT1, PT2} from the point of view of non-compact conformal field theory. Explicit construction of the positive representations $\cP_\l$ of $\cU_{q\til{q}}(\g_\R)$ associated to a simple Lie algebra $\g$ has been obtained for the simply-laced case in \cite{FI, IpTh, Ip2} and non-simply-laced case in \cite{Ip3}, where the generators of the quantum groups are realized by \emph{positive essentially self-adjoint operators} acting on certain Hilbert spaces. 

As a consequence of the construction, if one \emph{forgets} the real structure of such representations, one can express the generators in terms of Laurent polynomials of certain $q$-commuting variables, and we obtain a full embedding of quantum groups
\Eq{\cU_q(\g)\inj \C[\bT_q]}
into certain quantum torus algebra, thus solving the long-standing problem of generalizing the Feigin's homomorphism. 

The construction of the positive representations of $\cU_q(\g_\R)$ relies heavily on Lusztig's total positivity of reductive groups and is closely related to the structure of the quantum principal affine space $\cO_q[G/N]$. Its harmonic analysis on $L^2(G_{q\til{q}}^+(\R))$ through the Gauss-Lusztig decomposition \cite{IpRIMS, Ip5} also involves the structure of the coordinate ring $\cO_q[G]$ and the double Bruhat cell $\cO_q[G^{w_0,w_0}]$. Therefore the theory of positive representations is long considered to have a strong connection to the theory of quantum cluster algebra \cite{BZ} in which these objects represent \cite{GLS, GY}. In particular both theories share a similar \emph{positivity phenomenon} under some mutation operations, where for example the generators of $\cU_q(\g_\R)$ are always represented as Laurent polynomials of positive operators with positive $q$-integral coefficients, thus naturally acting on $\cP_\l$ as positive self-adjoint operators.

In a recent work of \cite{SS2}, Schrader and Shapiro found explicitly an embedding of $\cU_q(\sl_n)$ into certain quantum torus algebra $\cD_{\sl_n}$, which arises from quantizing the Fock and Goncharov's construction of the cluster coordinates on the moduli spaces of framed $PGL_n$-local systems on the punctured disk with two marked points, where the structure can be nicely summarized into certain quiver diagrams given by \emph{$n$-triangulations} \cite{FG1}. It turns out that their construction fit nicely into the framework of positive representations, and one can carry over the explicit constructions of $\cP_\l$ and obtain a new quantum torus algebra embedding for arbitrary type of $\cU_q(\g)$.

Our \textbf{first main result} (Theorem \ref{mainThm}) states that there is an embedding of algebra
\Eq{\label{emb}
\cU_q(\g)\inj \cD_\g/\sim
}
into a quantum torus algebra (modulo some central elements), which can be represented by some quiver diagrams associated to $\cD_\g$. The embeddings of the generators of $\cU_q(\g)$ can then be expressed explicitly by certain paths on the quiver. In particular, the previously rather \textit{ad hoc} explicit expressions, especially in the exceptional types, can now be visualized in a very simple manner (see Figure \ref{fig-An} - Figure \ref{fig-G2}).

Furthermore, a change of words of the reduced expression of the longest element $w_0\in W$ of the Weyl group, which induces unitary equivalences of the positive representations, correspond to certain quiver mutations and hence quantum cluster mutations of $\cD_\g$. This makes the connection between positive representations and the theory of (quantum) cluster algebra much more explicit. It strongly suggests that in fact we have an embedding into the global functions on the corresponding cluster $\cX$-variety
\Eq{
\cU_q(\g)\inj\bigcap_\bi \cD_\g^\bi/\sim
} associated to all the seed equivalence class of $\cD_\g$, where the generators of $\cU_q(\g)$ stay polynomial in any cluster. However this requires a separate proof and will be considered in future works.

Finally, the proof of injectivity of the embedding in type $A_n$ by \cite{SS2} involves explicitly some combinatorial \emph{hive-type} conditions related to the work of Knutson-Tao \cite{KTao}. It will be interesting to see the analogues of such combinatorics coming from other types of quantum groups from our construction using positive representations.

\subsection{Basic quivers and framed $G$-local systems}
The quiver corresponding to $\cD_\g$ is naturally associated to the triangulation of a punctured disk with two marked points. It can be constructed by gluing (amalgamating) two copies of \emph{``basic quivers"} $Q$ associated to a triangle. It turns out that the basic quiver is mutation equivalent to the quiver giving a (classical) cluster algebra structure on the moduli space of framed $G$-local system, or the configuration space $Conf_3\cA_G$ of triples of principal flags, recently discovered for classical types \cite{Le1} and type $G_2$ \cite{Le2}. Both constructions require the use of \emph{elementary quivers} associated to simple reflections of the longest element $w_0$.

In particular, by providing a different construction than the ones in \cite{Le1, Le2}, the description of $Q$ in this paper may allow us to construct quantum higher Teichm\"uller theory in full generality in a representation theoretical setting of quantum groups, where the quiver describes the coordinates of the framed $G$-local system and their Poisson structure, and hence also the quantization of these coordinates. The uniqueness of $Q$ can also potentially be used to solve the series of conjectures proposed in \cite{Le2}. We also expect that such geometric description of the basic quivers will let us better understand the geometric construction of another quantum group embedding via the Grothendieck-Springer resolution proposed by \cite{SS1}, which turns out to be quite hard to write down explicitly. 

The basic quiver plays an important role in the description of the universal $R$-matrix realized as half-Dehn twist, which we will described next.

\subsection{Universal $R$-matrix as half-Dehn twist}
Using the quantum cluster embedding \eqref{emb}, our \textbf{second main result} (Theorem \ref{main2} and Corollary \ref{CorR1234}) of the paper gives the factorization of the reduced $R$ matrix into products of quantum dilogarithms such that the arguments are given by monomials of the quantum cluster variables $X_i\in \cD_\g$ associated to the chosen reduced expression of $w_0$, and the factorization is invariant under the change of reduced expression. 

This result generalizes the factorization of \cite{SS2} in type $A_n$ for a specific choice of $w_0$, and the earlier well-known result for $\cU_q(\sl_2)$ by Faddeev \cite{Fa2}. It is different from the usual multiplicative formula discovered independently by Kirillov-Reshetikhin \cite{KR} and Levendorskii-Soibelman \cite{LS1,LS2}, which is further extended to the superalgebra case in \cite{KT}. Since each factor is expressed in terms of quantum cluster variables, in fact it can be viewed as a sequence of quiver mutations on two copies of the $\cD_\g$-quiver associated to a disk of two punctures and two marked points. 

Our \textbf{final main result} of the paper (Theorem \ref{main3}) shows that the conjugation by the universal $R$-matrix corresponds to a sequence of quiver mutations which produces the \emph{half-Dehn twist} rotating one puncture about the other in the twice punctured disk. This factorization can also be split into 4 blocks such that each block corresponds to a \emph{flip} of triangulations of the twice punctured disk, where the basic quiver $Q$ associated to each triangle is being mutated to a different configuration. 

In the case of $\cU_q(\sl_2)$, such identification of the factorization of the $R$-matrix appears in quantum Teichm\"uller theory \cite{Ka} as an element of the mapping class group, and the corresponding factorization is also used to re-derive Kashaev's knot invariant \cite{HI}. For general Hopf algebra $\cA$, Kashaev has constructed an embedding $\phi:\cD(\cA)\to H(A)\ox H(A)^{op}$ of the Drinfeld's double $\cD(\cA)$ into a tensor square of the Heisenberg double $H(\cA)$, and the image of the universal $R$-matrix can be similarly decomposed into a product of 4 variants of the \emph{$S$-tensors} \cite{Ka97}:
\Eq{
\phi^{\ox2}(\cR)=S_{14}''S_{13}\til{S}_{24}S_{23}'\in (H(A)\ox H(A)^{op})^{\ox2}.
}
This has been utilized for example to construct new quantum invariant for ``colored triangulations" of topological spaces recently \cite{Su}. Although the two factorizations are realized on different tensor spaces, we believe there is a strong connection between the two different factorizations, where $\cA$ is identified with the Borel part of $\cU_q(\g)$, and it will be interesting to find an explicit relationship between them. We hope that the factorization in this paper opens up a new class of invariants which can be explicitly constructed.

\subsection{Generalization to the split real setting}
The embeddings of quantum groups as well as the factorization of $R$-matrix in this paper is treated in a formal algebraic setting. However, as the construction comes from the positive representations of split real quantum groups, it is natural to conclude that the theory developed in this paper can be generalized to the split real setting. For example, the monomials of the embedding constructed out of the quantum cluster variables $X_i\in \cD_\g$ are all manifestly positive self-adjoint if we put back in the split real form. 

In particular, throughout the paper, we use the correspondence (see Remark \ref{qd} for more details):
 \Eq{
 \Psi^q(x)\sim g_b^*(x)} to identify both the \emph{compact} and \emph{non-compact} quantum dilogarithm functions. This suggests that in fact all the quiver mutations and $R$-matrix decomposition work in the split real setting. In this case the non-compact version is well-defined as the quantum cluster variables are positive self-adjoint, therefore the formal power series manipulations can be replaced by actions of unitary operators. 
 
Furthermore, Faddeev's modular double can be easily recovered by applying the \emph{transcendental relations} \cite{Fa2,FI} to the quantum cluster variables:
\Eq{
\til{X_i}:=X_i^{\frac{1}{b_i}},
}
 and the simple analytic version of the Langlands duality \cite{Ip3} interchanging the long and short roots can then be easily recovered as well (this is made more explicit in the quiver diagrams of type $B_n, C_n$ and $G_2$). The perspectives of the applications of such phenomenon in the split real case look very promising, and will be explored elsewhere.

\subsection{Outline of the paper}
The paper is organized as follows. In Section \ref{sec:def}, we fix the convention used throughout the paper and recall the definition of quantum group $\cU_q(\g)$. In Section \ref{sec:tori}, we recall the definition and properties of quantum torus algebra, the associated quivers, and their cluster structure. In Section \ref{sec:pos}, we recall the construction of the positive representations of split real quantum groups, and define the new quantum torus algebra $\cD_\g$ in which $\cU_q(\g)$ embeds. In Section \ref{sec:quiver} we construct explicitly the $\cD_\g$-quiver associated to the algebra $\cD_\g$ using the \emph{elementary quivers}, and in Section \ref{sec:embedding} we give an explicit embedding of $\cU_q(\g)$ for all each simple types of $\g$, where the generators are represented by certain paths on the quivers. 

In Section \ref{sec:mutation} we discuss the quiver mutations associated to a change of reduced expression of the longest element $w_0\in W$ of the Weyl group, and we use this to show in Section \ref{sec:basicquiver} that the \emph{basic quiver} associated to a triangle of a triangulation is uniquely defined. 

In Section \ref{sec:R} we recall the definition of universal $R$-matrix, and using the quantum group embedding, we give a factorization formula of $R$, which is proved in Section \ref{sec:Thm}. Finally in Section \ref{sec:Dehn}, we show that the factorization of $R$ can be realized as half-Dehn twist of a twice punctured disk with two marked points, where the basic quiver associated to each triangle is mutated to certain new configurations, and we give explicitly its sequence of mutations.
%==============================================================================
\section{Notations and definitions of $\cU_q(\g)$}\label{sec:def}

In order to fix the convention we use throughout the paper, we follow the notations used in \cite{Ip3, Ip4} for the root systems and recall the definition of the quantum group $\cU_q(\g)$, where $\g$ is of general type \cite{CP}, as well as the Drinfeld's double $\fD_\g$ of the Borel part.

\begin{Def}
Let $I$ denote the set of nodes of the Dynkin diagram of $\g$ where 
\Eq{
|I|=n=rank(\g).}
Let $w_0\in W$ be the longest element of the Weyl group of $\g$, and let
\Eq{
N:=l(w_0)=\dim \fn_-}
be its length, which is also the dimension of the unipotent subgroup $\fn_-$ of $\g$. 

We call a sequence $\bi=(i_1,...,i_N)\in I^N$ a reduced word of $w_0$ if $w_0=s_{i_1}...s_{i_N}$ is a reduced expression, where $s_{i_k}$ are the simple reflections of the root space\footnote{We will sometimes omit the commas in $\bi$ for typesetting purpose.}. We denote by $\fR$ the set of all reduced words of $w_0$. 

We let $n_i^\bi \in \Z_{>0}$ be the number of letters $i$ appearing in $\bi$. We will write $n_i:=n_i^\bi$ if no confusion arises. 
If we have another reduced word $\bi'\in\fR$, we will sometimes write $\bi'=(i_1',...,i_N')$ and $n_i':=n_i^{\bi'}$. 

Clearly we have
\Eq{
\sum_{i=1}^n n_i=N.
}

\end{Def}

\begin{Def}\label{Dynkin}
We index the nodes of the Dynkin diagrams as follow, where black nodes correspond to short roots, and white nodes correspond to long roots.

\begin{itemize}
\item Type $A_n$: 
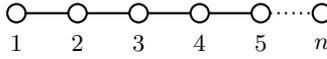
\begin{center}
  \begin{tikzpicture}[scale=.4]
    \draw[xshift=0 cm,thick] (0 cm, 0) circle (.3 cm);
    \foreach \x in {1,...,5}
    \draw[xshift=\x cm,thick] (\x cm,0) circle (.3cm);
    \draw[dotted,thick] (8.3 cm,0) -- +(1.4 cm,0);
    \foreach \y in {0.15,...,3.15}
    \draw[xshift=\y cm,thick] (\y cm,0) -- +(1.4 cm,0);
    \foreach \z in {1,...,5}
    \node at (2*\z-2,-1) {$\z$};
\node at (10,-1){$n$};
  \end{tikzpicture}
\end{center}
\item Type $B_n$: 
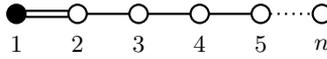
\begin{center}
  \begin{tikzpicture}[scale=.4]
    \draw[xshift=0 cm,thick,fill=black] (0 cm, 0) circle (.3 cm);
    \foreach \x in {1,...,5}
    \draw[xshift=\x cm,thick] (\x cm,0) circle (.3cm);
    \draw[dotted,thick] (8.3 cm,0) -- +(1.4 cm,0);
    \foreach \y in {1.15,...,3.15}
    \draw[xshift=\y cm,thick] (\y cm,0) -- +(1.4 cm,0);
    \draw[thick] (0.3 cm, .1 cm) -- +(1.4 cm,0);
    \draw[thick] (0.3 cm, -.1 cm) -- +(1.4 cm,0);
    \foreach \z in {1,...,5}
    \node at (2*\z-2,-1) {$\z$};
\node at (10,-1){$n$};
  \end{tikzpicture}
\end{center}
\item Type $C_n$: 
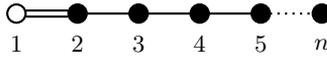
\begin{center}
  \begin{tikzpicture}[scale=.4]
    \draw[xshift=0 cm,thick] (0 cm, 0) circle (.3 cm);
    \foreach \x in {1,...,5}
    \draw[xshift=\x cm,thick,fill=black] (\x cm,0) circle (.3cm);
    \draw[dotted,thick] (8.3 cm,0) -- +(1.4 cm,0);
    \foreach \y in {1.15,...,3.15}
    \draw[xshift=\y cm,thick] (\y cm,0) -- +(1.4 cm,0);
    \draw[thick] (0.3 cm, .1 cm) -- +(1.4 cm,0);
    \draw[thick] (0.3 cm, -.1 cm) -- +(1.4 cm,0);
    \foreach \z in {1,...,5}
    \node at (2*\z-2,-1) {$\z$};
\node at (10,-1){$n$};
  \end{tikzpicture}
\end{center}
\item Type $D_n$: 
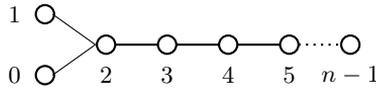
\begin{center}
  \begin{tikzpicture}[scale=.4]
    \draw[xshift=0 cm,thick] (0 cm, 1) circle (.3 cm);
    \draw[xshift=0 cm,thick] (0 cm, -1) circle (.3 cm);
    \foreach \x in {1,...,5}
    \draw[xshift=\x cm,thick] (\x cm,0) circle (.3cm);
    \draw[dotted,thick] (8.3 cm,0) -- +(1.4 cm,0);
   \draw[xshift=0.25 cm] (0 cm,1) -- +(1.4 cm,-1);
   \draw[xshift=0.25 cm] (0 cm,-1) -- +(1.4 cm,1);   
 \foreach \y in {1.15,...,3.15}
    \draw[xshift=\y cm,thick] (\y cm,0) -- +(1.4 cm,0);
    \foreach \z in {2,...,5}
    \node at (2*\z-2,-1) {$\z$};
\node at (10,-1){$n-1$};
\node at (-1,-1){$0$};
\node at (-1,1){$1$};
  \end{tikzpicture}
\end{center}
\item Type $E_n$:
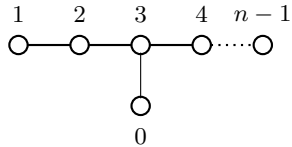
\begin{center}
  \begin{tikzpicture}[scale=.4]
    \draw[xshift=0 cm,thick] (0 cm, 0) circle (.3 cm);
    \foreach \x in {1,...,4}
    \draw[xshift=\x cm,thick] (\x cm,0) circle (.3cm);
    \foreach \y in {0.15,...,2.15}
    \draw[xshift=\y cm,thick] (\y cm,0) -- +(1.4 cm,0);
    \draw[xshift=3.15cm, dotted, thick] (3.15cm,0) --+(1.4 cm,0);
    \foreach \z in {1,...,4}
    \node at (2*\z-2,1) {$\z$};
    \node at (8,1) {$n-1$};
\draw[xshift=0 cm,thick] (4 cm, -2) circle (.3 cm);
  \draw[xshift=0 cm] (4 cm,-0.25) -- +(0 cm,-1.5);
\node at (4,-3){$0$};
  \end{tikzpicture}
\end{center}
\item Type $F_4$:
 
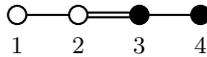
\begin{center}
  \begin{tikzpicture}[scale=.4]
    \draw[thick] (-2 cm ,0) circle (.3 cm);
	\node at (-2,-1) {$1$};
    \draw[thick] (0 ,0) circle (.3 cm);
	\node at (0,-1) {$2$};
    \draw[thick,fill=black] (2 cm,0) circle (.3 cm);
	\node at (2,-1) {$3$};
    \draw[thick,fill=black] (4 cm,0) circle (.3 cm);
	\node at (4,-1) {$4$};
    \draw[thick] (15: 3mm) -- +(1.5 cm, 0);
    \draw[xshift=-2 cm,thick] (0: 3 mm) -- +(1.4 cm, 0);
    \draw[thick] (-15: 3 mm) -- +(1.5 cm, 0);
    \draw[xshift=2 cm,thick] (0: 3 mm) -- +(1.4 cm, 0);
  \end{tikzpicture}
\end{center}
\item Type $G_2$: 
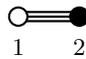
\begin{center}
  \begin{tikzpicture}[scale=.4]
    \draw[thick] (0 ,0) circle (.3 cm);
	\node at (0,-1) {$1$};
    \draw[thick,fill=black] (2 cm,0) circle (.3 cm);
	\node at (2,-1) {$2$};
    \draw[thick] (30: 3mm) -- +(1.5 cm, 0);
    \draw[thick] (0: 3 mm) -- +(1.5 cm, 0);
    \draw[thick] (-30: 3 mm) -- +(1.5 cm, 0);
  \end{tikzpicture}
\end{center}
\end{itemize}
\end{Def}

\begin{Def} \label{qi} Let $q$ be a formal parameter. Let $\{\a_i\}_{i\in I}$ be the set of positive simple roots. Let $(-,-)$ be the inner product of the root lattice, and we define
\Eq{
a_{ij}:=\frac{2(\a_i,\a_j)}{(\a_i,\a_i)},
}
such that $A:=(a_{ij})$ is the \emph{Cartan matrix}. 

We normalize $(-,-)$ as follows: we choose the symmetrization factors (also called the \emph{multipliers})
\Eq{d_i:=\frac{1}{2}(\a_i,\a_i)=\case{1&\mbox{$i$ is long root or in the simply-laced case,}\\\frac{1}{2}&\mbox{$i$ is short root in type $B,C,F$,}\\\frac{1}{3}&\mbox{$i$ is short root in type $G_2$,}}}and $(\a_i,\a_j)=-1$ when $i,j$ are adjacent in the Dynkin diagram, such that
$$d_ia_{ij}=d_ja_{ji}.$$ We then define
\Eq{
q_i:=q^{d_i},
} which we will also write as
\Eq{
q_l&:=q,\\ 
q_s&:=\case{q^{\frac12}&\mbox{$\g$ is of type $B_n, C_n, F_4$},\\q^{\frac13}&\mbox{$\g$ is of type $G_2$},}
}
for the $q$ parameters corresponding to long and short roots respectively.
\end{Def}

\begin{Def} Let $A=(a_{ij})$ denote the Cartan matrix. We define $\fD_\g$ to be the $\C(q_s)$-algebra generated by the elements
$$\{E_i, F_i,K_i^{\pm1}, K_i'^{\pm1}| i\in I\}$$ 
subject to the following relations (we will omit the relations involving $K_i^{-1}, {K_i'}^{-1}$ below for simplicity):
\Eq{
K_iE_j&=q_i^{a_{ij}}E_jK_i, &K_iF_j&=q_i^{-a_{ij}}F_jK_i,\\
K_i'E_j&=q_i^{-a_{ij}}E_jK_i', &K_i'F_j&=q_i^{a_{ij}}F_jK_i',\\
K_iK_j&=K_jK_i, &K_i'K_j'&=K_j'K_i', &K_iK_j' = K_j'K_i,\\
&&[E_i,F_j]&= \d_{ij}\frac{K_i-K_i'}{q_i-q_i\inv},\label{EFFE}
}
together with the \emph{Serre relations} for $i\neq j$:
\begin{eqnarray}
\sum_{k=0}^{1-a_{ij}}(-1)^k\frac{[1-a_{ij}]_{q_i}!}{[1-a_{ij}-k]_{q_i}![k]_{q_i}!}E_i^{k}E_jE_i^{1-a_{ij}-k}&=&0,\label{SerreE}\\
\sum_{k=0}^{1-a_{ij}}(-1)^k\frac{[1-a_{ij}]_{q_i}!}{[1-a_{ij}-k]_{q_i}![k]_{q_i}!}F_i^{k}F_jF_i^{1-a_{ij}-k}&=&0,\label{SerreF}
\end{eqnarray}
where $[k]_q:=\frac{q^k-q^{-k}}{q-q\inv}$ is the $q$-number and $[n]_q!:=\prod_{k=1}^n [k]_q$ the $q$-factorial.
\end{Def}

The algebra $\fD_\g$ is a Hopf algebra with comultiplication
\Eq{
\D(E_i)=&1\ox E_i+E_i\ox K_i,&\D(K_i)&=K_i\ox K_i,\\
\D(F_i)=&F_i\ox 1+K_i'\ox F_i,&\D(K_i')&=K_i'\ox K_i',
}
the counit
\Eq{
\e(E_i)&=\e(F_i)=0, & \e(K_i)&=\e(K_i')=1,\\
}
and antipode
\Eq{
S(E_i)&=-K_i\inv E_i,  &S(K_i)&=K_i\inv,\\
S(F_i)&=-F_iK_i, &S(K_i')&=(K_i')\inv.
}

\begin{Def}
The quantum group $\cU_q(\g)$ is defined as the quotient
\Eq{
\cU_g(\g):=\fD_\g/\<K_iK_i'=1|i\in I\>,
}
and it inherits a well-defined Hopf algebra structure from $\fD_\g$. 
\end{Def}
\begin{Rem} $\fD_\g$ is the Drinfeld's double of the quantum Borel subalgebra $\cU_q(\fb)$ generated by $E_i$ and $K_i$.
\end{Rem}
\begin{Def}We define the rescaled generators
\Eq{
\be_i :=\left(\frac{\sqrt{-1}}{q_i-q_i\inv}\right)\inv E_i,&&\bf_i :=\left(\frac{\sqrt{-1}}{q_i-q_i\inv}\right)\inv F_i.\label{rescaleFF}
}
By abuse of notation, we will also denote by $\fD_\g$ the $\C(q_s)$-algebra generated by $$\{\be_i, \bf_i, K_i, K_i'|i\in I\}$$ and the corresponding quotient by $\cU_q(\g)$. The generators satisfy all the defining relations above except \eqref{EFFE} which is modified to be 
\Eq{
[\be_i, \bf_j]=\d_{ij} (q_i-q_i\inv)(K_i'-K_i).
}
\end{Def}

%==============================================================================
\section{Quantum cluster $\cX$-tori}\label{sec:tori}
We recall the definition of the quantum cluster $\cX$-tori following \cite{FG1, SS2} and their properties that are needed, as well as some notations and modification that fit the needs of this paper.
\subsection{Quantum torus algebra and quivers}
\begin{Def}[Quantum torus algebra]
A seed $\bi$ is a triple $(I, I_0, B, D)$ where $I$ is a finite set, $I_0\subset I$ is a subset called the \emph{frozen subset}, $B=(b_{ij})_{i,j\in I}$ a skew-symmetrizable $\Q$-valued matrix called the \emph{exchange matrix}, and $D=diag(d_i)_{i\in I}$ is a diagonal matrix such that $DB=B^TD$ is skew-symmetric.

Let $q$ be a formal parameter. We define the \emph{quantum torus algebra} $\cX_{\bi}$ associated to the seed $\bi$ to be an associative algebra over $\C(q^d)$, where $d=\min_{i\in I}(d_i)$, defined by generators $X_i^{\pm1}, i\in I$ subject to the relations
\Eq{
X_iX_j=q_i^{-2b_{ij}}X_jX_i,\tab i,j\in I
}
where $q_i:=q^{d_i}$. The generators $X_i$ are called the \emph{quantum cluster variables}, and they are said to be \emph{frozen} if $i\in I_0$. We call $d_i$ the \emph{multipliers} of the variables $X_i$. We denote by $\bT_\bi$ the non-commutative field of fraction of $\cX_\bi$.
\end{Def}

The structure of the quantum torus algebra $\cX_{\bi}$ associated to a seed $\bi$ can be conveniently encoded in a quiver:
\begin{Def}[Quiver associated to $\bi$]\label{quiver} We associate to each seed $\bi$ a generalized quiver $Q^\bi=(Q_0,w)$ with vertices $Q_0$ labeled by $I$, and for each pair $i,j\in Q_0$ a weight 
\Eq{
w_{ij}:=d_ib_{ij}=-w_{ji}.
} We will draw arrows from $i\xto{w_{ij}} j$ if $w_{ij}>0$. 
We will call an isomorphism $\pi:S\simeq Q_0$ from a finite set $S$ an \emph{external label} of the quiver $Q$.
\end{Def}

We will use squares to denote frozen nodes $i\in I_0$ and circles otherwise. In the sequel, when $q_i=q_s$ or $q_l$ given by Definition \ref{qi}, we will distinguish the arrows by thick or thin arrows instead of writing the weights. We will also use dashed lines to denote arrows with half the weights, which only occurs between frozen nodes.

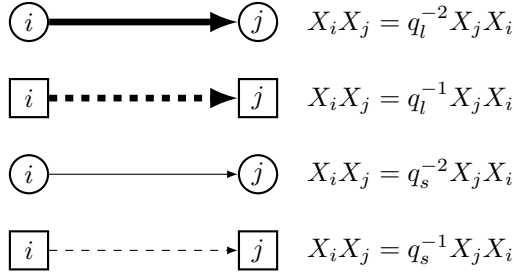
\begin{figure}[H]
\centering
  \begin{tikzpicture}[every node/.style={inner sep=0, minimum size=0.5cm, thick}, x=1cm, y=1cm]
\node[draw,circle] (i1) at(0,4) {$i$};
  \node[draw,circle] (j1) at (3,4) {$j$};
   \draw[vthick, ->](i1) to (j1);
\node at (5,4) {$X_iX_j = q_l^{-2} X_j X_i$}; 
\node[draw] (i2) at(0,3) {$i$};
  \node[draw] (j2) at (3,3) {$j$};
   \draw[vthick, dashed, ->](i2) to (j2);
\node at (5,3) {$X_iX_j = q_l^{-1} X_j X_i$};
\node[draw,circle] (i3) at(0,2) {$i$};
  \node[draw,circle] (j3) at (3,2) {$j$};
   \draw[thin, ->](i3) to (j3);
\node at (5,2) {$X_iX_j = q_s^{-2} X_j X_i$};
\node[draw] (i4) at(0,1) {$i$};
  \node[draw] (j4) at (3,1) {$j$};
   \draw[dashed, ->](i4) to (j4);
\node at (5,1) {$X_iX_j = q_s^{-1} X_j X_i$};
  \end{tikzpicture}
  \caption{Arrows between nodes and their algebraic meaning.}
  \end{figure}

We introduce the following notations which will be useful throughout the paper:
\begin{Def}\label{Xnotation} We denote by
\Eq{
X_{i_1^{m_1},...,i_n^{m_n}}:=q^C X_{i_1}^{m_1}... X_{i_n}^{m_n},\label{Xnotation1}
}
where $C$ is the unique rational number such that
$$q^C X_{i_1}^{m_1}... X_{i_n}^{m_n}=q^{-C} X_{i_n}^{m_n}... X_{i_1}^{m_1}.$$
Explicitly, if $X_iX_j=q^{c_{ij}}X_jX_i$, then
\Eq{
C=-\half \sum_{p>q} m_p m_q c_{i_p}c_{i_q}.
}
If we introduce a $*$-structure such that $q^*=q\inv$ and $X_i^*=X_i$ (and positive), then the expression $X_{i_1^{m_1},...,i_n^{m_n}}$ is also (positive) self-adjoint.

We also denote by
\Eq{
X(i_1,..., i_n):=\sum_{k=1}^n X_{i_1,..., i_k}.\label{Xnotation2}
}
\end{Def}
\begin{Def} A permutation of a seed $\s: \bi\to \bi'$ is a bijection $\s:I\to I'$ such that
\Eqn{
\s(I_0)&=I_0',\\
b'_{ij}&= b_{\s(i)\s(j)},\\
d'_i&=d_{\s(i)}.
}
It induces an isomorphism $\s^*:\bT_{\bi'}\to \bT_{\bi}$ by
\Eqn{
\s^*(\what{X}_{\s(i)}):=X_i,
}
where $\what{X}_{\s(i)}$ denotes the quantum cluster variables of $\bT_{\bi'}$.
\end{Def}
\subsection{Quantum cluster mutation}
Next we define the cluster mutations of a seed and its quiver, and the quantum cluster mutations for the algebra.
\begin{Def}[Cluster mutation] Given a pair of seeds $\bi=(I,I_0,B, D)$, $\bi'=(I',I_0',B',D')$ and an element $k\in I\setminus I_0$, a \emph{cluster mutation in direction $k$} is an isomorphism $\mu_k:\bi\to \bi'$ such that $\mu_k(I_0)=I_0'$, 
\Eq{
b'_{\mu_k(i),\mu_k(j)} &= \case{-b_{ij}&\mbox{if $i=k$ or $j=k$},\\ b_{ij}+\frac{b_{ik}|b_{kj}|+|b_{ik}|b_{kj}}{2}&\mbox{otherwise},}\\
d'_{\mu_k(i)}&=d_i.
}

Then the quiver mutation $Q^\bi\to Q^{\bi'}$ corresponding to the mutation $\mu_k$ can be performed by:
\begin{itemize}
\item[(1)] reverse all the arrows incident to the vertex $k$;
\item[(2)] for each pair of arrows $i\xto{w_{ik}} k$ and $k\xto{w_{kj}}j$ associate the arrow $i\xto{w_{ij}+\frac{w_{ik}w_{kj}}{d_k}}j$.
\item[(3)] delete any arrows with weight $w_{ij}=0$.
\end{itemize}
\end{Def}
\begin{Def}[Quantum cluster mutation]
The cluster mutation in direction $k$, $\mu_k:\bi\to \bi'$, induces an isomorphism $\mu_k^q:\bT_{\bi'}\to \bT_{\bi}$ called the \emph{quantum cluster mutation}, defined by
\Eq{
\mu_k^q(\what{X}_i)=\case{X_k\inv&\mbox{if $i=k$},\\ \dis X_i\prod_{r=1}^{|b_{ki}|}(1+q_i^{2r-1}X_k)&\mbox{if $i\neq k$ and $b_{ki}\leq 0$},\\\dis X_i\prod_{r=1}^{b_{ki}}(1+q_i^{2r-1}X_k\inv)\inv&\mbox{if $i\neq k$ and $b_{ki}\geq 0$},}
}
where we denote by $\what{X}_i$ the quantum cluster variables corresponding to $\cX_{\bi'}$ with exchange matrix $B'$ such that $b'_{ki}=-b_{ki}$ for every $i\in I$. 

The quantum cluster mutation $\mu_k^q$ can be written as a composition of two homomorphisms
\Eq{
\mu_k^q=\mu_k^\#\circ \mu_k',
}
where $\mu_k':\bT_{\bi'}\to \bT_\bi$ is a monomial transformation defined by
\Eq{
\mu_k'(\what{X}_i):=\case{X_k\inv&\mbox{if $i=k$},\\ X_i&\mbox{if $i\neq k$ and $b_{ki}\leq 0$},\\ q_i^{b_{ik}b_{ki}}X_iX_k^{b_{ik}}&\mbox{if $i\neq k$ and $b_{ki}\geq 0$},}
}
and $\mu_k^\#:\bT_\bi\to\bT_\bi$ is a conjugation by the \emph{quantum dilogarithm function}
\Eq{
\mu_k^\#:=Ad_{\Psi^{q_k}(X_k)}, 
}where
\Eq{
\Psi^{q}:=\prod_{r=0}^\oo (1+q^{2r+1}x)\inv.
}
\end{Def}

In the remaining of the paper, however, we will use the notation
\Eq{
g_{b_k}(x)&:=\Psi^{q_k}(x)\inv,\\
g_{b_k}^*(x)&:=g_{b_k}\inv(x)=\Psi^{q_k}(x)
}
instead, in accordance to the universal $R$-operator formula given in \cite{Ip4}. The various identities of $g_b(x)$ that are needed in this paper are summarized in Appendix \ref{sec:dilog}.

\begin{Rem}\label{qd}
We remark that $g_b(x)$, where $q=e^{\pi \sqrt{-1}b^2}$, is the notation for the non-compact quantum dilogarithm, which plays a central role in the theory of positive representation, various quantum Teichm\"uller theories \cite{FG1, Ka} and non-rational conformal field theories \cite{BT, PT1, PT2}. It is composed by two commuting copies, associated to the so-called \emph{Faddeev's modular double}, of the compact quantum dilogarithm $\Psi^q(x)$ \cite{Fa2, FKa}, and it is a unitary operator when $x$ is positive self-adjoint.

In this paper however, we are only interested in the formal algebraic theory, hence one may consider only the compact part and think of the correspondence
\Eq{
g_b(x)\sim \Psi^q(x)\inv=\prod_{r=0}^\oo (1+q_k^{2r+1}x) = Exp_{q^{-2}}\left(-\frac{u}{q-q\inv}\right),
}
where 
\Eq{
Exp_q(x)&:=\sum_{k\geq 0} \frac{x^k}{(k)_q!},\\
(k)_q&:=\frac{1-q^k}{1-q}.
}

The use of the notation $g_b(x)$ suggests that the theory of the current paper can be naturally applied to the case of the non-compact split real setting, where all the algebraic relations are satisfied, and naturally the positivity and self-adjointness of the operators are automatically taken care into account, which makes the choice extremely natural.
\end{Rem}

The following version of the useful Lemma from \cite[Lemma 1.1]{SS2} is rewritten in the notation of the current paper:
\begin{Lem}\label{useful} Let $\mu_{i_1}, ... ,\mu_{i_k}$ be a sequence of mutation, and denote the intermediate seeds by $\bi_j:=\mu_{i_j}...\mu_{i_1}(\bi)$. Then the induced quantum cluster mutation $\mu_{i_1}^q ...\mu_{i_k}^q: \bT_{\bi_k}\to \bT_{\bi}$ can be written as 
\Eq{
\mu_{i_1}^q ...\mu_{i_k}^q=\Phi_k\circ M_k,
}
where $M_k:\bT_{\bi_k}\to \bT_{\bi}$ and $\Phi_k:\bT_\bi\to\bT_\bi$ with
\Eq{
M_k&:= \mu_{i_1}'\mu_{i_2}'... \mu_{i_k}',\\
\Phi_k&:= Ad_{g_{b_{i_1}}^*(X_{i_1})}Ad_{g_{b_{i_2}}^*(\mu_{i_1}'(X_{i_2}^{\bi_1}))}... Ad_{g_{b_{i_k}}^*(\mu_{i_1}'...\mu_{i_{k-1}}'(X_{i_k}^{\bi_{k-1}}))},
}
and $X^{\bi}_i$ denotes the corresponding quantum cluster variables of the algebra $\cX_\bi$. 
\end{Lem}

\subsection{Amalgamation}
We also recall the procedure of \emph{amalgamation} of two quivers \cite{FG2}:
\begin{Def}
Let $Q_1:=Q^{\bi_1}$ and $Q_2:=Q^{\bi_2}$ be a pair of quivers associated to the seed $\bi_1=(I^1,I_0^1,B^1,D^1), \bi_2=(I^2,I_0^2,B^2,D^2)$ and with edge weights $w^1, w^2$ respectively, and let $J_1\subset I_0^1, J_2\subset I_0^2$ be certain subsets of the frozen nodes of $Q_1$ and $Q_2$ respectively. Assume there exists a bijection $\phi:J_1\to J_2$ such that $d_{\phi(i)}=d_i$ for $i\in J_1$. Then the \emph{amalgamation} of $Q_1$ and $Q_2$ along $\phi$ is a new quiver $Q$ constructed as follows:
\begin{itemize}
\item[(1)] The vertices of $Q$ are given by $Q_1\cup_\phi Q_2$ by identifying vertices $i\in Q_1$ and $\phi(i)\in Q_2$ and assigned with the same weight $d_i$,
\item[(2)] The frozen nodes of $Q$ are given by $(I_0^1\setminus J_1)\sqcup (I_0^2\setminus J_2)$, i.e. we ``defroze" the vertices that are glued.
\item[(3)] The weights $w$ of the edges of $Q$ are given by
$$w_{ij}=\case{0&\mbox{if $i\in I^k\setminus J_k$ and $j\in I^{2-k}\setminus J_{2-k}$ for $k=1,2$},\\
w^k_{ij}&\mbox{if $i\in I^k\setminus J_k$ or $j\in I^k\setminus J_k$ for $k=1,2$},\\
w^1_{ij}+w^2_{\phi(i)\phi(j)}&\mbox{if $i,j\in J_1$}.}$$
\end{itemize}
Amalgamation of a pair of quiver induces an embedding $\cX\to \cX_1\ox \cX_2$ of the corresponding quantum cluster $\cX$-tori by
\Eq{
X_i\mapsto \case{
X_i\ox 1&\mbox{if $i\in Q_1\setminus J_1$},\\
1\ox X_i&\mbox{if $i\in Q_2\setminus J_2$},\\
X_i\ox X_{\phi(i)}&\mbox{otherwise}.
}
}
\end{Def}
Visually it is just gluing two quivers together along the chosen frozen nodes, such that the weights of the corresponding arrows among those nodes are added.
%==============================================================================
\section{From positive representations to quantum group embedding}\label{sec:pos}
In this section, we introduce the notion of positive representations, and show that one can construct from it the quantum group embedding into certain quantum torus algebra.
\subsection{Positive representations $\cP_\l$ of $\cU_q(\g_\R)$}\label{sec:pos:pos}
In \cite{FI, Ip2, Ip3}, a special class of representations called the \emph{positive representations} is constructed for split real quantum groups $\cU_q(\g_\R)$ (and its modular double, which is not needed in this paper). Here $\cU_q(\g_\R)$ is defined to be $\cU_q(\g)$ endowed with the star structure
\Eq{
\be_i^*=\be_i, \tab \bf_i^*=\bf_i, \tab K_i^*=K_i,
}
and necessarily $|q_i|=1$ for every $i\in I$, whence we let $q_i=e^{\pi \sqrt{-1} b_i^2}\in\C$ for $b_i\in \R$. We assume the $q_i$'s are not root of unity for simplicity.

\begin{Thm} [Positive representations]
There exists a family of irreducible representations $\cP_{\l}$ of $\cU_{q}(\g_\R)$ parametrized by the $\R_+$-span of the cone of positive weights $\l\in P_\R^+\subset \fh_\R^*$, or equivalently by $\l:=(\l_1,...,\l_n)\in \R_+^n$ where $n=rank(\g)$, such that 
\begin{itemize}
\item For each reduced word $\bi\in \fR$, the generators $\be_i,\bf_i,K_i$ are represented by positive essentially self-adjoint operators acting on $L^2(\R^N)$,
\item Each generators can be represented by monomials generated by the positive operators
$$\{e^{\pm\pi b_i x_i}, e^{\pm2\pi b_i p_i}\}_{i=1,..., N},$$
where $p_i=\frac{1}{2\pi \sqrt{-1}}\del[,x_i]$ are the momentum operators such that $[p_i,x_i]=\frac{1}{2\pi \sqrt{-1}}$, and each monomials are positive essentially self-adjoint.
\item There exists a unitary equivalence $\Phi$ between positive representations corresponding to different reduced words, hence the representation does not depend on the choice of reduced expression of $w_0$.
\end{itemize}
\end{Thm}

In the theory of positive representations of split real quantum groups, the representation carries a real structure and the operators are represented by unbounded positive operators on certain Hilbert spaces. However in this paper, we will only be dealing with the representation formally, so all the generators and relations are treated on the algebraic level only. Hence if we define formally
\Eq{
X_i^{\pm1}=e^{\pm\pi b_i x_i},\tab Y_i^{\pm1} = e^{\pm2\pi b_i p_i},\label{XY}
}
then algebraically we have for $i=1,..., N$:
\Eq{
X_iY_i&=q_i Y_i X_i,\label{qtorus}\\
X_iY_j&=Y_jX_i,\tab i\neq j.\nonumber
}

As a corollary, if we just consider the quantum torus algebra $\C[\bT_q]$ generated by the elements $\<X_i^{\pm1}, Y_i^{\pm1}\>_{i=1,...,N}$ subjected to \eqref{qtorus}, we obtain 
\begin{Cor} The positive representations give an embedding of $\cU_q(\g)$ into $\C[\bT_q]$, generalizing the Feigin's homomorphism $\cU_q(\fb)\to \C[\bT_q]$.
\end{Cor}
\begin{Rem} In \cite{Ip2, Ip3}, we showed that one can shift the generators $\be_i, \bf_i$ by some appropriate $K_i$ factors such that the modified quantum group embeds into the ``true'' quantum torus algebra $\C\<X_i^{\pm1}, Y_i^{\pm1}\>$ with the relations $X_i Y_i=q_i^2 Y_i X_i$ instead.
\end{Rem}
%==============================================================================

\subsection{Explicit construction of $\cP_\l$}\label{sec:pos:con}
The positive representations $\cP_\l$ was computed explicitly for all types of $\g$. Let us first recall some notations used in \cite{Ip2, Ip3}. 

\begin{Def}\label{standardform}We denote by $p_u=\frac{1}{2\pi \sqrt{-1}}\del[,u]$ and 
\Eq{e(u)&:=e^{\pi bu},\tab [u]:=q^\half e(u)+q^{-\half}e(-u),}
so that whenever $[p,u]=\frac{1}{2\pi \sqrt{-1}}$,
\Eqn{[u]e(-2p)&:=(q^\half e^{\pi bu}+q^{-\half}e^{-\pi bu})e^{-2\pi bp} \\
&= e^{\pi bu-2\pi b p}+e^{-2\pi b u-2\pi b p}\\
&= e(u-2p)+e(-u-2p)
}
is self-adjoint.
\end{Def}

\begin{Def}\label{notation}By abuse of notation, we denote by
\Eq{[u_s+u_l]e(-2p_s-2p_l):=e^{\pi b_s(-u_s-2p_s)+\pi b_l(-u_l-2p_l)}+e^{\pi b_s(u_s-2p_s)+\pi b_l(u_l-2p_l)},}
where $u_s$ (resp. $u_l$) is a linear combination of the variables corresponding to short roots (resp. long roots). The parameters $\l_i$ are also considered in both cases. Similarly $p_s$ (resp. $p_l$) are linear combinations of the momentum variables corresponding to the short roots (resp. long roots). This applies to all simple $\g$, with the convention given in Definition \ref{qi}.
\end{Def}

\begin{Def}[Notation]\label{variable} Let $\bi=(i_1,...,i_N)\in\fR$ be a reduced word for $w_0$. We associate to $\bi$ a set of $N$ variables indexed in two ways:
\begin{itemize}\item $u_{i}^k$ denotes the $k$-th variables from the left\footnote{This differs from the previous notation used in \cite{Ip2, Ip3} where the variables read from the right. This version will be more convenient in this paper.} in $\bi$ corresponding to the root index $i$.
\item $v_j$ denotes the $j$-th variable from the left in $\bi$, i.e. corresponding to $i_j$, and $i_j$ is the root index corresponding to $v_j$.
\item We denote the corresponding momentum operators as $p_i^k$ and $p_j$ respectively if no confusion arises.
\item $v(i,k)$ denotes the index such that $u_i^k= v_{v(i,k)}$.
\end{itemize}
\end{Def}
\begin{Ex} For type $A_3$, let $\bi=(1,2,1,3,2,1)$. Then the 6 variables are ordered as:
$$(u_1^1, u_2^1, u_1^2, u_3^1, u_2^2, u_1^3) = (v_1,v_2,v_3,v_4,v_5,v_6).$$
\end{Ex}

Now we can summarize the construction of the positive representations as follows:
\begin{Thm}\label{pos2} Given a reduced word $\bi\in \fR$, the positive representation $\cP_\l\simeq L^2(\R^N)$ of $\cU_q(\g_\R)$ is parametrized by $\l=(\l_i)\in \R_{\geq 0}^{n}$, and the generators are represented in the form
\Eq{
\bf_i &= F_i^1+F_i^2+...+F_i^{n_i},\\
K_i &= e(-2\l_i-\sum_{j=1}^N a_{i_j,i}v_j),
}
where 
\Eq{
F_i^k &= \left[-\sum_{j=1}^{v(i,k)}a_{i_j,i}v_j+u_i^k -2\l_1\right]e(2p_i^k)\label{Fik}\\
&= e\left(-\sum_{j=1}^{v(i,k)}a_{i_j,i}v_j+u_i^k -2\l_1+2p_i^k\right)+e\left(\sum_{j=1}^{v(i,k)}a_{i_j,i}v_j-u_i^k -2\l_1+2p_i^k\right)\nonumber\\
&=:F_i^{k,-}+F_i^{k,+}
}
splitting according to Definition \ref{notation}. 

The representation of $\be_j$ is explicitly written case by case. In general, if $j=i_N$, then
\Eq{
\be_{j} = [v_N]e(-p_N)\label{be},
}
Otherwise 
\Eq{
\be_j=\Phi\circ [v_N]e(-2p_N)\circ \Phi\inv\label{unitrans},
}
where we recall $\Phi$ is the unitary transformation, expressed in terms of quantum dilogarithms, that relates $\cP_\l$ to another representations corresponding to a reduced word $\bi'\in\fR$ with $i_N'=j$.
\end{Thm}
\begin{Thm}\label{initial} Each generator $\be_j$ is expressed as a polynomial in $X_i^{\pm1}$ and $Y_i^{\pm1}$ (cf. \eqref{XY}), with a unique \emph{initial term} of the form $[u_i^k]e(-2p_i^k +... )$ for some index $i,k$. One determines this initial term by applying the transformation \eqref{unitrans} and tracing the changes of the corresponding initial term by the following rules:
\begin{itemize}
\item if $j=i_N$, then from \eqref{be} the initial term is $[v_{i_N}]$ by definition.
\item If we have a change of word $(...i,j,i,...)\corr (...,j,i,j,...)$, inducing a change of variables
$$(..., u_i^{k}, u_j^{l}, u_i^{k+1},...) \corr (..., u_j^{l}, u_i^{k}, u_j^{l+1},... ),$$ 
then the initial term changes from $[u_i^k]\corr [u_j^{l+1}]$.
\item If we have a change of word $(...,i,j,i,j,...)\corr(...,j,i,j,i,...)$, inducing a change of variables
$$(..., u_i^{k}, u_j^{l}, u_i^{k+1}, u_j^{l+1}...) \corr (..., u_j^{l}, u_i^{k}, u_j^{l+1}, u_i^{k+1}... ),$$ 
then the initial term changes from $[u_j^l]\corr [u_j^{l+1}]$.
\item In type $G_2$, for the change of word $(2,1,2,1,2,1)\corr (1,2,1,2,1,2)$, the initial term for $\be_1$ is $[u_1^3]\corr [u_1^1]$ and initial term for $\be_2$ is $[u_2^1]\corr [u_2^3]$.
\end{itemize}
\end{Thm}

From the explicit expression \eqref{Fik}, we have
\begin{Prop}\label{Fi}
If we write $\bf_i$ as 
$$\bf_i=F_i^{n_i,-}+F_i^{n_i-1,-}... F_i^{1,-}+ F_i^{1,+}+ F_i^{2,+}+...+F_i^{n_i,+},$$
then each term $q_i^{-2}$-commute with all the terms on the right, and each term $q_i^{-2}$-commute with $K_i\inv$.
\end{Prop}
\begin{Rem} Feigin's homomorphism $U_q(\fb_-)\to \C[\bT_q]$ is given by the expression of $K_i$ and half of $\bf_i$:$$\bf_i':= F_i^{1,+}+ F_i^{2,+}+...+F_i^{n_i,+}$$ only, so the expression of Theorem \ref{pos2} is really a ``double" of Feigin's homomorphism.
\end{Rem}
%==============================================================================

\subsection{Embedding of $\cU_q(\g)$ into quantum torus algebra $\cD_\g$}\label{sec:pos:embed}

Now we are ready to construct the quantum torus algebra $\cD_\g$ in the favor of \cite{SS2} that will provide a clear description of the embedding of the generators of the quantum group $\cU_q(\g)$.

\begin{Def}\label{clusterX} Define $2N+2n$ variables indexed by 
$$S=\{f_i^{-n_i},..., f_i^{n_i}\}_{i\in I}\cup\{e_i^0\}_{i\in I}\simeq \{1,..., 2N+2n\}$$ as follows: For each $i\in I$, we take the consecutive ``ratio'' of the monomial terms of $\bf_i$ as:
\Eq{\label{XfromF}
X_{f_i^k}=\case{ F_i^{n_i,-}& k=-n_i,\\ 
q_iF_i^{k,-} (F_i^{k-1,-})\inv &n<0,\\ 
q_i F_i^{1,+}(F_i^{1,-})\inv&k=0,\\
q_i  F_i^{k+1,+}(F_i^{k,+})\inv&k>0,\\
q_i K_i\inv (F_i^{n_i,+})\inv& k=n_i.
}
}
Let the initial term of $\be_i$ be 
\Eqn{
[v_n]e(-2p_n)&=e(v_n-2p_n)+e(-v_n-2p_n)\\
&=:E_i^-+E_i^+
}
as in Theorem \ref{initial}. Then we define
\Eq{
X_{e_i^0}=q_i E_i^+(E_i^-)\inv  = e(-2v_n).
}
\end{Def}

We note that each $X_k$ are self-adjoint. Moreover, since all $X_k$ are expressed formally as a monomial, we have
\Eq{
\label{XjXk}X_jX_k=q_j^{-2b_{jk}}X_kX_j
}
for some skew-symmetrizable exchange matrix $B=(b_{jk})$ and $q_j:=q_i$ if $j=f_i^k$ or $e_i^0$. By abuse of notation, we will use the same variables for the definition below:
\begin{Def} We define the quantum torus algebra $\cD_\g$ to be the algebra generated by the elements
$$X_{f_i^{-n_i}},..., X_{f_i^{n_i}}, X_{e_i^0},\tab i=1,..., n$$
subject to the relations \eqref{XjXk}. 

The corresponding $\cD_\g$-quiver is associated to the seed $(S, S_0, B, D)$ where $D=diag(d_i)$ and the frozen nodes are $S_0=\{f_i^{-n_i}\}_{i\in I}\cup\{f_i^{n_i}\}_{i\in I}$.
\end{Def}

Now we can state our first main result of the paper.
\begin{Thm}\label{mainThm} We have an embedding of algebra
\Eq{
\iota:\fD_\g\inj \cD_\g,}
which induces an embedding of the quantum group into a quotient of $\cD_\g$
\Eq{
\cU_q(\g)\inj \cD_\g/\<\iota(K_i)\iota(K_i)'=1\>.
}
\end{Thm}
\begin{proof} By construction from \eqref{XfromF}, we can write (cf. Definition \ref{Xnotation})
\Eqn{
\bf_i&=X_{f_i^{-n_i}}+q_i X_{f_i^{-n_i}}X_{f_i^{-n_i+1}}+... + q_i^{n_i-1}X_{f_i^{-n_i}}...X_{f_i^{n_i-1}}\\
&=X(f_i^{-n_i}, f_i^{-n_i+1},..., f_i^{n_i-1})\\
K_i' &= X_{f_i^{-n_i},..., f_i^{n_i}}.
}
Given a reduced word $\bi=(i_1,...,i_N)\in\fR$, if $i=i_N$, then one computes explicitly
\Eqn{
\be_i &= [u_i^1]e(-2p_i^1) \\
&= e^{\pi b_i u_i^1-2\pi b_i p_i^1}+e^{\pi b_i u_i^1-2\pi b_i p_i^1}\\
&=X_{f_i^{n_i}}+q_iX_{f_i^{n_i}}X_{e_i^0}\\
&=X(f_i^{n_i}, e_i^0),\\
K_i &=X_{f_i^{n_i}, e_i^0, f_i^{-n_i}}.
}
Otherwise, from the construction of positive representation, each mutation of the reduced expression of $w_0$ correspond to a unitary transformation $\Phi$ given by the quantum dilogarithm function with an argument given by a consecutive difference of the $F_i^n$'s in the corresponding mutated quiver. (This is described in detail in Section \ref{sec:mutation}.) Hence $\be_i$ will be expressed as a sum of monomials, each of which is expressed as a product of $X_{f_i^n}$ and the ratios between the initial term, which is given by $X_{e_i^0}$. The explicit expression is given in the next section. 

Furthermore, the unitary transformation $\Phi$ has the properties that for any reduced word $\bi\in \fR$, if $\be_j$ is expressed as 
$$\be_i=X(i_1,..., i_k)=X_{i_1}+... + X_{i_1,..., i_k},$$
then the leading term $X_{i_1}=X_{f_i^{n_i}}$ and the ending term satisfies 
$$X_{i_1,...,i_k}X_{f_i^{-n_i}}=q_i^{-2}X_{f_i^{-n_i}}X_{i_1,...,i_k}.$$

The unitary transformation $\Phi$, while inducing a change of variables given by a linear transformation, will keep $K_i$ as a monomial. Hence from the relation $$[\be_i, \bf_i]=(q_i-q_i\inv)(K_i'-K_i),$$we see that the term $ X_{i_1,..., i_k, f_i^{-n_i}}$ does not vanish. Since we already have $K_i'=X_{f_i^{-n_i},..., f_i^{n_i}}$, we must have
$$K_i = X_{i_1,..., i_k, f_i^{-n_i}},$$
hence giving the desired homomorphism of $\fD_\g$ into $\cD_\g$.

Since the positive representation is a faithful irreducible representation coming from the quantization of the induced representation of the left regular representation, the homomorphism $\iota$ is an embedding. More explicitly, by choosing the parameters $\l$ such that
\Eq{\label{discreteL}
2\sqrt{-1}\l_i\in Q_i+b_i\N, \tab Q_i:=b_i+b_i\inv,
}
we recover every finite dimensional highest weight irreducible representation for the \emph{compact} quantum group $\cU_q(\g)$. This fact has been utilized for example to calculate the eigenvalues of the positive Casimir operators \cite{Ip6}. In particular all the PBW basis cannot be identically zero in the representation, hence the homomorphism $\iota$ is indeed an embedding.

Finally, we note that the algebra $\cD_\g$ has a center generated by $\iota(K_i)\iota(K_i')$. Hence taking the quotient with $\iota(K_i)\iota(K_i')=1, i\in I$ we obtain the desired embedding of $\cU_q(\g)$ as well.
\end{proof}
\begin{Rem}\label{iw} Note that by the Cartan involution, one can also define another embedding
\Eqn{
\iota^w: \fD_\g&\to \cD_\g,\\
\be_i &\mapsto \iota(\bf_i),\\
\bf_i &\mapsto \iota(\be_i),\\
K_i &\mapsto \iota(K_i'),\\
K_i' &\mapsto \iota(K_i).
}
This interchanges the expressions of the explicit embeddings of $\be_i$ and $\bf_i$ in the quantum torus algebra $\cD_\g$.
\end{Rem}

\begin{Rem}In \cite{SS2}, the proof of the injectivity of $\iota$ is explicitly checked on the PBW basis. The expression relating the PBW exponents to those of the $q$-tori generators turns out to involve some combinatorial \emph{hive-type} conditions from the work of Knutson-Tao \cite{KTao}. It will be interesting to see explicitly analogues of such combinatorics in other types.
\end{Rem}
%==============================================================================
\section{Construction of the $\cD_\g$-quiver}\label{sec:quiver}
Let us now describe the explicit construction of the quiver associated to the $\cD_\g$ algebra in more details. 
%==============================================================================
\subsection{Relation among cluster variables}\label{sec:quiver:rel}
First we have the obvious relations.

\begin{Lem} $X_{f_i^0}$ and $X_{e_i^0}$ mutually commute with each other.
\end{Lem}
\begin{proof} By Definition \ref{clusterX}, the formal expression of all the $X_{f_i^0}$'s and $X_{e_i^0}$'s do not contain any momentum operators $e(2p)$, hence they commute with each other.
\end{proof}

Next we have the following observation:
\begin{Lem}\label{FikLem} Recall that $F_i^{k,\pm}$ is defined in \eqref{Fik}. We have
\Eq{
F_i^{k,\pm} F_j^{l,\pm} = q_i^{\mp a_{ij}}F_j^{l,\pm}F_i^{k,\pm}
}
if $v(i,k) < v(j,l)$.
\end{Lem}
\begin{proof} Let us consider the $+$ case, while the $-$ case is similar. By definition, if $v(i,k)<v(j,l)$, then there are no terms of $u_j^l$ appearing in $F_i^{k,+}$, hence $e(2p_j^l)$ in $F_j^{l,+}$ commutes with everything in $F_i^{k,+}$, while $e(2p_i^k)$ from $F_i^{k,+}$ $q$-commutes with $e(...a_{ij}u_i^k+...)$ from $F_j^{l,+}$ giving the factor $q_i^{-a_{ij}}$.
\end{proof}

Thus one can derive the commutation relation directly between the variables $X_{f_i^k}$ and $X_{f_j^l}$. First, by definition we have
\Eq{
X_{f_i^k}X_{f_i^l}=q_i^{-2} X_{f_i^l}X_{f_i^k} 
}
whenever $l=k+1$ and commute otherwise.
\begin{Cor}\label{FikCor} Assume $i\neq j, k,l\geq0$ and $v(i,k)<v(j,l)$, we have:
\Eq{
X_{f_i^k}X_{f_j^l} = q_i^C X_{f_j^l}X_{f_i^k},\label{FikCorCom}
}
where 
$$C=\case{2a_{ij}& v(i,k)<v(j,l)<v(i,k+1)<v(j,l+1),\\ 
0 & v(i,k)<v(j,l)<v(j,l+1)<v(i,k+1),\\
0& v(i,k)<v(i,k+1)<v(j,l)<v(j,l+1),\\
a_{ij}& \mbox{$k=n_i$ and $l=n_j$},}$$
where in the inequalities we let the boundaries be $v(i,0)=-\oo$ and $v(i,n_i):=+\oo$.
\end{Cor}

Next, we observe that by construction, the cluster variables $X_{f_i^k}$ and $X_{f_j^l}$ with $k,l\leq0$ have exactly the commutation relations opposite to \eqref{FikCorCom}, while if $k,l\neq 0$ have different signs they commute with each other.

Finally, the cluster variables $X_{e_i^0}$ is given by $e(-2u_j^k)$ where $[u_j^k]$ is the initial term of the positive representation of $\be_i$ which is determined explicitly. Then we have
\Eqn{
X_{e_i^0}X_{f_j^k}&=q_i^2X_{f_j^k}X_{e_i^0},\\
X_{e_i^0}X_{f_j^{k-1}}&=q_i^{-2}X_{f_j^{k-1}}X_{e_i^0},\tab k\neq 1.
}
Combining the above relations among $X_k$, this completes the description of the $\cD_\g$-quiver.
%==============================================================================

\subsection{Elementary quiver associated to simple reflections}\label{sec:quiver:elementary}
With the above observations, a more conceptual way of constructing the $\cD_\g$ quiver motivated by \cite{Le2} is as follows. We define the following quiver:

\begin{Def}The \emph{elementary quiver} $Q_i^k$ associated to the $v(i,k)$-th simple reflection $s_i$ in $w_0$, i.e. corresponding to the variable $u_i^k$, is constructed by the frozen nodes
\begin{center}
  \begin{tikzpicture}[every node/.style={inner sep=0, minimum size=0.6cm, thick}, x=0.5cm, y=0.5cm]
\node[draw] (i) at(0,0) {$f_i^{k-1}$};
  \node[draw] (j) at (6,0) {$f_i^k$};
   \draw[thick, ->](i) to (j);
   \node at (3,0.3) {$d_i$};
\end{tikzpicture}
\end{center}
and for every $j$ connected to $i$ in the Dynkin diagram we have in addition
\begin{center}
  \begin{tikzpicture}[every node/.style={inner sep=0, minimum size=0.6cm, thick}, x=0.6cm, y=0.6cm]
\node[draw] (i) at(0,0) {$f_i^{k-1}$};
  \node[draw] (j) at (6,0) {$f_i^k$};
    \node[draw] (k) at (3,3) {$f_j^l$};
   \draw[thick, ->](j) to (k);
      \draw[thick, ->](k) to (i);
      \node at (5.5,1.5) {$\frac{|d_ia_{ij}|}{2}$};
            \node at (0.5,1.5) {$\frac{|d_ia_{ij}|}{2}$};
\end{tikzpicture}
\end{center}
for the unique $l$ with $v(j,l)<v(i,k)<v(j,l+1)$, and the nodes $f_i^a$ have weight $d_i$.
\end{Def}

In particular, according to the convention from Definition \ref{quiver}, in the last figure if either $i$ or $j$ is long, 
$X_{f_i^k}X_{f_j^l}=q\inv X_{f_j^l}X_{f_i^k}$ and we will denote both arrows by (thick) dashed arrows, while if both $i$ and $j$ are short, $X_{f_i^k}X_{f_j^l}=q_s\inv X_{f_j^l}X_{f_i^k}$ and we will denote both arrows by thin dashed arrows.

For example, in type $A_3$, an elementary quiver associated to $s_2$ is drawn as in Figure \ref{elemA3}, where we indicate the corresponding simple reflection, and indicate with dashed arrows the weights $\frac{|d_2a_{2k}|}{2}=\frac12$.
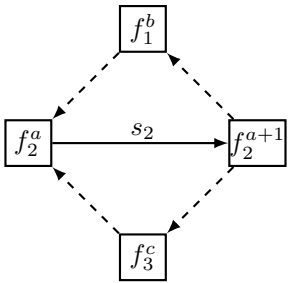
\begin{figure}[H]
\centering
\begin{tikzpicture}[every node/.style={inner sep=0, minimum size=0.6cm, thick}, x=0.5cm, y=0.5cm]
\node (1) at (0,0) [draw] {$f_2^a$};
\node (2) at (6,0) [draw]{$f_2^{a+1}$};
\node (3) at (3, 3) [draw]{$f_1^b$};
\node (4) at (3,-3) [draw]{$f_3^c$};
\path (1) edge[->, thick] (2);
\node at (3,0.3) {$s_2$};
\path (3) edge[->, thick, dashed] (1);
\path (4) edge[->, thick, dashed] (1);
\path (2) edge[->, thick, dashed] (3);
\path (2) edge[->, thick, dashed] (4);
\end{tikzpicture}
\caption{Elementary quiver in type $A_3$.}\label{elemA3}
\end{figure}

\begin{Prop} The subquiver of the $\cD_\g$-quiver generated by $f_i^n$ with $n\geq 0$ corresponding to the reduced word $(i_1,...,i_N)\in\fR$ is given by amalgamation of the elementary quivers $Q_{s_i}^k$ along vertices with the same indices, modulo the arrows between $f_i^0$'s.

The subquiver of the $\cD_\g$-quiver generated by $f_i^n$ for every $n$ corresponding to the reduced word $(i_1,...,i_N)\in\fR$ is obtained by amalgamation of the elementary quivers $Q_{s_i}^{\pm k}$ along vertices with the same indices, where $Q_{s_i}^{-k}$ are the elementary quivers corresponding to the opposite words $(i_N,...,i_1)$ of the opposite nodes $f_i^{-n}$.
\end{Prop}
\begin{proof}
This follows directly from Corollary \ref{FikCor}.
\end{proof}

The $\cD_\g$ quiver is then obtained by the above amalgamation together with the quivers connecting the nodes $e_i^0$.
%==============================================================================

\section{Explicit embedding of $\cU_q(\g)$}\label{sec:embedding}
In the previous section, we constructed the $\cD_\g$-quivers as amalgamation of the elementary quivers $Q_i^k$ together with the arrows joining the nodes $e_i^0$. In particular, they can be presented in a way that is symmetric along a vertical axis, where the arrows are flipped over. It turns out that this can be expressed as an amalgamation of a pair of \emph{basic quivers} associated to $\g$, and that these basic quivers are mutation equivalent to the cluster structure of framed $G$-local system associated to the disk with 3 marked points, recently discovered by \cite{Le1, Le2}. We will determine and describe the basic quivers in Section \ref{sec:basicquiver}.

By Theorem \ref{mainThm}, the action of $K_i$ (resp. $K_i'$) are obtained by multiplying $X_{f_i^{-n_i}}$ (resp. $X_{f_i^{n_i}})$ to the last term of $\be_i$ (resp. $\bf_i$), hence we will omit it from the description below for simplicity.

\begin{Def}[$E_i$ and $F_i$-path]
Since $\bf_i=X(f_i^{-n_i},..., f_i^{n_i-1})$, we will call the path of the quiver given by the nodes 
$$f_i^{-n_i}\to f_i^{-n_i+1}\to ...\to  f_i^{n_i-1}\to  f_i^{n_i}$$
an \emph{$F_i$-path}.

On the other hand, if $\be_i=X(m_1,m_2,..., m_k)$ (or similar variants in type $C_n$, $E_8$ and $F_4$), we call the path of the quiver given by the nodes
$$m_1\to m_2,...\to  m_k\to  f_i^{-n_i}$$ an \emph{$E_i$-path}. From the relations $[\be_i,\bf_i]=(q_i-q_i\inv)(K_i'-K_i)$, one can derive the fact that the path always begin with $m_1 = f_i^{n_i}$.
\end{Def}
\begin{Rem}In \cite{SS2}, the $E_i$-paths and $F_i$-paths are also known as the $V_i$-paths and $\Lambda_i$-paths respectively in type $A_n$, which describe the corresponding shapes of the paths, see Figure \ref{fig-An}.
\end{Rem}
%==============================================================================
\subsection{Toy example: Type $A_2$}\label{sec:embedding:A2}
In this section, we illustrate the method of recovering the $\cD_\g$-quiver. Let $\g$ be of type $A_2$, and recall the notation from Definition \ref{standardform}. For simplicity we label the variables $(u_1^1,u_2^1,u_1^2)$ below by $(u,v,w)$.

\begin{Prop} \cite{Ip2}\label{typeA2} The positive representation $\cP_\l$ of $\cU_{q\til{q}}(\sl(3,\R))$ with parameters $\l=(\l_1,\l_2)\in\R_+^2$, corresponding to the reduced word $\bi=(1,2,1)$ acting on $f(u,v,w)\in L^2(\R^3)$, is given by
\Eqn{
\be_1=&[w]e(-2p_w),\\
\be_2=&[u]e(-2p_u-2p_v+2p_w)+[v-w]e(-2p_v),\\
\bf_1=&[-u-2\l_2]e(2p_u)+[-2u+v-w-2\l_2]e(2p_w),\\
\bf_2=&[u-v-2\l_1]e(2p_v),\\
K_1=&e(-2u+v-2w-2\l_1),\\
K_2=&e(u-2v+w-2\l_2).
}
\end{Prop}
In the expanded form, we have
\Eqn{
\be_1=&e(w-2p_w)+e(-w-2p_2),\\
\be_2=&e(u-2p_u-2p_v+2p_w)+e(v-w-2p_v)+e(-v+w-2p_v)+e(-u-2p_u-2p_v+2p_w),\\
\bf_1=&e(-2u+v-w-2\l_2+2p_w)+e(-u-2\l_2+2p_u)+e(u+2\l_2+2p_u)+e(2u-v+w+2\l_2+2p_w),\\
\bf_2=&e(u-v-2\l_1+2p_v)+e(-u+v+2\l_1+2p_v).
}

We recover the following cluster variables following Definition \ref{clusterX} by taking successive ratios of the $\bf_i$ generators:
\Eqn{
X_{f_1^{-2}}&=e(-2u+v-w-2\l_2+2p_w),\\
X_{f_1^{-1}}&=e(u-v+w+2p_u-2p_w),\\
X_{f_1^0}&=e(2u+4\l_2),\\
X_{f_1^1}&=e(u-v+w-2p_u+2p_w),\\
X_{f_1^2}&=e(w-2p_w),\\
X_{f_2^{-1}}&=e(u-v-2\l_1+2p_v),\\
X_{f_2^0}&=e(-2u+2v+4\l_1),\\
X_{f_2^1}&=e(v-w-2p_v),
}
and taking the ratio of the initial terms of the $\be_i$ generators (i.e. the first and last terms in the expanded form) we have
\Eqn{
X_{e_1^0}&=e(-2w),\\
X_{e_2^0}&=e(-2u).
}
Hence treating the operators as formal algebraic variables, from their commutation relations we recover the quiver describing the quantum torus algebra $D_{\sl_3}$ in Figure \ref{A2quiver}.

\begin{figure}[!htb]
\centering
\begin{tikzpicture}[every node/.style={inner sep=0, minimum size=0.5cm, thick}, x=1.732cm, y=1cm]

\node (1) at (-2,4)[draw]{\tiny $f_1^{-2}$};
\node (2) at (-1,3)[draw,circle]{\tiny $f_1^{-1}$};
\node (3) at (0,2)[draw,circle]{\tiny $f_1^0$};
\node (4) at (1,3)[draw,circle]{\tiny $f_1^1$};
\node (5) at (2,4)[draw]{\tiny $f_1^2$};
\node (6) at (-2,2)[draw]{\tiny $f_2^{-1}$};
\node (7) at (0,0)[draw,circle]{\tiny $f_2^0$};
\node (8) at (2,2)[draw]{\tiny $f_2^1$};
\node (9) at (0,6)[draw,circle]{\tiny $e_1^0$};
\node (10) at (0,4)[draw,circle]{\tiny $e_2^0$};
\drawpath{1,...,5,9,1}{thick}
\drawpath{7,8,4,10,2,6,7,2,9,4,7}{thick}
\drawpath{5,9,1}{thick,red}
\drawpath{8,4,10,2,6}{thick,red}
 \path(6) edge[->, dashed, thick] (1);
 \path(5) edge[->, dashed, thick] (8);
\end{tikzpicture}
\caption{$A_2$-quiver, with the $E_i$-paths colored in red.}\label{A2quiver}
\end{figure}
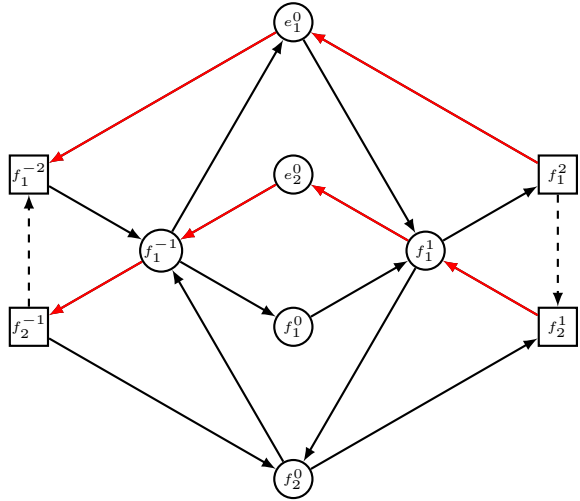

Using the notation from Definition \ref{Xnotation}, we see that the $F_i$-path is expressed as $V$-shaped paths in the quiver diagram.
\Eqn{
\bf_1&=X_{f_1^{-2}}+X_{f_1^{-2},f_1^{-1}}+X_{f_1^{-2},f_1^{-1},f_1^0}+X_{f_1^{-2},f_1^{-1},f_1^0,f_1^2}\\
&=X(f_1^{-2}, f_1^{-1}, f_1^0, f_1^1),\\
\bf_2&=X_{f_2^{-1}}+X_{f_2^{-1},f_2^0}\\
&=X(f_2^{-1}, f_2^0),\\
K_1'&=X_{f_1^{-2}, f_1^{-1}, f_1^0, f_1^1, f_1^2},\\
K_2'&=X_{f_2^{-1},f_2^0, f_2^1}.
}

Since the exponents of the variables $\{X_{f_i^{k}}\}_{k\neq 0}$ and $X_{e_i^0}$ for $i=1,2$ forms a complete basis of the linear space spanned by $\<u,v,w,p_u, p_v, p_w,\l_1,\l_2\>$, one can solve for the $\be_i$ action in terms of these cluster variables. As a result, we obtain
\Eqn{
\be_1&=X_{f_1^2}+X_{f_1^2, e_1^0}\\
&=X(f_1^2, e_1^0),\\
\be_2&=X_{f_2^1}+X_{f_2^1,f_1^1}+X_{f_2^1,f_1^1,e_2^0}+X_{f_2^1,f_1^1,e_2^0,f_1^{-1}}\\
&=X(f_2^1,f_1^1,e_2^0,f_1^{-1}),\\
K_1&=X_{f_1^2,e_1^0,f_1^{-2}},\\
K_2&=X_{f_2^1,f_1^1,e_2^0,f_1^{-2},f_2^{-1}},
}
which gives the $E_i$-path (highlighted in red) as $\L$-shaped paths in the quiver diagram as desired.

Let us now turn to the general cases.
%==============================================================================
\subsection{Type $A_n$}\label{sec:embedding:An}
The quiver associated to type $A_n$ and the quantum group embedding $\cU_q(\sl_{n+1})$ is fully described in \cite{SS2}. 
Let us choose the reduced word $$\bi = (1\; 21\; 321\;4321... n\; (n-1)\;...1).$$ Then $n_i=n+1-i$. Using the explicit expression of the positive representations from \cite{Ip2} in type $A_n$, we have
\Eqn{
\bf_i&=X(f_i^{-n+i-1},..., f_i^{n-i}),\\
\be_i&=X(f_i^{n-i+1}, f_{i-1}^{n-i+1},..., f_1^{n-i+1}, e_i^0, f_1^{-n+i-1},..., f_i^{-n+i-1}).
}
Here the initial terms are given by
$$X_{e_i^0} = e(-2u_1^{n+1-i}).$$

The quiver is shown in Figure \ref{fig-An}. We see that the $F_i$-path follows a $\L$-shaped path, while $E_i$-path follows a $V$-shaped path in the quiver (highlighted in red). The quiver can obviously be generalized to arbitrary rank.

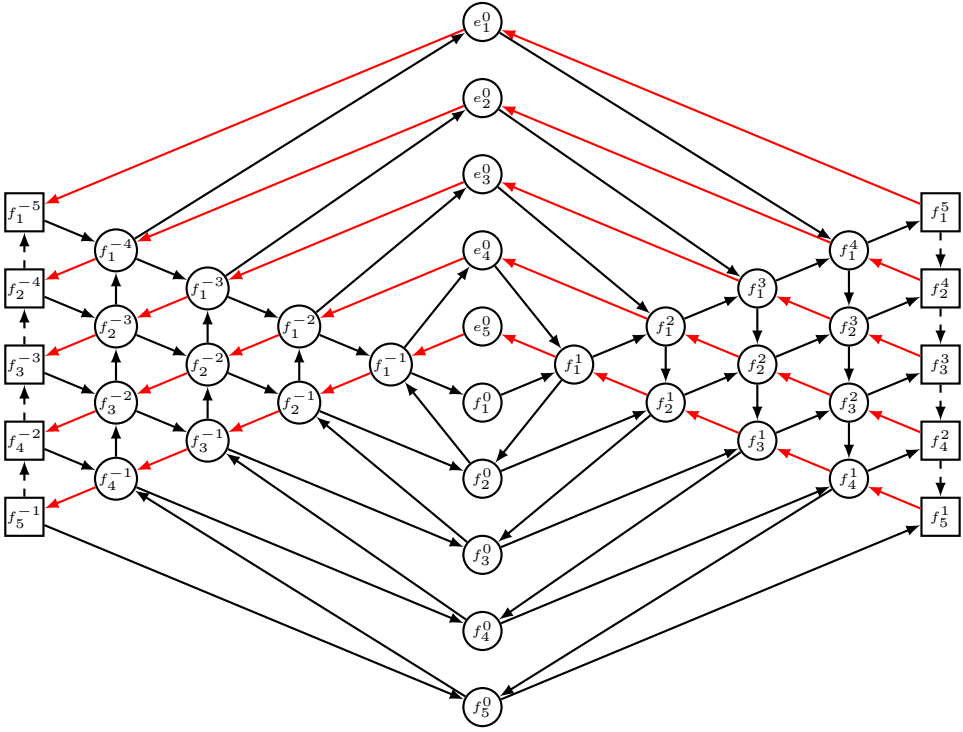
\begin{figure}[!htb]
\centering
\begin{tikzpicture}[every node/.style={inner sep=0, minimum size=0.5cm, thick}, x=1.2cm, y=0.5cm]
\foreach \y [count=\c from 1] in{33,28,21,12,1}
\foreach \x in{1,...,\c}{
	\pgfmathtruncatemacro{\ind}{\x+\y-1}
	\pgfmathtruncatemacro{\indop}{\y+2*\c-\x+1}
	\pgfmathtruncatemacro{\r}{6-\c}
	\pgfmathtruncatemacro{\ex}{\x-1-\c}
	\pgfmathtruncatemacro{\exop}{\c+1-\x}
	\ifthenelse{\x=1}{
		\node(\ind) at (\x-1,2*\c-1-\x)[draw] {\tiny$f_\r^{\ex}$};
		\node(\indop) at (11-\x,2*\c-1-\x)[draw] {\tiny$f_\r^{\exop}$};
	}{
		\node(\ind) at (\x-1,2*\c-1-\x)[draw, circle]{\tiny$f_\r^{\ex}$};
		\node(\indop) at (11-\x,2*\c-1-\x)[draw, circle]{\tiny$f_\r^{\exop}$};
	}
}
\foreach \x[count=\c from 1] in{6,16,24,30,34}{
\node(\x) at (5, 5-\c*2)[draw, circle]{\tiny$f_\c^0$};
}
\foreach \x[count=\c from 1] in{36,37,38,39,40}{
\node(\x) at (5, 15-\c*2)[draw, circle]{\tiny$e_\c^0$};
}

\foreach \x[count=\c from 1]in{33,28,21,12,1}{
\pgfmathtruncatemacro{\ee}{\x+2*\c}{
	\drawpath{\x,...,\ee}{thick}
}
}
\drawpath{11,36,1}{thick, red}
\drawpath{20,10,37,2,12}{thick, red}
\drawpath{27,19,9,38,3,13,21}{thick, red}
\drawpath{32,26,18,8,39,4,14,22,28}{thick, red}
\drawpath{35,31,25,17,7,40,5,15,23,29,33}{thick, red}
\drawpath{34,29,22,13,2,36,10,19,26,31,34}{thick}
\drawpath{30,23,14,3,37,9,18,25,30}{thick}
\drawpath{24,15,4,38,8,17,24}{thick}
\drawpath{16,5,39,7,16}{thick}
\drawpath{33,28,21,12,1}{dashed, thick}
\drawpath{11,20,27,32,35}{dashed, thick}

\end{tikzpicture}
\caption{$A_5$-quiver, with the $E_i$-paths colored in red.}
\label{fig-An}
\end{figure}

%==============================================================================
\subsection{Type $B_n$}\label{sec:embedding:Bn}
Using the explicit expression of the positive representations from \cite{Ip3}, we choose the reduced word 
\Eq{\label{Bni}
\bi = (1212\; 32123\; 4321234... n\; (n-1)\;...1...\;(n-1)\; n).
}
Here recall that $1$ is short and all other roots are long.
Then $n_1=n$ and $n_i = 2n+2-2i$.
\Eqn{
\bf_1&=X(f_{1}^{-n},..., f_1^{n-1}),\\
\bf_i&=X(f_i^{-2n+2i-2},..., f_i^{2n-2i+1})& i\geq 2,\\
\be_1&=X(f_1^n, f_2^{2n-3}, f_1^{n-1}, f_2^{2n-5},..., f_2^1, f_1^1, e_1^0, f_1^{-1}, f_2^{-1},..., f_2^{-2n+3}),\\
\be_i&=X(f_i^{2n-2i+2}, f_{i+1}^{2n-2i-1}, f_i^{2n-2i}, f_{i+1}^{2n-2i-3},..., f_{i+1}^1, f_i^2, e_i^0, f_i^{-2}, f_{i+1}^{-1},..., f_{i+1}^{-2n+2i+1})& i\geq 2.
}
The initial terms are given by
$$X_{e_i^0}=\case{e(-2u_1^1)&i=1,\\ e(-2u_i^2)&i>1.}$$
The quiver is shown in Figure \ref{fig-Bn}. Both the $E_i$-path and $F_i$-path follow a zig-zag shaped path in the quiver. Moreover, the quiver can naturally be generalized to arbitrary rank.
\begin{figure}[!htb]
\centering
\begin{tikzpicture}[every node/.style={inner sep=0, minimum size=0.5cm, thick}, x=1.2cm, y=1.2cm]
\foreach \x in{1,...,11}{
\pgfmathtruncatemacro{\ee}{\x-6}
\pgfmathtruncatemacro{\y}{6-abs(\x-6)}
\ifthenelse{\x=1 \OR \x=11}
{\node (\x) at (\x, \y)[draw] {\tiny $f_1^{\ee}$};}
{\node (\x) at (\x, \y)[draw, circle] {\tiny $f_1^{\ee}$};}
}
\foreach \y[count=\c from 1] in{51,42,29,12}{
\foreach \x in {1,...,\c}{
	\pgfmathtruncatemacro{\ind}{2*\x+\y-2}
	\pgfmathtruncatemacro{\indd}{2*\x+\y-1}
	\pgfmathtruncatemacro{\indop}{\y+4*\c-2*\x+2}
	\pgfmathtruncatemacro{\indopp}{\y+4*\c-2*\x+1}
	\pgfmathtruncatemacro{\r}{6-\c}
	\pgfmathtruncatemacro{\ex}{2*\x-2*\c-2}
	\pgfmathtruncatemacro{\exx}{2*\x-2*\c-1}
	\pgfmathtruncatemacro{\exop}{2*\c-2*\x+2}
	\pgfmathtruncatemacro{\exopp}{2*\c-2*\x+1}
	\ifthenelse{\x=1}
	{\node(\ind) at (5-\c+\x,\c-5+\x)[draw] {\tiny $f_\r^{\ex}$};
	\node(\indop) at (7+\c-\x,\c-5+\x)[draw] {\tiny $f_\r^{\exop}$};
	}
	{\node(\ind) at (5-\c+\x,\c-5+\x)[draw,circle]{\tiny $f_\r^{\ex}$};
	\node(\indop) at (7+\c-\x,\c-5+\x)[draw,circle]{\tiny $f_\r^{\exop}$};
	}
	\node(\indd) at (5-\c+\x,\c-4+\x)[draw,circle]{\tiny $f_\r^{\exx}$};
	\node(\indopp) at (7+\c-\x,\c-4+\x)[draw,circle]{\tiny $f_\r^{\exopp}$};
	}
}
\foreach \x[count=\c from 2] in {20,35,46,53}
\node(\x) at (6,8-2*\c)[draw,circle]{\tiny $f_\c^0$};
\foreach \x[count=\c from 1] in {56,...,60}
\node(\x) at (6,7-2*\c)[draw,circle]{\tiny $e_\c^0$};
\drawpath{1,12,29,42,51}{dashed,vthick}
\drawpath{55,50,41,28,11}{dashed,vthick}
\drawpath{1,...,11}{thin}
\drawpath{12,...,28}{vthick}
\drawpath{29,...,41}{vthick}
\drawpath{42,...,50}{vthick}
\drawpath{51,...,55}{vthick}
\drawpath{52,60,54,48,53,44,45,59,47,37,46,33,34,58,36,22,35,18,19,57,21,7,20,5}{vthick}
\drawpath{11,27,10,25,9,23,8,21,7}{red, vthick}
\drawpath{7,56,5}{red, thin}
\drawpath{5,19,4,17,3,15,2,13,1}{red, vthick}
\drawpath{28,40,26,38,24,36,22,57,18,34,16,32,14,30,12}{red, vthick}
\drawpath{41,49,39,47,37,58,33,45,31,43,29}{red, vthick}
\drawpath{50,54,48,59,44,52,42}{red, vthick}
\drawpath{55,60,51}{red, vthick}
\end{tikzpicture}
\caption{$B_5$-quiver, with the $E_i$-paths colored in red.}
\label{fig-Bn}
\end{figure}
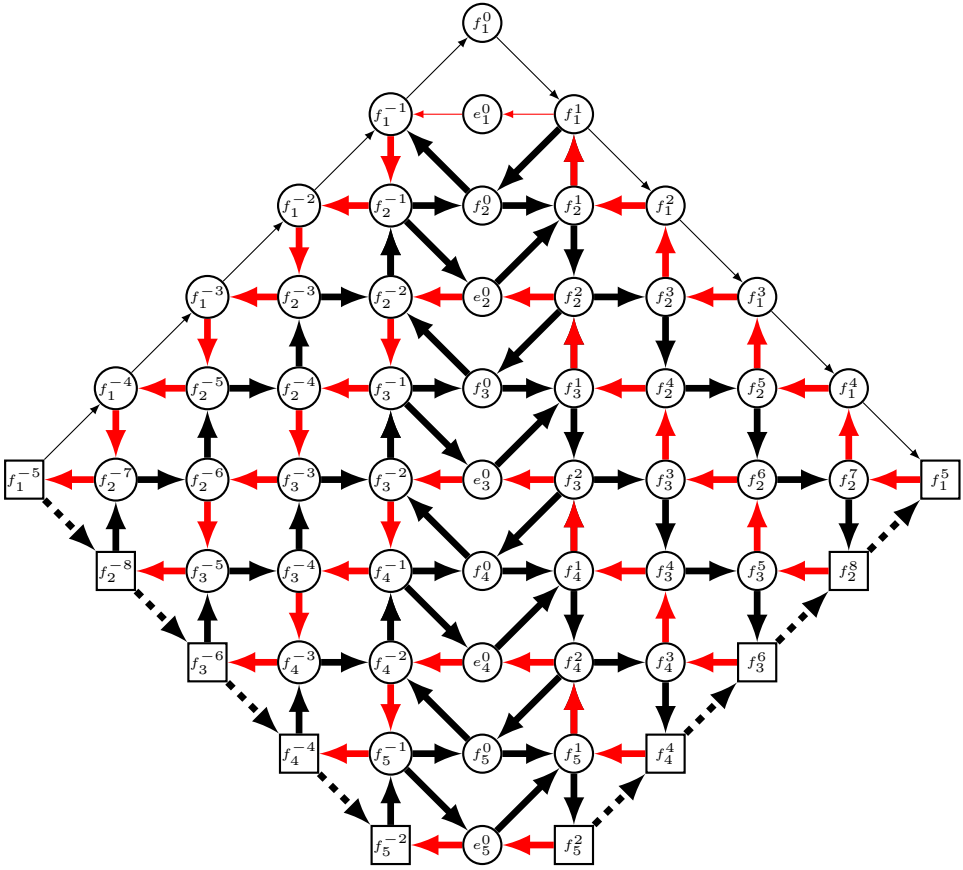
%==============================================================================
\subsection{Type $C_n$}\label{sec:embedding:Cn}

We choose the same word as type $B_n$:
$$\bi = (1212\; 32123\; 4321234... n\; (n-1)\;...1...\;(n-1)\; n).$$

Here $1$ is long and all other roots are short. Then the expression for $\bf_i$ is the same\footnote{Although the algebraic expressions are the same, the $q$-commuting relations are not due to different long and short roots.} as type $B_n$, while $\be_i$ are the same for $i\geq 2$, but modification is made to $\be_1$:
\Eqn{
\be_1&=X(f_1^n, *f_2^{2n-3}, f_1^{n-1}, *f_2^{2n-5},...,*f_2^1, f_1^1, e_1^0, f_1^{-1}, *f_2^{-1},..., *f_2^{-2n+3}),
}
where $X(...,a, *b,...)$ means adding the extra factors as follows:
\Eq{
&...+X_{...}+X_{...,a}+[2]_{q_s} X_{...,a,b}+X_{...,a,b^2}+X_{...,a,b^2,...}+...\label{stops}\\
&=...+X_{...}+(X_{...,a}^{\half}+X_{...,a,b^2}^{\half})^2+X_{...,a,b^2,...}+...\;.\nonumber
}
The initial terms are same as type $B_n$:
$$X_{e_i^0}=\case{e(-2u_1^1)&i=1,\\ e(-2u_i^2)&i>1.}$$

The quiver is shown in Figure \ref{fig-Cn}. We see that the quiver is exactly the same as type $B_n$ case, except that the weights of the arrows are modified. Furthermore, the $E_i$-path for $\be_1$ now ``stops'' at certain vertices (corresponding to \eqref{stops}), which we highlighted in red.

\begin{figure}[!htb]
\centering
\begin{tikzpicture}[every node/.style={inner sep=0, minimum size=0.5cm, thick}, x=1.2cm, y=1.2cm]
\foreach \x in{1,...,11}{
\pgfmathtruncatemacro{\ee}{\x-6}
\pgfmathtruncatemacro{\y}{6-abs(\x-6)}
\ifthenelse{\x=1 \OR \x=11}
{\node (\x) at (\x, \y)[draw] {\tiny $f_1^{\ee}$};}
{\node (\x) at (\x, \y)[draw, circle] {\tiny $f_1^{\ee}$};}
}
\foreach \y[count=\c from 1] in{51,42,29,12}{
\foreach \x in {1,...,\c}{
	\pgfmathtruncatemacro{\ind}{2*\x+\y-2}
	\pgfmathtruncatemacro{\indd}{2*\x+\y-1}
	\pgfmathtruncatemacro{\indop}{\y+4*\c-2*\x+2}
	\pgfmathtruncatemacro{\indopp}{\y+4*\c-2*\x+1}
	\pgfmathtruncatemacro{\r}{6-\c}
	\pgfmathtruncatemacro{\ex}{2*\x-2*\c-2}
	\pgfmathtruncatemacro{\exx}{2*\x-2*\c-1}
	\pgfmathtruncatemacro{\exop}{2*\c-2*\x+2}
	\pgfmathtruncatemacro{\exopp}{2*\c-2*\x+1}
	\ifthenelse{\x=1}
	{
	\node(\ind) at (5-\c+\x,\c-5+\x)[draw] {\tiny $f_\r^{\ex}$};
	\node(\indop) at (7+\c-\x,\c-5+\x)[draw] {\tiny $f_\r^{\exop}$};
	}
	{
	\node(\ind) at (5-\c+\x,\c-5+\x)[draw,circle]{\tiny $f_\r^{\ex}$};
	\node(\indop) at (7+\c-\x,\c-5+\x)[draw,circle]{\tiny $f_\r^{\exop}$};
	}
	\ifthenelse{\r=2}{
	\node(\indd) at (5-\c+\x,\c-4+\x)[draw,circle, red, thick]{\tiny $f_\r^{\exx}$};
	\node(\indopp) at (7+\c-\x,\c-4+\x)[draw,circle,red, thick ]{\tiny $f_\r^{\exopp}$};
	}
	{
	\node(\indd) at (5-\c+\x,\c-4+\x)[draw,circle]{\tiny $f_\r^{\exx}$};
	\node(\indopp) at (7+\c-\x,\c-4+\x)[draw,circle]{\tiny $f_\r^{\exopp}$};
	}
	}
}
\foreach \x[count=\c from 2] in {20,35,46,53}
\node(\x) at (6,8-2*\c)[draw,circle]{\tiny $f_\c^0$};
\foreach \x[count=\c from 1] in {56,...,60}
\node(\x) at (6,7-2*\c)[draw,circle]{\tiny $e_\c^0$};
\drawpath{1,12}{dashed,vthick}
\drawpath{12,29,42,51}{dashed,thin}
\drawpath{55,50,41,28}{dashed,thin}
\drawpath{28,11}{dashed, vthick}
\drawpath{1,...,11}{vthick}
\drawpath{12,...,28}{thin}
\drawpath{29,...,41}{thin}
\drawpath{42,...,50}{thin}
\drawpath{51,...,55}{thin}
\drawpath{52,60,54,48,53,44,45,59,47,37,46,33,34,58,36,22,35,18,19,57,21}{thin}
\drawpath{7,20,5}{vthick}
\drawpath{11,27,10,25,9,23,8,21,7}{red, vthick}
\drawpath{7,56,5}{red, vthick}
\drawpath{5,19,4,17,3,15,2,13,1}{red, vthick}
\drawpath{28,40,26,38,24,36,22,57,18,34,16,32,14,30,12}{red, thin}
\drawpath{41,49,39,47,37,58,33,45,31,43,29}{red, thin}
\drawpath{50,54,48,59,44,52,42}{red, thin}
\drawpath{55,60,51}{red, thin}
\end{tikzpicture}
\caption{$C_5$-quiver, with the $E_i$-paths and the repeated nodes of $\be_1$ colored in red.}
\label{fig-Cn}
\end{figure}
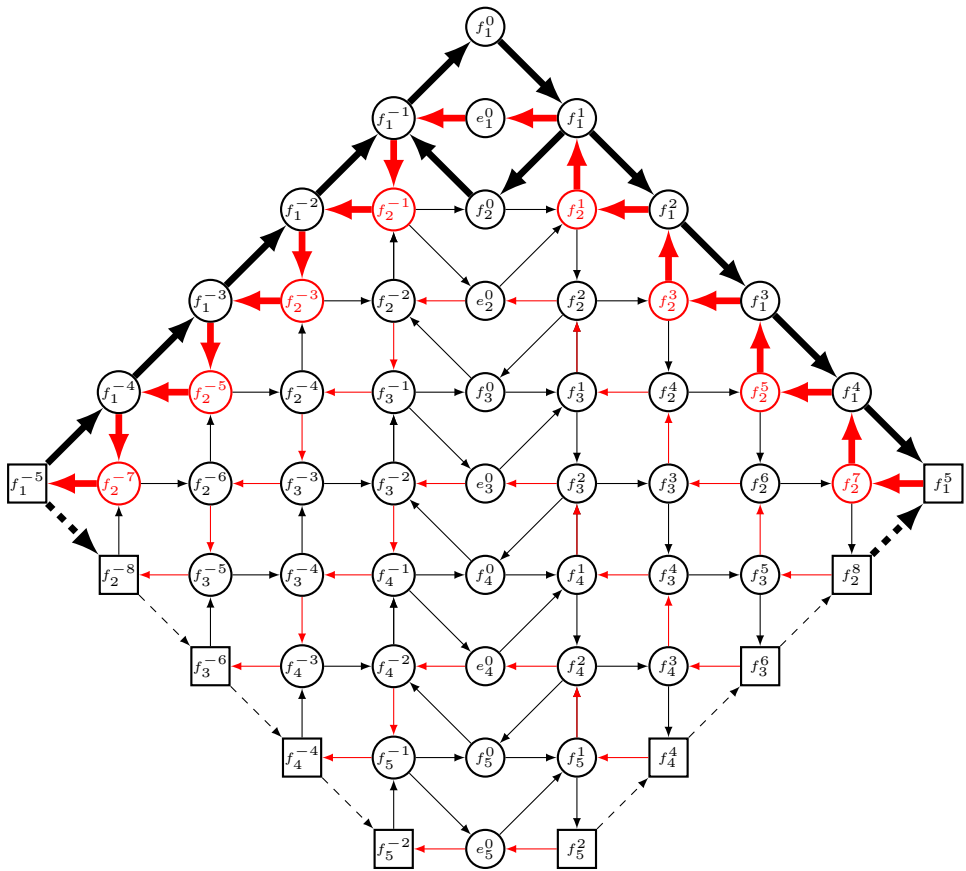
%==============================================================================
\subsection{Type $D_n$}\label{sec:embedding:Dn}

We choose the word corresponding to splitting of type $B_{n-1}$:
\Eq{\label{Dni}
\bi=(012012\;320123\;43201234... (n-1)\;...2012...\;(n-1)),
}
where 0 and 1 are the splitting nodes that are paired.

Then $n_0=n_1=n-1$ and $n_i=2n-2i$ for $i\geq 2$:
\Eqn{
\bf_0&=X(f_0^{-n+1},..., f_0^{n-2}),\\
\bf_1&=X(f_1^{-n+1},..., f_1^{n-2}),\\
\bf_i&=X(f_i^{-2n+2i},..., f_i^{2n-2i-1}) \tab i\geq 2,
}
and
\Eqn{
\be_0&=X(f_0^{n-1}, f_2^{2n-5}, f_1^{n-2}, f_2^{2n-7}, f_0^{n-3}, f_2^{2n-9}, f_1^{n-4}, f_2^{2n-11},..., f_2^1, f_{\e}^1, e_0^0, f_{\e}^{-1}, f_2^{-1},..., f_2^{-2n+5}),\\
\be_1&=X(f_1^{n-1}, f_2^{2n-5}, f_0^{n-2}, f_2^{2n-7}, f_1^{n-3}, f_2^{2n-9}, f_0^{n-4}, f_2^{2n-11},..., f_2^1, f_{1-\e}^1, e_1^0, f_{1-\e}^{-1}, f_2^{-1},..., f_2^{-2n+3}),\\
\be_i&=X(f_i^{2n-2i}, f_{i+1}^{2n-2i-3}, f_i^{2n-2i-2}, f_{i+1}^{2n-2i-5},..., f_{i+1}^1, f_i^2, e_i^0, f_i^{-2}, f_{i+1}^{-1},..., f_{i+1}^{-2n+2i-1})\tab i\geq 2,
}
where $\e=n \mbox{(mod 2)}\in\{0,1\}$.

The initial terms are given by
$$X_{e_i^0}=\case{e(-2u_i^1)&i=0,1,n\mbox{ is even,}\\e(-2u_{1-i}^1)&i=0,1,n\mbox{ is odd,}\\ e(-2u_i^2)&i>1.}$$

The quiver is shown in Figure \ref{fig-Dn}. Note that the action of $E_i$ and $F_i$ are the same as type $B_{n-1}$ for $i\neq 0,1$. Furthermore, it follows naturally that the $B_{n-1}$-quiver comes from a folding of the $D_n$-quiver, with the weights of the arrows appropriately adjusted. In the quiver, we highlight the $E_0$-path in red, where we see that it alternates between root 0 and root 1, while the $E_1$-path interchanges 0 and 1.

\begin{figure}[!htb]
\centering
\begin{tikzpicture}[every node/.style={inner sep=0, minimum size=0.5cm, thin}, x=1.3cm, y=1.3cm]
\foreach \x in{1,...,11}{
\pgfmathtruncatemacro{\ee}{\x-6}
\pgfmathtruncatemacro{\y}{6-abs(\x-6)}
\ifthenelse{\x=1 \OR \x=11}{
	\node (\x) at (\x, \y)[draw] {\tiny $f_1^{\ee}$};
	\node (a-\x) at (\x, \y+0.8)[draw] {\tiny $f_0^{\ee}$};
	}
	{
	\node (\x) at (\x, \y)[draw, circle] {\tiny $f_1^{\ee}$};
	\ifthenelse{\x<6}{
	\node (a-\x) at (\x-0.25, \y+0.55)[draw, circle] {\tiny $f_0^{\ee}$};
	}{
	\ifthenelse{\x>6}{
	\node (a-\x) at (\x+0.25, \y+0.55)[draw, circle] {\tiny $f_0^{\ee}$};
	}
	{
	\node (a-\x) at (\x, \y+0.8)[draw, circle] {\tiny $f_0^{\ee}$};
	}
	}
	}
	}
\foreach \y[count=\c from 1] in{51,42,29,12}{
\foreach \x in {1,...,\c}{
	\pgfmathtruncatemacro{\ind}{2*\x+\y-2}
	\pgfmathtruncatemacro{\indd}{2*\x+\y-1}
	\pgfmathtruncatemacro{\indop}{\y+4*\c-2*\x+2}
	\pgfmathtruncatemacro{\indopp}{\y+4*\c-2*\x+1}
	\pgfmathtruncatemacro{\r}{6-\c}
	\pgfmathtruncatemacro{\ex}{2*\x-2*\c-2}
	\pgfmathtruncatemacro{\exx}{2*\x-2*\c-1}
	\pgfmathtruncatemacro{\exop}{2*\c-2*\x+2}
	\pgfmathtruncatemacro{\exopp}{2*\c-2*\x+1}
	\ifthenelse{\x=1}
	{\node(\ind) at (5-\c+\x,\c-5+\x)[draw] {\tiny $f_\r^{\ex}$};
	\node(\indop) at (7+\c-\x,\c-5+\x)[draw] {\tiny $f_\r^{\exop}$};
	}
	{\node(\ind) at (5-\c+\x,\c-5+\x)[draw,circle]{\tiny $f_\r^{\ex}$};
	\node(\indop) at (7+\c-\x,\c-5+\x)[draw,circle]{\tiny $f_\r^{\exop}$};
	}
	\node(\indd) at (5-\c+\x,\c-4+\x)[draw,circle]{\tiny $f_\r^{\exx}$};
	\node(\indopp) at (7+\c-\x,\c-4+\x)[draw,circle]{\tiny $f_\r^{\exopp}$};
	}
}
\foreach \x[count=\c from 2] in {20,35,46,53}
\node(\x) at (6,8-2*\c)[draw,circle]{\tiny $f_\c^0$};
\foreach \x[count=\c from 2] in {57,...,60}
\node(\x) at (6,7-2*\c)[draw,circle]{\tiny $e_\c^0$};
\node(56) at (6,4.75)[draw,circle]{\tiny $e_1^0$};
\node(a-56) at (6,5.25)[draw,circle]{\tiny $e_0^0$};
\drawpath{1,12,29,42,51}{dashed,thin}
\drawpath{a-1,12}{dashed,thin}
\drawpath{55,50,41,28,11}{dashed,thin}
\drawpath{28,a-11}{dashed,thin}
\drawpath{1,...,11}{thin}
\drawpath{a-1,a-2,a-3,a-4,a-5,a-6,a-7,a-8,a-9,a-10,a-11}{thin}
\drawpath{12,...,28}{thin}
\drawpath{29,...,41}{thin}
\drawpath{42,...,50}{thin}
\drawpath{51,...,55}{thin}
\drawpath{52,60,54,48,53,44,45,59,47,37,46,33,34,58,36,22,35,18,19,57,21,7,20,5}{thin}
\drawpath{11,27,10,25,9,23,8,21,7,56,5,19,4,17,3,15,2,13,1}{thin}
\drawpath{a-11,27,a-10,25,a-9,23,a-8,21,a-7,a-56,a-5,19,a-4,17,a-3,15,a-2,13,a-1}{thin}
\drawpath{28,40,26,38,24,36,22,57,18,34,16,32,14,30,12}{thin}
\drawpath{41,49,39,47,37,58,33,45,31,43,29}{thin}
\drawpath{50,54,48,59,44,52,42}{thin}
\drawpath{55,60,51}{thin}
\drawpath{a-7,20,a-5}{thin}
\drawpath{a-11,27,10,25,a-9,23,8,21,a-7,a-56,a-5,19,4,17,a-3,15,2,13,a-1}{red, thick}
\end{tikzpicture}
\caption{$D_6$-quiver, with the $E_i$-path of $\be_0$ colored in red.}
\label{fig-Dn}
\end{figure}
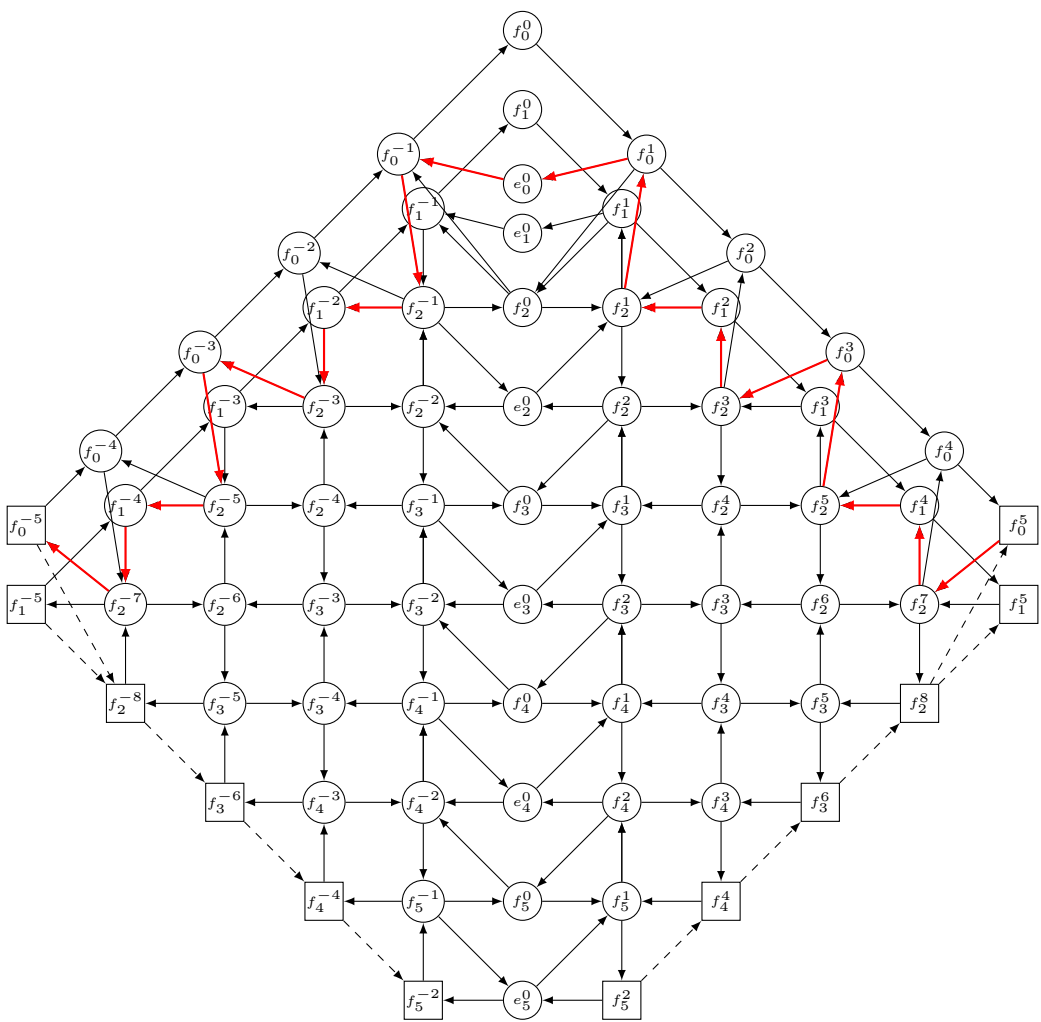

%==============================================================================
\subsection{Type $E_n$}\label{sec:embedding:En}
For type $E_n$, we let $0$ be the extra node (cf. Definition \ref{Dynkin}). The explicit expression of the positive representations for the generators $\bf_i$ and $K_i'$ is given by Theorem \ref{pos2} and Theorem \ref{mainThm}, while the expression for the generators $\be_i$ is given in the Appendix of the author's Ph.D. Thesis \cite{IpTh} \footnote{Here we choose $\bi$ to begin with 343 instead of 434 for technical convenience.}. The explicit expression is, however, rather ad hoc. 

Using the procedure describe in the beginning of this section, we can solve for the cluster variables $X_i$ to rewrite the expression as certain $E$-paths on some quiver diagrams, which is a lot easier to visualize. Interestingly, we see that unlike type $A$ to $D$, most of the actions of $\be_i$ actually pass through rows corresponding to other roots throughout the whole quiver.

As before, we only list the representations of $\be_i$ and $\bf_i$, while again the representation of the $K_i$ and $K_i'$ variables are expressed as the product of the last term of $\be_i$ and $\bf_i$ with $X_{f_i^{-n_i}}$ and $X_{f_i^{n_i}}$ respectively.

%==============================================================================

\subsubsection{Type $E_6$}\label{sec:embedding:E6}
Following \cite{IpTh}, we choose the longest word to be
$$\bi= (3\;43\;034\; 230432\;12340321\;5432103243054321),$$ 
which comes from the embedding of Dynkin diagram $$A_1\subset A_2\subset A_3\subset D_4\subset D_5\subset E_6$$ by successively adding the nodes 3,4,0,2,1,5 to the diagram.

Then the $\bf_i$ variables are expressed as
\Eqn{
\bf_1&=X(f_1^{-4},...,f_1^3),\\
\bf_2&=X(f_2^{-7},...,f_2^6),\\
\bf_3&=X(f_3^{-11},..., f_3^{10}),\\
\bf_4&=X(f_4^{-7},..., f_4^6),\\
\bf_5&=X(f_5^{-2},f_5^{-1},f_5^0,f_5^1),\\
\bf_0&=X(f_0^{-5},..., f_0^{4}),
}
while the $\be_i$ variables are expressed as certain paths on the quiver:
\Eqn{
\be_1&=X(f_1^4,e_1^0),\\
\be_2&=X(f_2^7, f_1^{3}, f_2^{5}, f_3^{8}, f_0^{3}, f_3^{6}, f_4^{3}, 
 f_3^{4}, f_0^{1}, f_3^{2}, e_2^0, f_3^{-2}, 
 f_0^{-1}, f_3^{-4}, f_4^{-3}, f_3^{-6}, f_0^{-3}, f_3^{-8}, f_2^{-5},
  f_1^{-3}),\\
\be_3&=X(f_3^{11}, f_2^{6}, f_3^{9}, f_4^{5}, f_3^{7}, f_2^{3}, f_3^{5}, f_2^{1}, f_3^{3}, f_4^{1}, f_3^{1}, e_3^0, f_3^{-1}, f_4^{-1}, f_3^{-3}, f_2^{-1}, f_3^{-5}, f_2^{-3}, f_3^{-7}, f_4^{-5}, f_3^{-9}, f_2^{-6})\\
\be_4&=X(f_4^7, f_3^{10}, f_0^{4}, f_3^{8}, f_2^{4}, f_1^{1}, f_2^{2}, 
e_4^0, f_2^{-2}, f_1^{-1}, f_2^{-4}, f_3^{-8}, f_0^{-4}, f_3^{-10}),\\
\be_5&=X(f_5^2, f_4^{6}, f_3^{9}, f_2^{5}, f_1^{2}, e_5^0, f_1^{-2}, 
 f_2^{-5}, f_3^{-9}, f_4^{-6}),\\
\be_0&=X(f_0^{5}, f_3^{10}, f_4^{6}, f_5^{1}, f_4^{4}, f_3^{6}, f_0^{2}, 
 f_3^{4}, f_4^{2}, e_0^0, f_4^{-2}, f_3^{-4}, f_0^{-2}, f_3^{-6}, 
 f_4^{-4}, f_5^{-1}, f_4^{-6}, f_3^{-10}).
}

The initial terms are given by
\Eqn{
&X_{e_1^0}=e(-2u_1^4),&& X_{e_2^0}=e(-2u_3^2),&& X_{e_3^0}=e(-2u_3^1),\\
&X_{e_4^0}=e(-2u_2^2),&& X_{e_5^0}=e(-2u_1^2),&& X_{e_0^0}=e(-2u_4^2).
}

The quiver is shown in Figure \ref{fig-E6}, where the labeling of each row is given by $f_i^{-n_i},..., f_i^{n_i}$, hence each $F_i$-path is represented as a horizontal path. The different $E_i$-paths, starting from $f_i^{n_i}$ and ending at $f_i^{-n_i}$, are shown in different colors.

\begin{landscape}
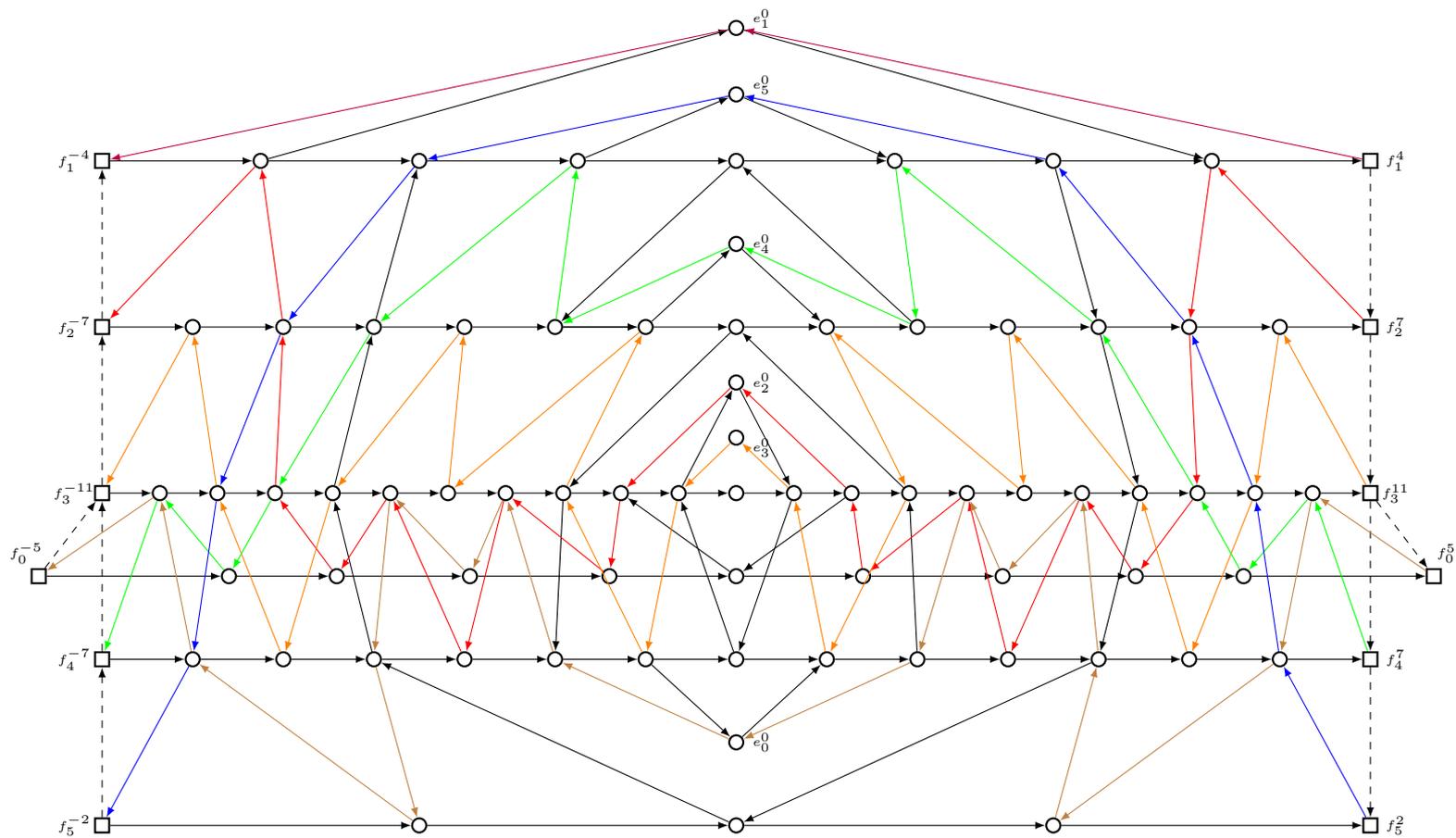
\begin{figure}[!htb]
\centering
\begin{tikzpicture}[every node/.style={inner sep=0, minimum size=0.2cm, thick}, x=0.18cm, y=0.47cm]
\xdef\c{0}
\foreach \y[count=\d from 0] in {9,15,23,15,5,11}{
	\foreach \x in {1,..., \y}{
		\pgfmathtruncatemacro{\ind}{\x+\c}
		\pgfmathsetmacro{\xx}{100*(\x-1)/(\y-1)}
		\ifthenelse{\d=5}{
			\ifthenelse{\x=1 \OR \x=11}{
				\node(\ind) at (\x-6+\xx, 7.5)[draw]{};
				}
				{\ifthenelse{\x=3 \OR \x=4 \OR \x=8\OR \x=9}{
				\node(\ind) at (0.5*\x-3+\xx, 7.5)[draw,circle]{};
				}{
				\node(\ind) at (\xx, 7.5)[draw, circle]{};
				}
				}
			}
		{	\ifthenelse{\x=1 \OR \x=\y}
		{
		\node(\ind) at (\xx, 20-5*\d)[draw]{};
		}
		{
		\node(\ind) at (\xx, 20-5*\d)[draw, circle]{};
		}
		}
	}
	\xdef\c{\c+\y}
}
\foreach \y [count=\d from 79] in {24,13.33,11.67,17.5,22,2.5}{
\node(\d) at (50, \y)[draw, circle]{};
}
\drawpath{63,48,25,10,1}{dashed,thin}
\drawpath{68,25}{dashed,thin}
\drawpath{9,24,47,62,67}{dashed,thin}
\drawpath{47,78}{dashed, thin}
\xdef\cc{1}
\foreach \y in {9,24,47,62,67,78}{
	\drawpath{\cc,...,\y}	{}
	\pgfmathtruncatemacro{\ind}{\y+1}
	\xdef\cc{\ind}
}
\drawpath{2,79,8}{thin}
\drawpath{4,83,6}{thin}
\drawpath{65,51,29,13,3}{thin}
\drawpath{7,21,43,59,65}{thin}
\drawpath{19,5,15,16,82,18}{thin}
\drawpath{57,39,17,33,53}{thin}
\drawpath{38,73,34}{thin}
\drawpath{54,84,56}{thin}
\drawpath{35,80,37,55,35}{thin}
\drawpath{9,79,1}{purple,thin}
\drawpath{24,8,22,44,76,42,58,40,74,38,80,34,72,32,52,30,70,28,12,2,10}{red,thin}
\drawpath{47,23,45,60,43,20,41,18,39,56,37,81,35,54,33,16,31,14,29,50,27,11,25}{orange,thin}
\drawpath{62,46,77,44,21,6,19,82,15,4,13,28,69,26,48}{green,thin}
\drawpath{67,61,45,22,7,83,3,12,27,49,63}{blue,thin}
\drawpath{78,46,61,66,59,42,75,40,57,84,53,32,71,30,51,64,49,26,68}{brown,thin}
\foreach \y[count=\c from 1] in {4,7,11,7,2}{
\node at (-2.2, 25-5*\c){\tiny$f_\c^{-\y}$};
\node at (102, 25-5*\c){\tiny$f_\c^{\y}$};
}
\node at (-6, 8.2){\tiny$f_0^{-5}$};
\node at (106, 8.2){\tiny$f_0^{5}$};
\foreach \y[count=\c from 0] in {2.5, 24.3,13.33,11.37,17.5,22.3}{
\node at (52, \y){\tiny$e_\c^0$};
}
\end{tikzpicture}
\caption{$E_6$-quiver, with the $E_i$-paths colored in different colors.}
\label{fig-E6}
\end{figure}
\end{landscape}
%==============================================================================
\subsubsection{Type $E_7$}\label{sec:embedding:E7}
Following \cite{IpTh}, we choose the longest word to be
$$\bi=(3\;43\;034\;230432\;12340321\;5432103243054321\;654320345612345034230123456),$$
which comes from the embedding of Dynkin diagram $$A_1\subset A_2\subset A_3\subset D_4\subset D_5\subset E_6\subset E_7$$ by successively adding the nodes 3,4,0,2,1,5,6 to the diagram.

Then the $\bf_i$ variables are expressed as
\Eqn{
\bf_1&=X(f_1^{-6},..., f_1^5),\\
\bf_2&=X(f_2^{-11},..., f_2^{10}),\\
\bf_3&=X(f_3^{-17},..., f_3^{16}),\\
\bf_4&=X(f_4^{-12},..., f_4^{11}),\\
\bf_5&=X(f_5^{-6},..., f_5^5),\\
\bf_6&=X(f_6^{-3},..., f_6^2),\\
\bf_0&=X(f_0^{-8},..., f_0^7),
}
while the $\be_i$ variables are expressed as certain paths on the quiver:
\Eqn{
\be_1=&X(f_1^{7}, f_2^{10}, f_3^{15}, f_4^{10}, f_5^{4}, f_6^{1}, f_5^{2}, 
 f_4^{6}, f_3^{9}, f_2^{5}, f_1^{3}, e_1^0, f_1^{-1}, f_2^{-5}, 
 f_3^{-9}, f_4^{-6}, f_5^{-2}, f_6^{-1}, f_5^{-4}, f_4^{-10}, 
 f_3^{-15}, f_2^{-10}),\\
\be_2=&X(f_2^{11}, f_3^{16}, f_0^{7}, f_3^{14}, f_4^{9}, f_5^{3}, f_4^{7}, 
 f_3^{10}, f_0^{4}, f_3^{8}, f_2^{4}, f_1^{2}, f_2^{2}, e_2^0, \\
&\tab f_2^{-2}, f_1^{0}, f_2^{-4}, f_3^{-8}, f_0^{-4}, f_3^{-10}, f_4^{-7}, 
 f_5^{-3}, f_4^{-9}, f_3^{-14}, f_0^{-7}, f_3^{-16}),\\
\be_3=&X(f_3^{17}, f_4^{11}, f_3^{15}, f_2^{9}, f_3^{13}, f_4^{8}, f_3^{11}, 
 f_2^{6}, f_3^{9}, f_4^{5}, f_3^{7}, f_2^{3}, f_3^{5}, f_2^{1}, 
 f_3^{3}, f_4^{1}, f_3^{1}, e_3^0, \\
&\tab f_3^{-1}, f_4^{-1}, f_3^{-3}, f_2^{-1}, f_3^{-5}, 
 f_2^{-3}, f_3^{-7}, f_4^{-5}, f_3^{-9}, f_2^{-6}, f_3^{-11}, 
 f_4^{-8}, f_3^{-13}, f_2^{-9}, f_3^{-15}, f_4^{-11}),\\
\be_4=&X(f_4^{12}, f_5^{5}, f_4^{10}, f_3^{14}, f_0^{6}, f_3^{12}, f_2^{7}, 
 f_1^{4}, f_2^{5}, f_3^{8}, f_0^{3}, f_3^{6}, f_4^{3}, f_3^{4}, 
 f_0^{1}, f_3^{2}, e_4^0, \\
&\tab  f_3^{-2}, f_0^{-1}, 
 f_3^{-4}, f_4^{-3}, f_3^{-6}, f_0^{-3}, f_3^{-8}, f_2^{-5}, f_1^{-2},
  f_2^{-7}, f_3^{-12}, f_0^{-6}, f_3^{-14}, f_4^{-10}, f_5^{-5}),\\
\be_5=&X(f_5^{6}, f_6^{2}, f_5^{4}, f_4^{9}, f_3^{13}, f_2^{8}, f_1^{4}, 
e_5^0, f_1^{-4}, f_2^{-8}, f_3^{-13}, f_4^{-9}, f_5^{-4}, 
 f_6^{-2}),\\
\be_6=&X(f_6^3, e_6^0),\\
\be_0=&X(f_0^{8}, f_3^{16}, f_2^{10}, f_1^{6}, f_2^{8}, f_3^{12}, f_0^{5}, 
 f_3^{10}, f_4^{6}, f_5^{1}, f_4^{4}, f_3^{6}, f_0^{2}, f_3^{4}, 
 f_4^{2}, e_0^0,\\
&\tab  f_4^{-2}, f_3^{-4}, f_0^{-2}, f_3^{-6}, f_4^{-4}, 
 f_5^{-1}, f_4^{-6}, f_3^{-10}, f_0^{-5}, f_3^{-12}, f_2^{-8}, 
 f_1^{-4}, f_2^{-10}, f_3^{-16}).
}

The initial terms are given by
\Eqn{
&X_{e_1^0}=e(-2u_1^2),&& X_{e_2^0}=e(-2u_2^2),&& X_{e_3^0}=e(-2u_3^1),&&X_{e_4^0}=e(-2u_3^2),\\
& X_{e_5^0}=e(-2u_1^4),&& X_{e_6^0}=e(-2u_6^3), && X_{e_0^0}=e(-2u_4^2).
}

The quiver is shown in Figure \ref{fig-E7}, again the labeling of each row is given by $f_i^{-n_i},..., f_i^{n_i}$, hence each $F_i$-path is represented as a horizontal path. The different $E_i$-paths, starting from $f_i^{n_i}$ and ending at $f_i^{-n_i}$, are shown in different colors.
\begin{landscape}
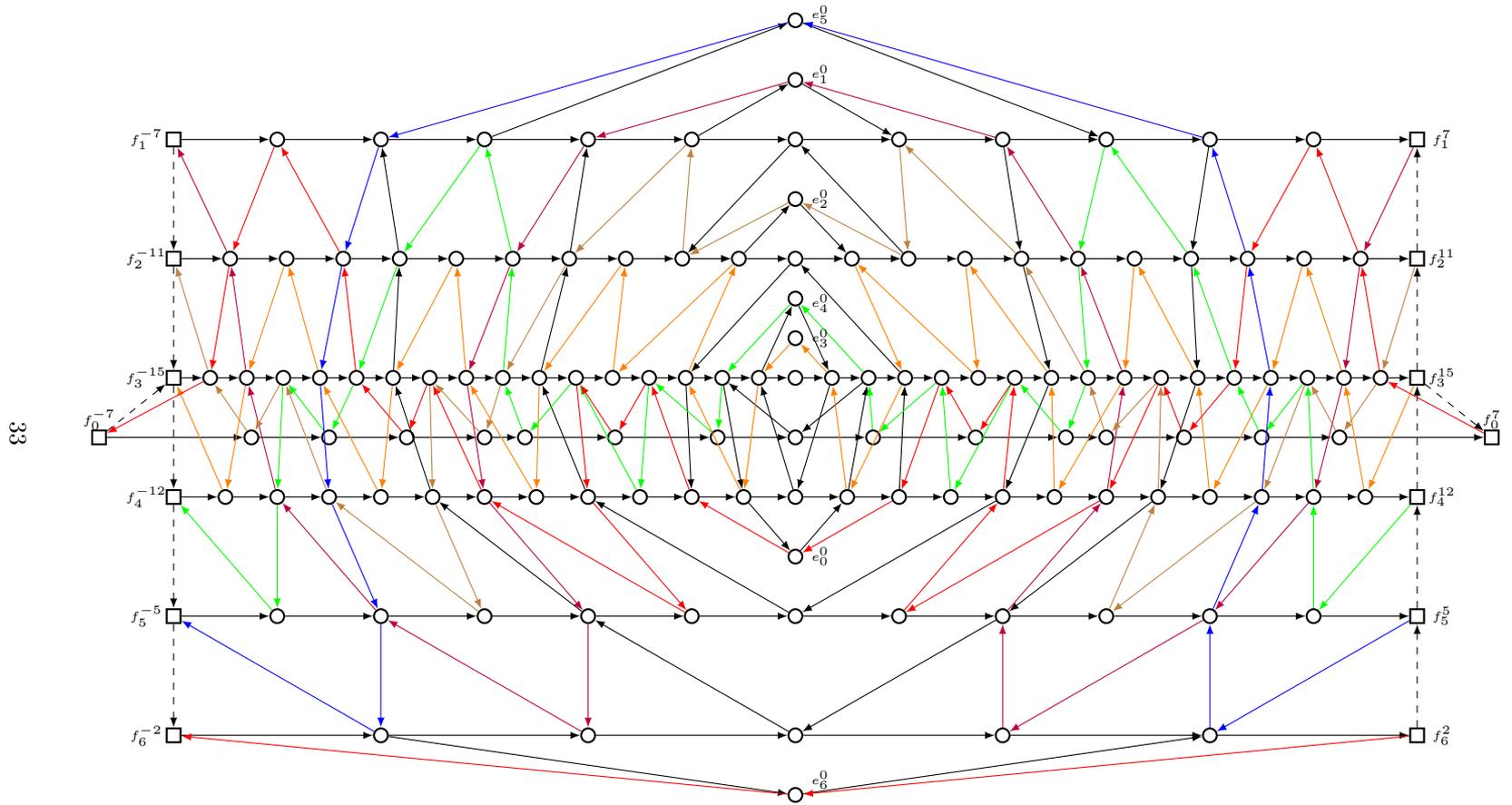
\begin{figure}[!htb]
\centering
\begin{tikzpicture}[every node/.style={inner sep=0, minimum size=0.2cm, thick}, x=0.18cm, y=0.43cm]
\xdef\c{0}
\foreach \y[count=\d from 0] in {13,23,35,25,13,7,17}{
	\foreach \x in {1,..., \y}{
		\pgfmathtruncatemacro{\ind}{\x+\c}
		\pgfmathsetmacro{\xx}{100*(\x-1)/(\y-1)}
		\ifthenelse{\d=6}{
			\ifthenelse{\x=1 \OR \x=17}{
				\node(\ind) at (0.75*\x-0.75*9+\xx, 10)[draw]{};
				}
				{\ifthenelse{\x=6 \OR \x=7 \OR \x=11\OR\x=12}{
				\node(\ind) at (\x-9+\xx, 10)[draw,circle]{};
				}{
				\node(\ind) at (\xx, 10)[draw, circle]{};
				}
				}
			}
		{	\ifthenelse{\x=1 \OR \x=\y}
		{
		\node(\ind) at (\xx, 20-4*\d)[draw]{};
		}
		{
		\node(\ind) at (\xx, 20-4*\d)[draw, circle]{};
		}
		}
	}
	\xdef\c{\c+\y}
}
\foreach \y [count=\d from 134] in {22,18,13.33,14.66,24,-2,6}{
\node(\d) at (50, \y)[draw, circle]{};
}
\drawpath{1, 14, 37, 72, 97, 110}{dashed,thin}
\drawpath{117,37}{dashed,thin}
\drawpath{116, 109, 96, 71, 36, 13}{dashed,thin}
\drawpath{71, 133}{dashed, thin}
\xdef\cc{1}
\foreach \y in {13, 36, 71, 96, 109, 116, 133}{
	\drawpath{\cc,...,\y}	{}
	\pgfmathtruncatemacro{\ind}{\y+1}
	\xdef\cc{\ind}
}
\drawpath{4,138,10}{thin}
\drawpath{6,134,8}{thin}
\drawpath{11, 32, 65, 91, 105, 113, 101, 77, 43, 18, 3}{thin}
\drawpath{9, 29, 61, 88, 103, 80, 47, 21, 5}{thin}
\drawpath{7,23,24,135,26,27,7}{thin}
\drawpath{86, 57, 25, 51, 82}{thin}
\drawpath{84,53, 137, 55, 84}{thin}
\drawpath{125, 52, 83, 140, 85, 56, 125}{thin}
\drawpath{111,139,115}{thin}
\drawpath{13, 35, 69, 94, 107, 114, 105, 90, 63, 30, 9, 134, 5, 20, 45, 78,101, 112, 99, 74, 39, 15, 1}{purple,thin}
\drawpath{36, 70, 132, 68, 93, 106, 91, 64, 129, 62, 29, 8, 27, 135, 23, 6, 21, 46, 121, 44, 77, 100, 75, 40, 118, 38, 14}{brown,thin}
\drawpath{71, 95, 69, 34, 67, 92, 65, 31, 63, 89, 61, 28, 59, 26, 57, 85, 55,136,53, 83, 51, 24, 49, 22, 47, 79, 45, 19, 43, 76, 41, 16, 39, 73, 37}{orange,thin}
\drawpath{96, 108, 94, 68, 131, 66, 32, 10, 30, 62, 128, 60, 87, 58, 126, 56, 137, 52, 124, 50, 81, 48, 122, 46, 20, 4, 18, 42, 119, 40, 74, 98, 72}{green,thin}
\drawpath{109, 115, 107, 93, 67, 33, 11, 138, 3, 17, 41, 75, 99, 111, 97}{blue,thin}
\drawpath{116,139,110}{red,thin}
\drawpath{133, 70, 35, 12, 33, 66, 130, 64, 90, 104, 88, 60, 127, 58, 86, 140, 82, 50, 123, 48, 80, 102, 78, 44, 120, 42, 17, 2, 15, 38, 117}{red,thin}
\foreach \y[count=\c from 1] in {7,11,15,12,5,2}{
\node at (-2.2, 24-4*\c){\tiny$f_\c^{-\y}$};
\node at (102, 24-4*\c){\tiny$f_\c^{\y}$};
}
\node at (-6, 10.6){\tiny$f_0^{-7}$};
\node at (106, 10.6){\tiny$f_0^{7}$};
\foreach \y[count=\c from 0] in {6,22.2,18,13.33,14.5,24.2,-1.5}{
\node at (52, \y){\tiny$e_\c^0$};
}
\end{tikzpicture}
\caption{$E_7$-quiver, with the $E_i$-paths colored in different colors.}
\label{fig-E7}
\end{figure}
\end{landscape}

%==============================================================================
\subsubsection{Type $E_8$}\label{sec:embedding:E8}
Following \cite{IpTh}, we choose the longest word to be
\Eqn{
\bi=&(3\;43\;034\;230432\;12340321\;5432103243054321\;654320345612345034230123456\\
&765432103243546503423012345676543203456123450342301234567),
}
which comes from the embedding of Dynkin diagram $$A_1\subset A_2\subset A_3\subset D_4\subset D_5\subset E_6\subset E_7\subset E_8$$ by successively adding the nodes 3,4,0,2,1,5,6,7 to the diagram.

Then the $\bf_i$ variables are expressed as
\Eqn{
\bf_1&=X(f_1^{-10},..., f_1^9),\\
\bf_2&=X(f_2^{-19},..., f_2^{18}),\\
\bf_3&=X(f_3^{-29},..., f_3^{28}),\\
\bf_4&=X(f_4^{-22},..., f_4^{21}),\\
\bf_5&=X(f_5^{-14},..., f_5^{13}),\\
\bf_6&=X(f_6^{-9},..., f_6^8),\\
\bf_7&=X(f_7^{-3},..., f_7^2),\\
\bf_0&=X(f_0^{-14},..., f_0^{13}),
}
while the $\be_i$ variables are expressed as certain paths on the quiver:
\begin{tiny}
\Eqn{
\be_1=&X(f_1^{10}, f_2^{18}, f_3^{27}, f_4^{20}, f_5^{12}, f_6^{7}, f_5^{10}, 
 f_6^{5}, f_5^{8}, f_4^{14}, f_3^{19}, f_2^{12}, f_1^{6}, f_2^{10}, 
 f_3^{15}, f_4^{10}, f_5^{4}, f_6^{1}, f_5^{2}, f_4^{6}, f_3^{9}, 
 f_2^{5}, f_1^{2}, e_1^0, \\
 &\tab f_1^{-2}, f_2^{-5}, f_3^{-9}, f_4^{-6}, 
 f_5^{-2}, f_6^{-1}, f_5^{-4}, f_4^{-10}, f_3^{-15}, f_2^{-10}, 
 f_1^{-6}, f_2^{-12}, f_3^{-19}, f_4^{-14}, f_5^{-8}, f_6^{-5}, 
 f_5^{-10}, f_6^{-7}, f_5^{-12}, \\
 &\tab f_4^{-20}, f_3^{-27}, f_2^{-18}),\\
\be_2=&X(f_2^{19}, f_3^{28}, f_0^{13}, f_3^{26}, f_4^{19}, f_5^{11}, f_4^{17},
  f_5^{9}, f_4^{15}, f_3^{20}, f_0^{9}, f_3^{18}, f_2^{11}, f_3^{16}, 
 f_0^{7}, f_3^{14}, f_4^{9}, f_5^{3}, f_4^{7}, f_3^{10}, f_0^{4}, 
 f_3^{8}, f_2^{4}, f_1^{1}, f_2^{2}, e_2^0,\\
 &\tab  f_2^{-2}, f_1^{-1}, f_2^{-4}, f_3^{-8}, f_0^{-4}, f_3^{-10}, f_4^{-7}, f_5^{-3}, 
 f_4^{-9}, f_3^{-14}, f_0^{-7}, f_3^{-16}, f_2^{-11}, f_3^{-18}, 
 f_0^{-9}, f_3^{-20}, f_4^{-15}, f_5^{-9}, f_4^{-17}, f_5^{-11},\\ 
&\tab  f_4^{-19}, f_3^{-26}, f_0^{-13}, f_3^{-28}),\\
\be_3=&X(f_3^{29}, f_4^{21}, f_3^{27}, f_2^{17}, f_3^{25}, f_4^{18}, f_3^{23},
  f_4^{16}, f_3^{21}, f_2^{13}, f_3^{19}, f_4^{13}, f_3^{17}, 
 f_4^{11}, f_3^{15}, f_2^{9}, f_3^{13}, f_4^{8}, f_3^{11}, f_2^{6}, 
 f_3^{9}, f_4^{5}, f_3^{7}, f_2^{3}, f_3^{5}, f_2^{1}, f_3^{3}, 
 f_4^{1}, f_3^{1}, e_3^0,\\
&\tab f_3^{-1}, f_4^{-1}, f_3^{-3}, f_2^{-1}, f_3^{-5}, f_2^{-3}, 
 f_3^{-7}, f_4^{-5}, f_3^{-9}, f_2^{-6}, f_3^{-11}, f_4^{-8}, 
 f_3^{-13}, f_2^{-9}, f_3^{-15}, f_4^{-11}, f_3^{-17}, f_4^{-13}, 
 f_3^{-19}, f_2^{-13}, f_3^{-21}, f_4^{-16}, \\
&\tab f_3^{-23},f_4^{-18}, f_3^{-25}, f_2^{-17}, f_3^{-27}, f_4^{-21}),\\
\be_4=&X(f_4^{22}, f_5^{13}, f_4^{20}, f_3^{26}, f_0^{12}, f_3^{24}, f_2^{15},
  f_3^{22}, f_0^{10}, f_3^{20}, f_4^{14}, f_5^{7}, f_4^{12}, f_5^{5}, 
 f_4^{10}, f_3^{14}, f_0^{6}, f_3^{12}, f_2^{7}, f_1^{3}, f_2^{5}, 
 f_3^{8}, f_0^{3}, f_3^{6}, f_4^{3}, f_3^{4}, f_0^{1}, f_3^{2}, 
 e_4^0, \\
&\tab  f_3^{-2}, f_0^{-1}, f_3^{-4}, f_4^{-3}, 
 f_3^{-6}, f_0^{-3}, f_3^{-8}, f_2^{-5}, f_1^{-3}, f_2^{-7}, 
 f_3^{-12}, f_0^{-6}, f_3^{-14}, f_4^{-10}, f_5^{-5}, f_4^{-12}, 
 f_5^{-7}, f_4^{-14}, f_3^{-20}, f_0^{-10}, f_3^{-22}, \\
&\tab f_2^{-15}, f_3^{-24}, f_0^{-12}, f_3^{-26}, f_4^{-20}, f_5^{-13}),\\
\be_5=&X(f_5^{14}, f_6^{8}, f_5^{12}, f_4^{19}, f_3^{25}, f_2^{16}, f_1^{8}, 
 f_2^{14}, f_3^{21}, f_4^{15}, f_5^{8}, f_6^{4}, f_5^{6}, f_6^{2}, 
 f_5^{4}, f_4^{9}, f_3^{13}, f_2^{8}, f_1^{4}, e_5^0,\\
&\tab  f_1^{-4}, 
 f_2^{-8}, f_3^{-13}, f_4^{-9}, f_5^{-4}, f_6^{-2}, f_5^{-6}, 
 f_6^{-4}, f_5^{-8}, f_4^{-15}, f_3^{-21}, f_2^{-14}, f_1^{-8}, 
 f_2^{-16}, f_3^{-25}, f_4^{-19}, f_5^{-12}, f_6^{-8}),\\
\be_6=&X(f_6^{9}, f_7^{2}, f_6^{6}, f_6^{7}, f_5^{10}, f_5^{11}, f_4^{17}, 
 f_4^{18}, f_3^{23}, f_3^{24}, [f_0^{11}, f_2^{15}, f_0^{11}], f_3^{22}, f_3^{23},
  f_4^{16}, f_4^{17}, f_5^{9}, f_5^{10}, f_6^{5}, f_6^{6}, f_7^{1}, 
 f_6^{3}, e_6^0, \\
&\tab f_6^{-3}, f_7^{-1}, f_6^{-6}, f_6^{-5}, f_5^{-10}, 
 f_5^{-9}, f_4^{-17}, f_4^{-16}, f_3^{-23}, f_3^{-22}, [f_0^{-11}, f_2^{-15}, 
 f_0^{-11}], f_3^{-24}, f_3^{-23}, f_4^{-18}, f_4^{-17}, \\
&\tab f_5^{-11},  f_5^{-10}, f_6^{-7}, f_6^{-6}, f_7^{-2}),\\
\be_7=&X(f_7^3,e_7^0),\\
\be_0=&X(f_0^{14}, f_3^{28}, f_2^{18}, f_1^{9}, f_2^{16}, f_3^{24}, f_0^{11}, 
 f_3^{22}, f_2^{14}, f_1^{7}, f_2^{12}, f_3^{18}, f_0^{8}, f_3^{16}, 
 f_2^{10}, f_1^{5}, f_2^{8}, f_3^{12}, f_0^{5}, f_3^{10}, f_4^{6}, 
 f_5^{1}, f_4^{4}, f_3^{6}, f_0^{2}, f_3^{4}, f_4^{2}, e_0^0,\\
&\tab  
 f_4^{-2}, f_3^{-4}, f_0^{-2}, f_3^{-6}, f_4^{-4}, f_5^{-1}, f_4^{-6},
  f_3^{-10}, f_0^{-5}, f_3^{-12}, f_2^{-8}, f_1^{-5}, f_2^{-10}, 
 f_3^{-16}, f_0^{-8}, f_3^{-18}, f_2^{-12}, f_1^{-7}, f_2^{-14}, 
 f_3^{-22}, f_0^{-11},\\
&\tab  f_3^{-24}, f_2^{-16}, f_1^{-9}, f_2^{-18}, f_3^{-28}).
}
\end{tiny}
Here for the action of $\be_6$, the path corresponding to $...[A,B,A]...$ is split as:
$$...+X_{...}+ X_{..., A}+X_{..., B}+X_{..., A,B} + X_{..., A, B, ...}+...\;.$$
We see that the path for $\be_6$ is special in the sense that it revisited certain nodes twice. The same phenomenon also appear in type $F_4$ below.

Finally the initial terms are given by
\Eqn{
&X_{e_1^0}=e(-2u_1^2),&& X_{e_2^0}=e(-2u_2^2),&& X_{e_3^0}=e(-2u_3^1),&&X_{e_4^0}=e(-2u_3^2),\\
& X_{e_5^0}=e(-2u_1^4),&& X_{e_6^0}=e(-2u_6^3), && X_{e_7^0}=e(-2u_7^3), &&X_{e_0^0}=e(-2u_4^2).
}

The $E_8$-quiver is shown in Figure \ref{fig-E8}, where we have highlighted the different $E_i$-paths of the $\be_i$ generators except $\be_6$. For the special case of $\be_6$, we highlight it separately in Figure \ref{fig-E86}.
\begin{landscape}
\begin{subfigures}
\begin{figure}[!htb]
\centering
\begin{tikzpicture}[every node/.style={inner sep=0, minimum size=0.1cm, thick}, x=0.18cm, y=0.43cm]
\xdef\c{0}
\foreach \y[count=\d from 0] in {21,39,59,45,29,19,7,29}{
	\foreach \x in {1,..., \y}{
		\pgfmathtruncatemacro{\ind}{\x+\c}
		\pgfmathsetmacro{\xx}{100*(\x-1)/(\y-1)}
		\ifthenelse{\d=7}{
			\ifthenelse{\x=1 \OR \x=29}{
				\node(\ind) at (0.3*\x-0.3*15+\xx, 12.25)[draw]{};
				}
				{\ifthenelse{\x=13 \OR \x=17}{
				\node(\ind) at (0.4*\x-0.4*15+\xx, 12.25)[draw,circle]{};
				}{\ifthenelse{\x=9 \OR \x=12 \OR \x=18 \OR \x=21}
				{\node(\ind) at (0.1*\x-0.1*15+\xx, 12.25)[draw,circle]{};
				}{
				\node(\ind) at (\xx, 12.25)[draw, circle]{};
				}
				}
				}
			}
		{	\ifthenelse{\x=1 \OR \x=\y}
		{
		\node(\ind) at (\xx, 21-3.5*\d)[draw]{};
		}
		{
		\node(\ind) at (\xx, 21-3.5*\d)[draw, circle]{};
		}
		}
	}
	\xdef\c{\c+\y}
}
\foreach \y [count=\d from 249] in {22.75,19.25,15.167,16.333,24.5,1.75,-1.75,8.75}{
\node(\d) at (50, \y)[draw, circle]{};
}
\drawpath{1, 22, 61, 120, 165, 194, 213}{dashed,thin}
\drawpath{220,61}{dashed,thin}
\drawpath{219, 212, 193, 164, 119, 60, 21}{dashed,thin}
\drawpath{119,248}{dashed, thin}
\xdef\cc{1}
\foreach \y in {21,60,119,164,193,212,219,248}{
	\drawpath{\cc,...,\y}	{}
	\pgfmathtruncatemacro{\ind}{\y+1}
	\xdef\cc{\ind}
}

\foreach \y[count=\c from 1] in {10,19,29,22,14,9,3}{
\node at (-2.2, 24.5-3.5*\c){\tiny$f_\c^{-\y}$};
\node at (102, 24.5-3.5*\c){\tiny$f_\c^{\y}$};
}
\node at (-6, 12.5){\tiny$f_0^{-14}$};
\node at (106, 12.5){\tiny$f_0^{14}$};
\foreach \y[count=\c from 0] in {8.5,23,19.50,15.167,16.333,24.75,2.5,-1.25}{
\node at (51, \y){\tiny$e_\c^0$};
}

\drawpath{8,253,14}{thin}
\drawpath{10,249,12}{thin}
\drawpath{15, 48, 101, 149, 181, 203, 177, 135, 79, 34, 7}{thin}
\drawpath{13, 45, 97, 146, 179, 138, 83, 37, 9}{thin}
\drawpath{11,39,40,250,42,43,11}{thin}
\drawpath{11, 39, 40, 250, 42, 43, 11}{thin}
\drawpath{234, 88, 141, 256, 143, 92, 234}{thin}
\drawpath{142, 89, 252, 91, 142}{thin}
\drawpath{214,255,218}{thin}
\drawpath{201,254,205,206,216,200}{thin}
\drawpath{144,93,41,87,140}{thin}
\drawpath{212, 218, 209, 217, 206, 254, 200, 215, 197, 214, 194}{thin}
\drawpath{21, 59, 117, 162, 191, 210, 189, 208, 187, 156, 109, 53, 17, 51, 105, 152, 183, 204, 181, 148, 99, 46, 13, 249, 9, 36, 81, 136, 177, 202, 175, 132, 75, 31, 5, 29, 71, 128, 171, 198, 169, 196, 167, 122, 63, 23, 1}{red,thin}
\drawpath{60, 118, 247, 116, 161, 190, 159, 188, 157, 110, 243, 108, 52, 106, 241, 104, 151, 182, 149, 100, 238, 98, 45, 12, 43, 250, 39, 10, 37, 82, 230, 80, 135, 176, 133, 76, 227, 74, 30, 72, 225, 70, 127, 170, 125, 168, 123, 64, 221, 62, 22}{orange,thin}
\drawpath{119, 163, 117, 58, 115, 160, 113, 158, 111, 54, 109, 155, 107, 153, 105, 50, 103, 150, 101, 47, 99, 147, 97, 44, 95, 42, 93, 143, 91, 251, 89,141, 87, 40, 85, 38, 83, 137, 81, 35, 79, 134, 77, 32, 75, 131, 73, 129, 71, 28, 69, 126, 67, 124, 65, 24, 63, 121, 61}{brown,thin}
\drawpath{164, 192, 162, 116, 246, 114, 56, 112, 244, 110, 156, 186, 154, 184, 152, 104, 240, 102, 48, 14, 46, 98, 237, 96, 145, 94, 235, 92, 252, 88, 233, 86, 139, 84, 231, 82, 36, 8, 34, 78, 228, 76, 132, 174, 130, 172, 128, 70, 224, 68, 26, 66, 222, 64, 122, 166, 120}{green,thin}
\drawpath{193, 211, 191, 161, 115, 57, 19, 55, 111, 157, 187, 207, 185, 205, 183, 151, 103, 49, 15, 253, 7, 33, 77, 133, 175, 201, 173, 199, 171, 127, 69, 27, 3, 25, 65, 123, 167, 195, 165}{blue,thin}
\drawpath{219,255,213}{green,thin}
\drawpath{248, 118, 59, 20, 57, 114, 245, 112, 55, 18, 53, 108, 242, 106, 51,16, 49, 102, 239, 100, 148, 180, 146, 96, 236, 94, 144, 256, 140, 86, 232, 84, 138, 178, 136, 80, 229, 78, 33, 6, 31, 74, 226, 72, 29, 4, 27, 68, 223, 66, 25, 2, 23, 62, 220}{purple,thin}
\end{tikzpicture}
\caption{$E_8$-quiver, with the $E_i$-paths (except $E_6$) colored in different colors.}
\label{fig-E8}
\end{figure}
\begin{figure}[!htb]
\centering
\begin{tikzpicture}[every node/.style={inner sep=0, minimum size=0.1cm, thick}, x=0.18cm, y=0.43cm]
\xdef\c{0}
\foreach \y[count=\d from 0] in {21,39,59,45,29,19,7,29}{
	\foreach \x in {1,..., \y}{
		\pgfmathtruncatemacro{\ind}{\x+\c}
		\pgfmathsetmacro{\xx}{100*(\x-1)/(\y-1)}
		\ifthenelse{\d=7}{
			\ifthenelse{\x=1 \OR \x=29}{
				\node(\ind) at (0.3*\x-0.3*15+\xx, 12.25)[draw]{};
				}
				{\ifthenelse{\x=13 \OR \x=17}{
				\node(\ind) at (0.4*\x-0.4*15+\xx, 12.25)[draw,circle]{};
				}{\ifthenelse{\x=9 \OR \x=12 \OR \x=18 \OR \x=21}
				{\node(\ind) at (0.1*\x-0.1*15+\xx, 12.25)[draw,circle]{};
				}{
				\node(\ind) at (\xx, 12.25)[draw, circle]{};
				}
				}
				}
			}
		{	\ifthenelse{\x=1 \OR \x=\y}
		{
		\node(\ind) at (\xx, 21-3.5*\d)[draw]{};
		}
		{
		\node(\ind) at (\xx, 21-3.5*\d)[draw, circle]{};
		}
		}
	}
	\xdef\c{\c+\y}
}
\foreach \y [count=\d from 249] in {22.75,19.25,15.167,16.333,24.5,1.75,-1.75,8.75}{
\node(\d) at (50, \y)[draw, circle]{};
}
\drawpath{1, 22, 61, 120, 165, 194, 213}{dashed,thin}
\drawpath{220,61}{dashed,thin}
\drawpath{219, 212, 193, 164, 119, 60, 21}{dashed,thin}
\drawpath{119,248}{dashed, thin}
\xdef\cc{1}
\foreach \y in {21,60,119,164,193,212,219,248}{
	\drawpath{\cc,...,\y}	{}
	\pgfmathtruncatemacro{\ind}{\y+1}
	\xdef\cc{\ind}
}

\foreach \y[count=\c from 1] in {10,19,29,22,14,9,3}{
\node at (-2.2, 24.5-3.5*\c){\tiny$f_\c^{-\y}$};
\node at (102, 24.5-3.5*\c){\tiny$f_\c^{\y}$};
}
\node at (-6, 12.5){\tiny$f_0^{-14}$};
\node at (106, 12.5){\tiny$f_0^{14}$};
\foreach \y[count=\c from 0] in {8.5,23,19.50,15.167,16.333,24.75,2.5,-1.25}{
\node at (51, \y){\tiny$e_\c^0$};
}

\drawpath{8,253,14}{thin}
\drawpath{10,249,12}{thin}
\drawpath{15, 48, 101, 149, 181, 203, 177, 135, 79, 34, 7}{thin}
\drawpath{13, 45, 97, 146, 179, 138, 83, 37, 9}{thin}
\drawpath{11,39,40,250,42,43,11}{thin}
\drawpath{11, 39, 40, 250, 42, 43, 11}{thin}
\drawpath{234, 88, 141, 256, 143, 92, 234}{thin}
\drawpath{142, 89, 252, 91, 142}{thin}
\drawpath{214,255,218}{thin}
\drawpath{201,254,205,206,216,200}{thin}
\drawpath{144,93,41,87,140}{thin}
\drawpath{212, 218, 209, 217, 206, 254, 200, 215, 197, 214, 194}{thin}
\drawpath{21, 59, 117, 162, 191, 210, 189, 208, 187, 156, 109, 53, 17, 51, 105, 152, 183, 204, 181, 148, 99, 46, 13, 249, 9, 36, 81, 136, 177, 202, 175, 132, 75, 31, 5, 29, 71, 128, 171, 198, 169, 196, 167, 122, 63, 23, 1}{thin}
\drawpath{60, 118, 247, 116, 161, 190, 159, 188, 157, 110, 243, 108, 52, 106, 241, 104, 151, 182, 149, 100, 238, 98, 45, 12, 43, 250, 39, 10, 37, 82, 230, 80, 135, 176, 133, 76, 227, 74, 30, 72, 225, 70, 127, 170, 125, 168, 123, 64, 221, 62, 22}{thin}
\drawpath{119, 163, 117, 58, 115, 160, 113, 158, 111, 54, 109, 155, 107, 153, 105, 50, 103, 150, 101, 47, 99, 147, 97, 44, 95, 42, 93, 143, 91, 251, 89,141, 87, 40, 85, 38, 83, 137, 81, 35, 79, 134, 77, 32, 75, 131, 73, 129, 71, 28, 69, 126, 67, 124, 65, 24, 63, 121, 61}{thin}
\drawpath{164, 192, 162, 116, 246, 114, 56, 112, 244, 110, 156, 186, 154, 184, 152, 104, 240, 102, 48, 14, 46, 98, 237, 96, 145, 94, 235, 92, 252, 88, 233, 86, 139, 84, 231, 82, 36, 8, 34, 78, 228, 76, 132, 174, 130, 172, 128, 70, 224, 68, 26, 66, 222, 64, 122, 166, 120}{thin}
\drawpath{193, 211, 191, 161, 115, 57, 19, 55, 111, 157, 187, 207, 185, 205, 183, 151, 103, 49, 15, 253, 7, 33, 77, 133, 175, 201, 173, 199, 171, 127, 69, 27, 3, 25, 65, 123, 167, 195, 165}{thin}
\drawpath{219,255,213}{thin}
\drawpath{248, 118, 59, 20, 57, 114, 245, 112, 55, 18, 53, 108, 242, 106, 51,16, 49, 102, 239, 100, 148, 180, 146, 96, 236, 94, 144, 256, 140, 86, 232, 84, 138, 178, 136, 80, 229, 78, 33, 6, 31, 74, 226, 72, 29, 4, 27, 68, 223, 66, 25, 2, 23, 62, 220}{thin}
\drawpath{212, 218, 209, 210, 189, 190, 159, 160, 113, 114, 245, 112, 113, 158, 159, 188, 189, 208, 209, 217, 206, 254, 200, 215, 197, 198, 169, 170, 125, 126, 67, 68, 223, 66, 67, 124, 125, 168, 169, 196, 197, 214, 194}{red,thick}
\drawpath{114,56,112}{red,thick}
\drawpath{68,26,66}{red,thick}
\end{tikzpicture}
\caption{$E_8$-quiver, with the $E_6$-paths colored in red.}
\label{fig-E86}
\end{figure}
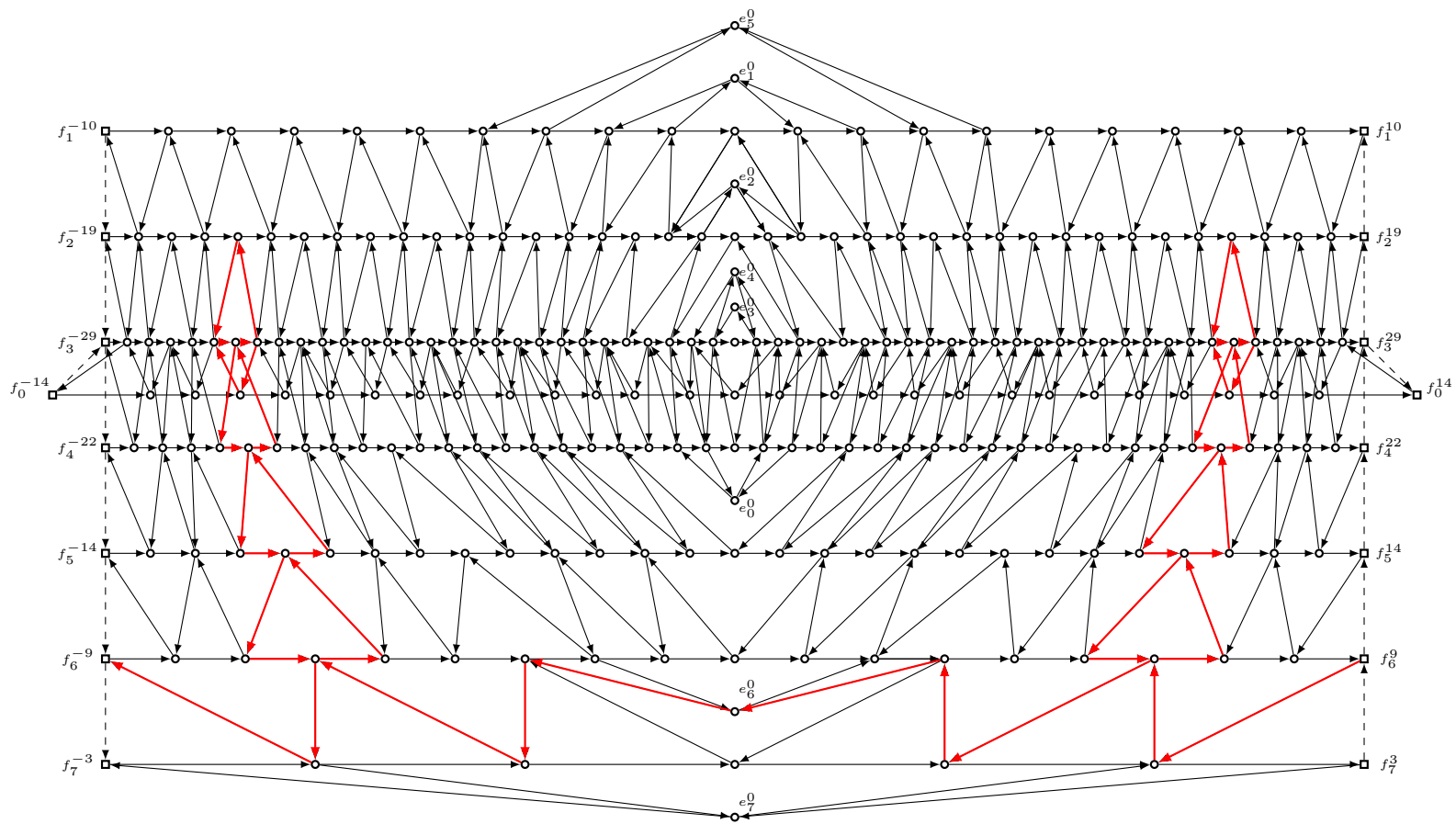
\end{subfigures}
\end{landscape}
%==============================================================================
\subsection{Type $F_4$}\label{sec:embedding:F4}
The explicit expression for type $F_4$ positive representations can be found in \cite{Ip3}, where we choose
$$\bi=(3232\;12321\;432312343213234),$$
where 1,2 is long, 3,4 is short, corresponding to the embedding of the Dynkin diagram:
$$B_2\subset B_3\subset F_4.$$

Then the $\bf_i$ variables are expressed as
\Eqn{
\bf_1&=X(f_1^{-4},..., f_1^3),\\
\bf_2&=X(f_2^{-8},..., f_2^7),\\
\bf_3&=X(f_3^{-9},..., f_3^8),\\
\bf_4&=X(f_4^{-3},..., f_4^2),
}
while the $\be_i$ variables are expressed as certain paths on the quiver:
\Eqn{
\be_1=&X(f_1^{4}, f_2^{7}, *f_3^{7}, f_2^{6}, *f_3^{5}, f_2^{5}, f_1^{2}, e_1^0, 
 f_1^{-2}, f_2^{-5}, *f_3^{-5}, f_2^{-6}, *f_3^{-7}, f_2^{-7}),\\
\be_2=&X(f_2^{8}, *f_3^{8}, f_2^{7}, f_1^{3}, f_2^{5}, *f_3^{4}, f_2^{4}, 
 f_1^{1}, f_2^{2}, e_2^0, f_2^{-2}, f_1^{-1}, f_2^{-4}, *f_3^{-4}, 
 f_2^{-5}, f_1^{-3}, f_2^{-7}, *f_3^{-8}),\\
\be_3=&X(f_3^{9}, f_4^{2}, f_3^{6}, f_3^{7}, f_2^{6}, f_3^{5}, f_3^{6}, 
 f_4^{1}, f_3^{3}, f_2^{3}, f_3^{2}, f_2^{1}, f_3^{1}, e_3^0, f_3^{-1},
  f_2^{-1}, f_3^{-2}, f_2^{-3}, f_3^{-3}, f_4^{-1}, f_3^{-6}, 
 f_3^{-5}, f_2^{-6},\\
 &\tab f_3^{-7}, f_3^{-6}, f_4^{-2}),\\
\be_4=&X(f_4^3, e_4^0),
}
where we recall from type $C_n$ that $X(...,a, *b, ...)$ corresponds to the extra factors as follows:
\Eqn{
&...+X_{...}+X_{...,a}+[2]_{q_s} X_{...,a,b}+X_{...,a,b^2}+X_{...,a,b^2,...}+...\\
&=...+X_{...}+(X_{...,a}^{\half}+X_{...,a,b^2}^{\half})^2+X_{...,a,b^2,...}+...\;.
}

The initial terms are given by
\Eqn{
&X_{e_1^0}=e(-2u_1^2),&&X_{e_2^0}=e(-2u_2^2),&& X_{e_3^0}=e(-2u_3^1),&& X_{e_4^0}=e(-2u_4^3).
}
The quiver is shown in Figure \ref{fig-F4}, where the repeated nodes $*$ are highlighted. We note that the $E_1$ and $E_3$ paths overlapped a little bit.

\begin{landscape}
\begin{figure}[!htb]
\centering
\begin{tikzpicture}[every node/.style={inner sep=0, minimum size=0.5cm, thick}, x=0.18cm, y=0.47cm]
\xdef\c{0}
\foreach \y[count=\d from 1] in {9,17,19,7}{
	\foreach \x in {1,..., \y}{
		\pgfmathtruncatemacro{\ex}{\x-(\y+1)/2}
		\pgfmathtruncatemacro{\ind}{\x+\c}
		\pgfmathsetmacro{\xx}{100*(\x-1)/(\y-1)}
		\ifthenelse{\x=1 \OR \x=\y}
		{
		\node(\ind) at (\xx, 20-5*\d)[draw]{\tiny $f_\d^\ex$};
		}
		{
		\ifthenelse{\ind=29 \OR \ind=31 \OR \ind=41 \OR \ind=43}{
			\node(\ind) at (\xx, 20-5*\d)[draw, circle, red]{\tiny $f_\d^{\ex}$};
			}{
			\ifthenelse{\ind=28 \OR \ind=32 \OR \ind=40 \OR \ind=44}{
				\node(\ind) at (\xx, 20-5*\d)[draw, circle, orange]{\tiny $f_\d^{\ex}$};
			}{
				\node(\ind) at (\xx, 20-5*\d)[draw, circle]{\tiny $f_\d^{\ex}$};
			}
		}
	}
	}
	\xdef\c{\c+\y}
}
\foreach \y [count=\d from 1] in {17, 12.5, 2.5, -2}{
\pgfmathtruncatemacro{\ind}{\d+52}
\node(\ind) at (50, \y)[draw, circle]{\tiny $e_\d^0$};
}
\drawpath{1,10,27}{dashed, vthick}
\drawpath{27,46}{dashed,thin}
\drawpath{45,26,9}{dashed,vthick}
\drawpath{52,45}{dashed, thin}
\drawpath{1,...,9}{vthick}
\drawpath{10,...,26}{vthick}
\drawpath{27,...,45}{}
\drawpath{46,...,52}{}

\drawpath{33,14,3,4,53,6,7,22,39}{vthick}
\drawpath{5,16,17,54,19,20,5}{vthick}
\drawpath{37,18,35}{vthick}
\drawpath{39,49,33}{thin}
\drawpath{47,56,51}{thin}
\drawpath{9, 25, 43}{red, vthick}
\drawpath{41, 23, 7, 53, 3, 13, 31}{red,vthick}
\drawpath{29,11,1}{red, vthick}
\drawpath{26, 44, 25, 8, 23, 40, 22, 6, 20, 54, 16, 4, 14, 32, 13, 2, 11, 28, 10}{orange, vthick}
\drawpath{45, 51, 42, 43}{green, thin}
\drawpath{41, 42, 50, 39}{green, thin}
\drawpath{37,55,35}{green, thin}
\drawpath{33, 48, 30, 31}{green,thin}
\drawpath{29, 30, 47, 27}{green,thin}
\path ([xshift=-0.03cm]43.north)  edge[->, green, vthick]  ([xshift=-0.03cm]24.south);
\path ([xshift=0.03cm]24.south west)  edge[->, green, vthick]  ([xshift=0.03cm]41.north east);
\path ([xshift=-0.03cm]31.north west)  edge[->, green, vthick]  ([xshift=-0.03cm]12.south east);
\path ([xshift=0.03cm]12.south)  edge[->, green, vthick]  ([xshift=0.03cm]29.north);
\drawpath{39,21,38,19,37}{green, vthick}
\drawpath{35,17,34,15,33}{green, vthick}
\drawpath{52, 56, 46}{blue,thin}
\path ([xshift=0.03cm]43.north)  edge[->, red, vthick]  ([xshift=0.03cm]24.south);
\path ([xshift=-0.03cm]24.south west)  edge[->, red, vthick]  ([xshift=-0.03cm]41.north east);
\path ([xshift=0.03cm]31.north west)  edge[->, red, vthick]  ([xshift=0.03cm]12.south east);
\path ([xshift=-0.03cm]12.south)  edge[->, red, vthick]  ([xshift=-0.03cm]29.north);

\end{tikzpicture}
\caption{$F_4$-quiver, with the $E_i$-paths colored in different colors.}
\label{fig-F4}
\end{figure}
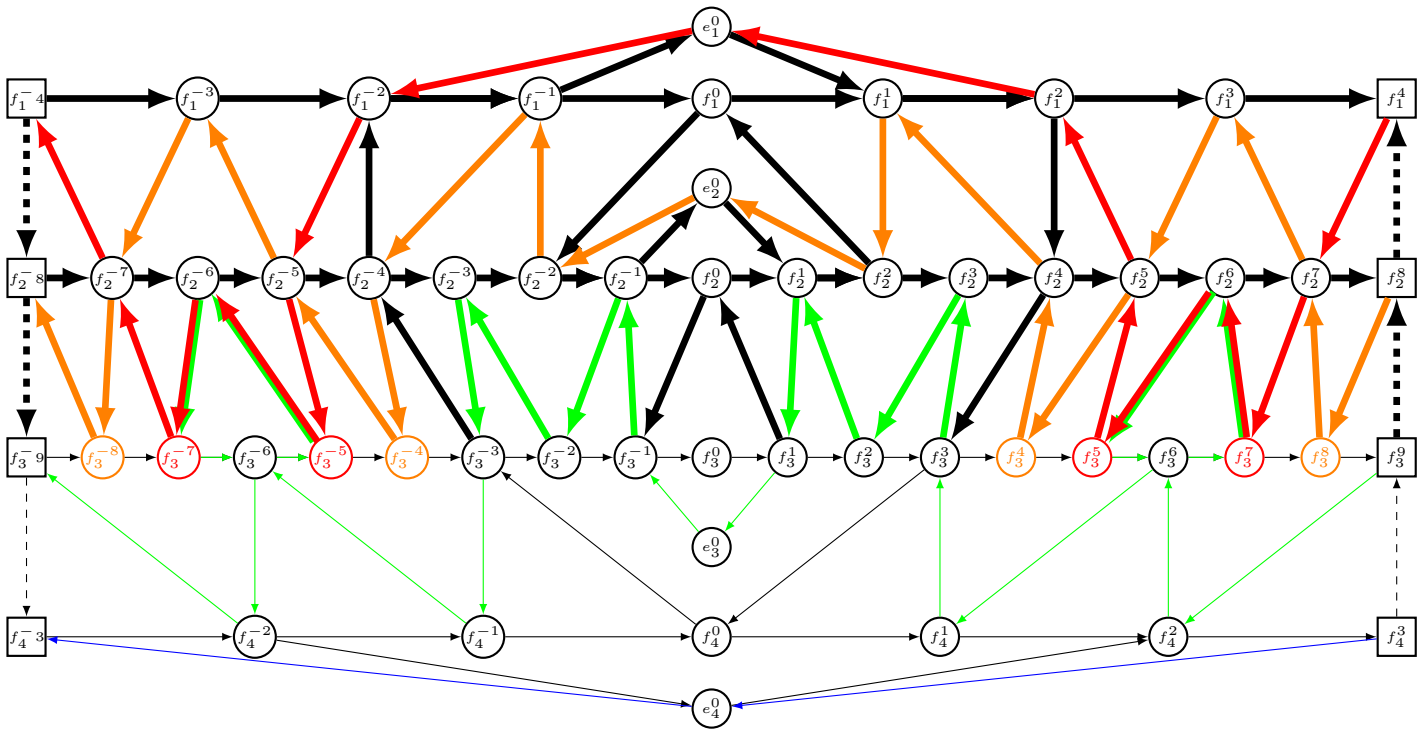
\end{landscape}

%==============================================================================
\subsection{Type $G_2$}\label{sec:embedding:G2}
The explicit expression for type $G_2$ positive representations can be found in \cite{Ip3}. We choose $\bi=(2,1,2,1,2,1)$. Then we have
\Eqn{
\bf_1&=X(f_1^{-3},..., f_1^2),\\
\bf_2&=X(f_2^{-3},...,f_2^2),\\
\be_1&=X(f_1^3, e_1^0),\\
\be_2&=X(f_2^3, f_1^2, *f_2^2, f_1^1,f_2^1,e_2^0, f_2^{-1},f_1^{-1},*f_2^{-2}, f_1^{-2}),
}
where again $X(...,a, *b, ...)$ corresponds to the extra factors:
\Eqn{
&...+X_{...}+X_{...,a}+[2]_{q_s} X_{...,a,b}+X_{...,a,b^2}+X_{...,a,b^2,...}+...\;.
}

The inital terms are given by
$$X_{e_1^0}=e(-2u_1^3), \tab X_{e_2^0}=e(-2u_2^1).$$

The quiver is shown in Figure \ref{fig-G2}.
\begin{figure}[!htb]
\centering
\begin{tikzpicture}[every node/.style={inner sep=0, minimum size=0.5cm, thick}, x=0.10cm, y=0.47cm]
\foreach \d in {1,2}{
	\foreach \x in {1,..., 7}{
		\pgfmathtruncatemacro{\ex}{\x-4}
		\pgfmathtruncatemacro{\ind}{\x+(\d-1)*7}
		\pgfmathsetmacro{\xx}{20*(\x-1)}
		\ifthenelse{\x=1 \OR \x=7}
		{
		\node(\ind) at (\xx, 15-5*\d)[draw]{\tiny $f_\d^\ex$};
		}
		{
		\ifthenelse{\ind=13 \OR \ind=9}{
			\node(\ind) at (\xx, 15-5*\d)[draw, circle, red]{\tiny $f_\d^{\ex}$};
			}{
				\node(\ind) at (\xx, 15-5*\d)[draw, circle]{\tiny $f_\d^{\ex}$};
			}
		}
	}
}
\node(15) at (60, 15)[draw,circle]{\tiny $e_1^0$};
\node(16) at (60, 0)[draw,circle]{\tiny $e_2^0$};
\drawpath{8,1}{vthick, dashed}
\drawpath{7,14}{vthick, dashed}
\drawpath{1,...,7}{vthick}
\drawpath{8,...,14}{thin}
\drawpath{2,15,6}{vthick}
\drawpath{12,4,10}{vthick}
\drawpath{7,15,1}{vthick, red}
\drawpath{14,6,13,5,12}{vthick, red}
\drawpath{12,16,10}{thin, red}
\drawpath{10,3,9,2,8}{vthick, red}
\end{tikzpicture}
\caption{$G_2$-quiver, with the $E_i$-paths colored in red.}
\label{fig-G2}
\end{figure}
%==============================================================================
\section{Quiver mutations for different choice of $w_0$}\label{sec:mutation}
Recall from the construction of the positive representations that a change of reduced expression of $w_0$ corresponds to a unitary transformation $\Phi$ (cf. \eqref{unitrans}). This is expressed in terms of conjugation by quantum dilogarithms, followed by a linear transformation. As we have seen in Section \ref{sec:tori}, conjugation by the quantum dilogarithms naturally correspond to mutations of the quiver diagram. In this section we will describe the corresponding mutation associated to a change of words. In particular, by extending the mutations below to the full quiver, we obtain an alternate proof of Theorem \ref{initial} for the rules of finding the initial term $X_{e_j^0}$ of the generators $\be_j$.
%==============================================================================

\subsection{Quiver mutation in simply-laced case}\label{sec:mutation:simply}
First we note that if $a_{ij}=0$, i.e. $s_is_j=s_js_i$, there is no mutation or change of variables occurring. That is, swapping the reflections does not affect the quiver diagram at all.

In the simply-laced case, the unitary transformation $\Phi$ corresponding to the change of words
$$w_0=...s_is_js_i... \corr ... s_j s_i s_j ...$$
is expressed in terms of conjugation by a single quantum dilogarithm. 

Consider the following amalgamation $Q$ of elementary quivers corresponding to $s_is_js_i$, where we exclude the nodes outside the root indices $i$ and $j$:
\begin{figure}[H]
\centering
\begin{tikzpicture}[every node/.style={inner sep=0, minimum size=0.8cm, thick}, x=0.5cm, y=0.5cm]
\node (1) at (0,3) [draw, circle] {$f_i^{k-1}$};
\node (2) at (6,3) [draw, circle] {$f_i^{k}$};
\node (3) at (12,3) [draw, circle] {$f_i^{k+1}$};
\node (4) at (3, 0) [draw, circle]{$f_j^{l-1}$};
\node (5) at (9,0) [draw,circle]{$f_j^l$};
\path (1) edge[->, thick] (2);
\node at (3,3.3) {$s_i$};
\path (2) edge[->, thick] (3);
\node at (9,3.3) {$s_i$};
\path (2) edge[->, thick] (4);
\path (5) edge[->, thick] (2);
\path (4) edge[->, thick] (5);
\node at (6,0.3) {$s_j$};
\path (4) edge[->, thick, dashed] (1);
\path (3) edge[->, thick, dashed] (5);
\end{tikzpicture}
\caption{The $s_is_js_i$ quiver.}
\end{figure}
This corresponds to the representation of the $\bf_i$ generators in the full quiver $\cD_\g$ as
\Eqn{
\bf_i &= .... + X_{...f_i^{k-1}}+ X_{... f_i^{k-1}, f_i^k} + X_{... f_i^{k-1}, f_i^k, f_i^{k+1}} + ...\\
\bf_j &= .... + X_{...f_j^{l-1}}+ X_{... f_j^{l-1}, f_j^l} +  ...
}
Then the mutation corresponding to the unitary transformation $\Phi$ giving the change of words $s_is_js_i\corr s_js_is_j$ is given by mutation at $f_i^k$, followed by a renaming of variables, where we have defined a new external labeling for the mutated quiver $\what{Q}$ by the rules:
\Eqn{
\hat{f}_i^t&:= f_i^{t+1}\tab t\geq k,\\
\hat{f}_j^t&:= f_j^{t-1} \tab t\geq l+1,\\
\hat{f}_j^l&:=f_i^k
}
and stays the same otherwise.
\begin{subfigures}
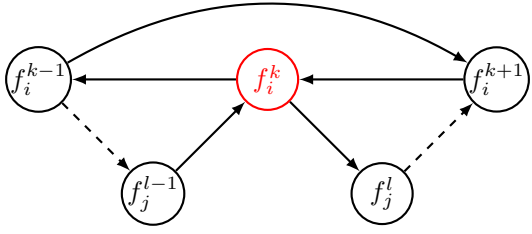
\begin{figure}[!htb]
\centering
\begin{tikzpicture}[every node/.style={inner sep=0, minimum size=0.8cm, thick}, x=0.5cm, y=0.5cm]
\node (1) at (0,3) [draw, circle] {$f_i^{k-1}$};
\node (2) at (6,3) [draw, circle, red, thick] {$f_i^{k}$};
\node (3) at (12,3) [draw, circle] {$f_i^{k+1}$};
\node (4) at (3, 0) [draw, circle]{$f_j^{l-1}$};
\node (5) at (9,0) [draw,circle]{$f_j^l$};
\path (2) edge[->, thick] (1);
\path (3) edge[->, thick] (2);
\path (4) edge[->, thick] (2);
\path (2) edge[->, thick] (5);
\path (1) edge[->, thick, dashed] (4);
\path (5) edge[->, thick, dashed] (3);
\path (1) edge[->, bend left, thick] (3);
\end{tikzpicture}
\caption{After mutation at $f_i^k$.}
\end{figure}
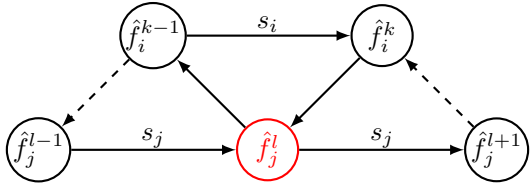
\begin{figure}[!htb]
\centering
\begin{tikzpicture}[every node/.style={inner sep=0, minimum size=0.8cm, thick}, x=0.5cm, y=0.5cm]
\node (1) at (0,0) [draw, circle] {$\hat{f}_{j}^{l-1}$};
\node (2) at (6,0) [draw, circle, red, thick] {$\hat{f}_{j}^{l}$};
\node (3) at (12,0) [draw, circle] {$\hat{f}_{j}^{l+1}$};
\node (4) at (3, 3) [draw, circle]{$\hat{f}_{i}^{k-1}$};
\node (5) at (9,3) [draw,circle]{$\hat{f}_{i}^{k}$};
\path (1) edge[->, thick] (2);
\node at (3,0.3) {$s_j$};
\path (2) edge[->, thick] (3);
\node at (9,0.3) {$s_j$};
\path (2) edge[->, thick] (4);
\path (5) edge[->, thick] (2);
\path (4) edge[->, thick] (5);
\node at (6,3.3) {$s_i$};
\path (4) edge[->, thick, dashed] (1);
\path (3) edge[->, thick, dashed] (5);
\end{tikzpicture}
\caption{Rearranging and renaming the quiver.}
\end{figure}
\end{subfigures}

In the representation level, a change of words corresponds to a unitary transformation $\Phi$ by the conjugation of the quantum dilogarithm $g_b(X_{f_i^k})$:
\Eqn{
Ad_{g_b(X_{f_i^k})}\cdot \bf_i &= ... + X_{...f_i^{k-1}}+  X_{... f_i^{k-1}, f_i^k, f_i^{k+1}} + ...\\
&=\mu_{f_i^k}'(... +X_{...f_{i}^{k-1}}+ X_{... f_{i}^{k-1}, f_{i}^{k+1}} + ...)\\
&=\mu_{f_i^k}'(... +X_{...\hat{f}_{i}^{k-1}}+  X_{... \hat{f}_{i}^{k-1}, \hat{f}_i^k} + ...),\\
Ad_{g_b(X_{f_i^k})}\cdot \bf_j &= ... + X_{...f_j^{l-1}, f_i^k}+ X_{... f_j^{l-1}} + X_{... f_j^{l-1}, f_j^l} +  ...\\
&=\mu_{f_i^k}'(...+X_{...f_j^{l-1}} + X_{...f_j^{l-1}, f_i^k}+X_{... f_j^{l-1}, f_i^k, f_j^l} +...)\\
&=\mu_{f_i^k}'(...+X_{...\hat{f}_j^{l-1}} +X_{...\hat{f}_j^{l-1}, \hat{f}_j^l}+X_{... \hat{f}_j^{l-1}, \hat{f}_j^l, \hat{f}_j^{l+1}} +...).
}
Hence using $\mu_k^q=Ad_{g_b^*(X_j)}\circ \mu_k'$, we have
\Eqn{
\bf_i &= ...+\what{X}_{...\hat{f}_{i}^{k-1}}+ \what{X}_{... \hat{f}_{i}^{k-1}, \hat{f}_i^k}+...,\\
\bf_j &= ...+\what{X}_{...\hat{f}_j^{l-1}} + \what{X}_{...\hat{f}_j^{l-1}, \hat{f}_j^l}+\what{X}_{... \hat{f}_j^{l-1}, \hat{f}_j^l, \hat{f}_j^{l+1}} +...,
}
where we denote the mutated cluster variables by $\what{X}_j:=\mu_{f_i^k}^q(X_j)$ associated to the mutated quiver $\what{Q}$, and we see that the representation of the $\bf_i$ generators are invariant under the quiver mutation.

When we take into account the whole quiver $\cD_\g$, we see that the nodes precisely come in pair. Hence we have
\begin{Cor} The cluster embedding $\iota: \fD_\g\to \cD_\g$ corresponding to $\bi=(...iji...)$ and $\bi'=(...jij...)$ is related by quiver mutations at the pair of nodes $\{f_i^k, f_i^{-k}\}$ (the order does not matter).
\end{Cor}
%==============================================================================

\subsection{Quiver mutation in doubly-laced case}\label{sec:mutation:doubly}
Following the notation above, we consider the following amalgamation of quiver corresponding to $s_is_js_is_j$ where the root $i$ is long and $j$ is short. All the arrows are thick except the two corresponding to $s_j$.
\begin{figure}[H]
\centering
\begin{tikzpicture}[every node/.style={inner sep=0, minimum size=0.8cm, thick}, x=0.5cm, y=0.5cm]
\node (1) at (0,3) [draw, circle] {$f_i^{k-1}$};
\node (2) at (6,3) [draw, circle] {$f_i^{k}$};
\node (3) at (12,3) [draw, circle] {$f_i^{k+1}$};
\node (4) at (3, 0) [draw, circle]{$f_j^{l-1}$};
\node (5) at (9,0) [draw,circle]{$f_j^l$};
\node (6) at (15,0) [draw,circle]{$f_j^{l+1}$};
\path (1) edge[->, vthick] (2);
\path (2) edge[->, vthick] (3);
\node at (3,3.3) {$s_i$};
\node at (9,3.3) {$s_i$};
\path (4) edge[->, thin] (5);
\path (5) edge[->, thin] (6);
\node at (6,0.3) {$s_j$};
\node at (12,0.3) {$s_j$};
\path (2) edge[->, vthick] (4);
\path (5) edge[->, vthick] (2);
\path (3) edge[->, vthick] (5);
\path (4) edge[->, vthick, dashed] (1);
\path (6) edge[->, vthick, dashed] (3);
\end{tikzpicture}
\caption{The $s_is_js_is_j$ quiver.}
\end{figure}

The unitary transformation $\Phi$ of the positive representations corresponding to the change of words $$s_is_js_is_j \corr s_js_is_js_i$$ is expressed as 3 pairs of quantum dilogarithm transformations \cite{Ip2}. The mutation corresponding to $\Phi$ is then given by mutation at $f_j^l, f_i^k, f_j^l$, with the weights $d_i$ of each nodes taken into account.
\begin{subfigures}
\begin{figure}[!htb]
\centering
\begin{tikzpicture}[every node/.style={inner sep=0, minimum size=0.8cm, thick}, x=0.5cm, y=0.5cm]
\node (1) at (0,3) [draw, circle] {$f_i^{k-1}$};
\node (2) at (6,3) [draw, circle] {$f_i^{k}$};
\node (3) at (12,3) [draw, circle] {$f_i^{k+1}$};
\node (4) at (3, 0) [draw, circle]{$f_j^{l-1}$};
\node (5) at (9,0) [draw,circle, red, thick]{$f_j^l$};
\node (6) at (15,0) [draw,circle]{$f_j^{l+1}$};
\path (1) edge[->, vthick] (2);
\path (5) edge[->, thin] (4);
\path (6) edge[->, thin] (5);
\path (2) edge[->, vthick] (5);
\path (5) edge[->, vthick] (3);
\path (3) edge[->, vthick] (2);
\path (4) edge[->, vthick, dashed] (1);
\path (3) edge[->, vthick, dashed] (6);
\path (4) edge[->, thin, bend right] (6);
\end{tikzpicture}
\caption{After mutation at $f_j^l$.}
\end{figure}
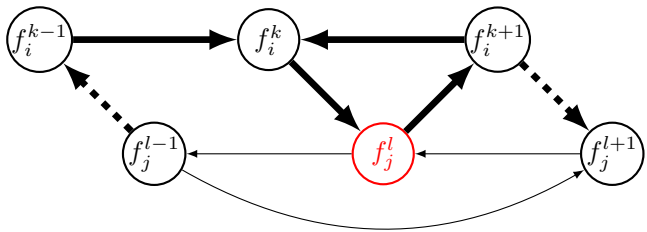
\begin{figure}[H]
\centering
\begin{tikzpicture}[every node/.style={inner sep=0, minimum size=0.8cm, thick}, x=0.5cm, y=0.5cm]
\node (1) at (0,3) [draw, circle] {$f_i^{k-1}$};
\node (2) at (6,3) [draw, circle, red, vthick] {$f_i^{k}$};
\node (3) at (12,3) [draw, circle] {$f_i^{k+1}$};
\node (4) at (3, 0) [draw, circle]{$f_j^{l-1}$};
\node (5) at (9,0) [draw,circle]{$f_j^l$};
\node (6) at (15,0) [draw,circle]{$f_j^{l+1}$};
\path (2) edge[->, vthick] (1);
\path (5) edge[->, thin] (4);
\path (6) edge[->, thin] (5);
\path (5) edge[->, vthick] (2);
\path (1) edge[->, vthick] (5);
\path (2) edge[->, vthick] (3);
\path (4) edge[->, vthick, dashed] (1);
\path (3) edge[->, vthick, dashed] (6);
\path (4) edge[->, thin, bend right] (6);
\end{tikzpicture}
\caption{After mutation at $f_i^k$.}
\end{figure}
\begin{figure}[H]
\centering
\begin{tikzpicture}[every node/.style={inner sep=0, minimum size=0.8cm, thick}, x=0.5cm, y=0.5cm]
\node (1) at (3,3) [draw, circle] {$f_i^{k-1}$};
\node (2) at (9,3) [draw, circle] {$f_i^{k}$};
\node (3) at (15,3) [draw, circle] {$f_i^{k+1}$};
\node (4) at (0, 0) [draw, circle]{$f_j^{l-1}$};
\node (5) at (6,0) [draw,circle, red, thick]{$f_j^l$};
\node (6) at (12,0) [draw,circle]{$f_j^{l+1}$};
\path (1) edge[->, vthick] (2);
\path (4) edge[->, thin] (5);
\path (5) edge[->, thin] (6);
\path (6) edge[->, vthick] (2);
\path (2) edge[->, vthick] (5);
\path (5) edge[->, vthick] (1);
\path (2) edge[->, vthick] (3);
\path (1) edge[->, vthick, dashed] (4);
\path (3) edge[->, vthick, dashed] (6);
\node at (6,3.3) {$s_i$};
\node at (12,3.3) {$s_i$};
\node at (3,0.3) {$s_j$};
\node at (9,0.3) {$s_j$};
\end{tikzpicture}
\caption{After second mutation at $f_j^l$.}
\end{figure}
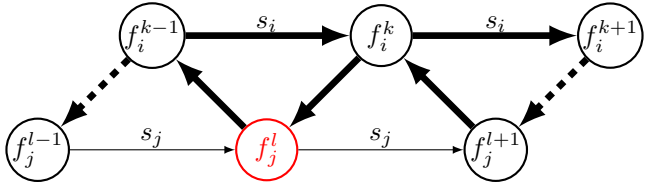
\end{subfigures}
No renaming of the variables is necessary after the last step, and again we have expressed the generators $\bf_i$ in terms of the mutated cluster variables $\what{X}_j$ associated to the mutated quiver. Similarly as before, for the full quiver we have
\begin{Cor} The cluster embedding $\iota: \fD_\g\to \cD_\g$ corresponding to $\bi=(...ijij...)$ and $\bi'=(...jiji...)$ is related by quiver mutations at the pair of nodes $\{f_j^l,f_j^{-l}\}$, $\{f_i^k,f_i^{-k}\}$ and $\{f_j^l,f_j^{-l}\}$.
\end{Cor}
%==============================================================================
\subsection{Quiver mutation in type $G_2$}\label{sec:mutation:G2}
We consider the following amalgamation of quiver corresponding to $s_2s_1s_2s_1s_2s_1$ where the root $1$ is long and $2$ is short. All the arrows are thick except the three corresponding to $s_2$.

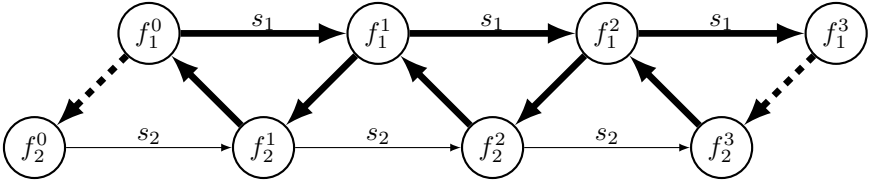
\begin{figure}[H]
\centering
\begin{tikzpicture}[every node/.style={inner sep=0, minimum size=0.8cm, thick}, x=0.5cm, y=0.5cm]
\node (1) at (0,3) [draw, circle] {$f_1^{0}$};
\node (2) at (6,3) [draw, circle] {$f_1^{1}$};
\node (3) at (12,3) [draw, circle] {$f_1^{2}$};
\node (4) at (18, 3) [draw, circle]{$f_1^{3}$};
\node (5) at (-3,0) [draw, circle] {$f_2^{0}$};
\node (6) at (3,0) [draw, circle] {$f_2^{1}$};
\node (7) at (9,0) [draw, circle] {$f_2^{2}$};
\node (8) at (15, 0) [draw, circle]{$f_2^{3}$};
\drawpath{1,...,4}{vthick}
\drawpath{5,...,8}{thin}
\drawpath{8,3,7,2,6,1}{vthick}
\path (1) edge[->, vthick, dashed] (5);
\path (4) edge[->, vthick, dashed] (8);
\node at (3,3.3) {$s_1$};
\node at (9,3.3) {$s_1$};
\node at (15,3.3) {$s_1$};
\node at (0,0.3) {$s_2$};
\node at (6,0.3) {$s_2$};
\node at (12,0.3) {$s_2$};
\end{tikzpicture}
\caption{The $s_2s_1s_2s_1s_2s_1$ quiver.}
\end{figure}

In \cite{Ip2}, we found that the unitary transformation $\Phi$ changing the words
$$s_2s_1s_2s_1s_2s_1 \corr s_1s_2s_1s_2s_1s_2$$
is given by conjugations by 11 quantum dilogarithms. This corresponds to the following sequence of mutations (starting from the left):
$$f_1^2, f_1^1,f_2^2,f_1^2, f_2^2,f_2^1,f_2^2, f_1^2, f_1^1, f_2^2, f_1^2$$

\begin{subfigures}
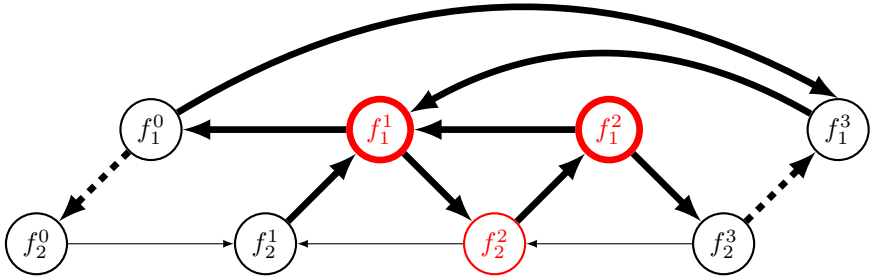
\begin{figure}[!htb]
\centering
\begin{tikzpicture}[every node/.style={inner sep=0, minimum size=0.8cm, thick}, x=0.5cm, y=0.5cm]
\node (1) at (0,3) [draw, circle] {$f_1^{0}$};
\node (2) at (6,3) [draw, circle, red, vthick] {$f_1^{1}$};
\node (3) at (12,3) [draw, circle, red, vthick] {$f_1^{2}$};
\node (4) at (18, 3) [draw, circle]{$f_1^{3}$};
\node (5) at (-3,0) [draw, circle] {$f_2^{0}$};
\node (6) at (3,0) [draw, circle] {$f_2^{1}$};
\node (7) at (9,0) [draw, circle, red, thick] {$f_2^{2}$};
\node (8) at (15, 0) [draw, circle]{$f_2^{3}$};
\path (1) edge[->, vthick, bend left] (4.north);
\path (4) edge[->, vthick, bend right] (2);
\drawpath{6,2,7,3,2,1}{vthick}
\drawpath{3,8}{vthick}
\drawpath{8,7,6}{thin}
\drawpath{5,6}{thin}
\drawpath{1,5}{vthick, dashed}
\drawpath{8,4}{vthick, dashed}
\end{tikzpicture}
\caption{After mutation at $f_1^2, f_1^1,f_2^2,f_1^2$.}
\end{figure}

\begin{figure}[H]
\centering
\begin{tikzpicture}[every node/.style={inner sep=0, minimum size=0.8cm, thick}, x=0.5cm, y=0.5cm]
\node (1) at (0,3) [draw, circle] {$f_1^{0}$};
\node (2) at (6,3) [draw, circle] {$f_1^{1}$};
\node (3) at (12,3) [draw, circle] {$f_1^{2}$};
\node (4) at (18, 3) [draw, circle]{$f_1^{3}$};
\node (5) at (-3,0) [draw, circle] {$f_2^{0}$};
\node (6) at (3,0) [draw, circle,red, thick] {$f_2^{1}$};
\node (7) at (9,0) [draw, circle,red, thick] {$f_2^{2}$};
\node (8) at (15, 0) [draw, circle]{$f_2^{3}$};
\path (1) edge[->, vthick, dashed] (5);
\path (8) edge[->, vthick, dashed] (4);
\path (1) edge[->, vthick, bend left] (4.north);
\path (4) edge[->, vthick, bend right] (2);
\drawpath{3,2,1}{vthick}
\drawpath{5,2,7,3}{vthick}
\path (7) edge[->, thin, bend left] (5);
\path (6) edge[->, thin] (7);
\path (6) edge[->, thin, bend right] (8);
\drawpath{3,6}{vthick}

\end{tikzpicture}
\caption{After mutation at $f_2^2,f_2^1,f_2^2$.}
\end{figure}

\begin{figure}[H]
\centering
\begin{tikzpicture}[every node/.style={inner sep=0, minimum size=0.8cm, thick}, x=0.5cm, y=0.5cm]
\node (1) at (0,3) [draw, circle] {$f_1^{0}$};
\node (2) at (6,3) [draw, circle,red, vthick] {$f_1^{1}$};
\node (3) at (12,3) [draw, circle,red, vthick] {$f_1^{2}$};
\node (4) at (18, 3) [draw, circle]{$f_1^{3}$};
\node (5) at (-3,0) [draw, circle] {$f_2^{0}$};
\node (6) at (3,0) [draw, circle] {$f_2^{1}$};
\node (7) at (9,0) [draw, circle,red, thick] {$f_2^{2}$};
\node (8) at (15, 0) [draw, circle]{$f_2^{3}$};
\drawpath{1,...,4}{vthick}
\path (5) edge[->, thin, bend right] (7);
\path (7) edge[->, thin] (6);
\path (6) edge[->, thin, bend right] (8);
\drawpath{4,6,3,7,2,5}{vthick}
\path (5) edge[->, vthick, dashed] (1);
\path (8) edge[->, vthick, dashed] (4);
\end{tikzpicture}
\caption{After mutation again at $f_1^2, f_1^1, f_2^2, f_1^2$.}
\end{figure}
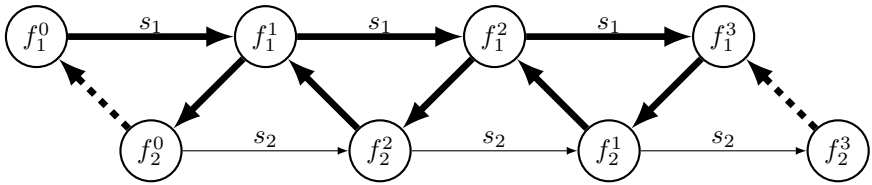
\begin{figure}[H]
\centering
\begin{tikzpicture}[every node/.style={inner sep=0, minimum size=0.8cm, thick}, x=0.5cm, y=0.5cm]
\node (1) at (-3,3) [draw, circle] {$f_1^{0}$};
\node (2) at (3,3) [draw, circle] {$f_1^{1}$};
\node (3) at (9,3) [draw, circle] {$f_1^{2}$};
\node (4) at (15, 3) [draw, circle]{$f_1^{3}$};
\node (5) at (0,0) [draw, circle] {$f_2^{0}$};
\node (6) at (6,0) [draw, circle] {$f_2^{2}$};
\node (7) at (12,0) [draw, circle] {$f_2^{1}$};
\node (8) at (18, 0) [draw, circle]{$f_2^{3}$};
\drawpath{1,...,4}{vthick}
\drawpath{5,...,8}{thin}
\drawpath{4,7,3,6,2,5}{vthick}
\path (5) edge[->, vthick, dashed] (1);
\path (8) edge[->, vthick, dashed] (4);
\node at (0,3.3) {$s_1$};
\node at (6,3.3) {$s_1$};
\node at (12,3.3) {$s_1$};
\node at (3,0.3) {$s_2$};
\node at (9,0.3) {$s_2$};
\node at (15,0.3) {$s_2$};
\end{tikzpicture}
\caption{Rearranging the quiver.}
\end{figure}
\end{subfigures}

We see that we have to permute the index:
$$\hat{f}_2^1:=f_2^2,\tab \hat{f}_2^2:=f_2^1$$ at the end. Similarly as before, this expresses the generators $\bf_i$ in terms of the mutated cluster variables $\what{X}_j$, and for the full quiver we have
\begin{Cor} The cluster embedding $\iota: \fD_\g\to \cD_\g$ corresponding to $\bi=(2,1,2,1,2,1)$ and $\bi'=(1,2,1,2,1,2)$ is related by quiver mutations at the pair of nodes 
\Eqn{
&\{f_1^2,f_1^{-2}\},\{f_1^1,f_1^{-1}\},\{f_2^2,f_2^{-2}\},\{f_1^2,f_1^{-2}\},\{f_2^2,f_2^{-2}\}, \{f_2^1,f_2^{-1}\}\\
& \{f_2^2,f_2^{-2}\},\{f_1^2, f_1^{-2}\}, \{f_1^1,f_1^{-1}\},\{f_2^2, f_2^{-2}\},\{f_1^2, f_1^{-2}\}.
}
\end{Cor}

\begin{Rem} In \cite{Le2}, it is also found that the above change of words can be realized by 12 mutations coming from a more natural geometric consideration:
$$(f_1^1,f_1^2,f_2^2,f_1^2),(f_2^2,f_2^1,f_1^1,f_2^2),(f_1^1,f_1^2,f_2^2,f_1^2),$$
where the groups correspond to the permutations (12)(23)(12) of the vertices of the triangle where the quiver is attached in the framed $G$-local system. The end result differs from the above quiver by an additional permutation of $f_1^1$ and $f_1^2$, but such difference will not play a role in this paper.
\end{Rem}
%==============================================================================
\section{Basic quivers}\label{sec:basicquiver}\label{sec:basicquiver}
In Section \ref{sec:embedding}, we obtain explicitly the $\cD_\g$-quiver corresponding to the embedding of the quantum group $\cU_q(\g)$ associated to the reduced expression of the longest element $w_0=s_{i_1}...s_{i_N}$. By their symmetric presentations, we observe that the $\cD_\g$-quiver is given by amalgamation of some quivers $Q$ and $\til{Q}$ where $\til{Q}$ is obtained by a mirror image of $Q$ along the vertical axis, together with flipping all the arrows. 

More precisely, let us arrange the quiver $Q$ so that its frozen vertices $\{e_i^0, f_i^0, f_i^{n_i}\}_{i\in I}$ are fixed on a triangle $ABC$ as shown in Figure \ref{quiverQ}. Let $\til{Q}$ be the mirror image of $Q$ with frozen vertices $\{e_i^0, f_i^0, f_i^{-n_i}\}_{i\in I}$ fixed on a triangle $A'B'C'$, but such that all arrows are flipped (i.e. with the same indexing, it has the exchange matrix $-B$ instead). Then the $\cD_\g$-quiver is obtained by amalgamating $Q$ and $\til{Q}$ along the frozen vertices at $\{e_i^0, f_i^0\}_{i\in I}$. We will call such external labeling of the basic quivers $Q$ and $\til{Q}$ the \emph{standard form}.

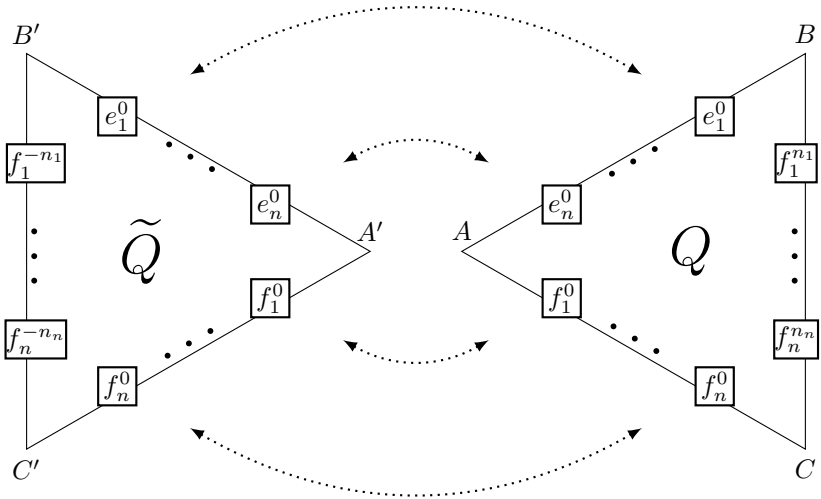
\begin{figure}[H]
\centering
\begin{tikzpicture}[every node/.style={inner sep=0, minimum size=0.5cm, thick},x=0.6cm,y=0.6cm]
\begin{scope}[shift={(-6,0)}]
\draw (0:5) node[anchor=south]{$A'$} 
-- (120:5) node[anchor=south]{$B'$}
 -- (240: 5) node[anchor=north]{$C'$}
 --cycle;
\node (een0) at (20: 3) [draw, fill=white, thick, outer sep=20] {$e_n^0$};
\node (ee10) at (100: 3) [draw, fill=white, thick, outer sep=20] {$e_1^0$};
\node (ff1n) at (140: 3) [draw, fill=white, thick] {$f_1^{-n_1}$};
\node (ffnn) at (220: 3) [draw, fill=white, thick] {$f_n^{-n_n}$};
\node (ffn0) at (260: 3) [draw, fill=white, thick, outer sep=20] {$f_n^0$};
\node (ff10) at (340: 3) [draw, fill=white, thick, outer sep=20] {$f_1^0$};
\node[rotate=90] at (180:2.3) {\huge$\cdots$};
\node[rotate=30] at (300:2.3) {\huge$\cdots$};
\node[rotate=-30] at (60:2.3) {\huge$\cdots$};
\node (QQ) at (0,0) [minimum size=2cm]{\huge $\til{Q}$};
\end{scope}
\begin{scope}[shift={(6,0)}]
\draw (60:5)node[anchor=south]{$B$}  
-- (180:5) node[anchor=south]{$A$} 
-- (300: 5) node[anchor=north]{$C$} 
--cycle;
\node (fnn) at (-40: 3) [draw, fill=white, thick] {$f_n^{n_n}$};
\node (f1n) at (40: 3) [draw, fill=white, thick] {$f_1^{n_1}$};
\node (e10) at (80: 3) [draw, fill=white, thick, outer sep=20] {$e_1^0$};
\node (en0) at (160: 3) [draw, fill=white, thick, outer sep=20] {$e_n^0$};
\node (f10) at (200: 3) [draw, fill=white, thick, outer sep=20] {$f_1^0$};
\node (fn0) at (280: 3) [draw, fill=white, thick, outer sep=20] {$f_n^0$};
\node[rotate=30] at (120:2.3) {\huge$\cdots$};
\node[rotate=-30] at (240:2.3) {\huge$\cdots$};
\node[rotate=90] at (0:2.3) {\huge$\cdots$};
\node (Q) at (0,0) [minimum size=2cm]{\huge $Q$};
\end{scope}
\path (ee10) edge[<->, dotted, thick, bend left] (e10);
\path (een0) edge[<->, dotted, thick, bend left] (en0);
\path (ff10) edge[<->, dotted, thick, bend right] (f10);
\path (ffn0) edge[<->, dotted, thick, bend right] (fn0);
\end{tikzpicture}
\caption{Amalgamating the quivers $Q$ and $\til{Q}$ in standard form.}
\label{quiverQ}
\end{figure}

We note that there are some freedom of choices of the quivers $Q$, namely, we can choose arbitrary arrows among the nodes $e_i^0$ and $f_i^0$. In order to fix the ambiguity, we first note that such amalgamation of two triangles give a triangulation of the disk with one puncture and two marked points:
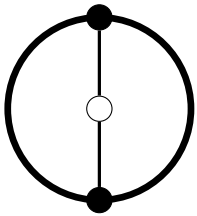
\begin{figure}[H]
\centering
\begin{tikzpicture}[x=0.6cm,y=0.6cm]
\draw [vthick] (0,0) circle (2);
\node (a) at (0,0) [draw, circle, minimum size=0.2]{};
\node (b) at (0,2) [draw, circle, minimum size=0.2, fill=black]{};
\node (c) at (0,-2) [draw, circle, minimum size=0.2, fill=black]{};
\path (a) edge[-, very thick] (b);
\path (a) edge[-, very thick] (c);
\end{tikzpicture}
\caption{Triangulation of a disk with one puncture and two marked points.}
\end{figure}

Therefore in order to realize the embedding naturally as associated to triangulations of such surface, the quiver $Q$ associated to the triangle $ABC$, should be mutation equivalent to the quiver $\til{Q}$ associated to triangle $B'A'C'$ in this clockwise order.

Let $\cM$ be the mutation sequence reversing the reduced word $(i_1,...,i_N)\to (i_N,...,i_1)$. If $Q'$ is the subquiver of $Q$ with nodes $\{f_i^k\}$, then $\cM(Q')$ is naturally given by a mirror image of $Q'$ with all the arrows flipped, or in terms of Figure \ref{quiverQ}, the triangle is flipped from $ABC$ to $B'A'C'$. It turns out that we have to identify the frozen nodes with its \emph{Dynkin involution} $$\h:I\to I,$$ where by definition, the longest element acts on simple roots as \Eq{
w_0\cdot\a_i=-\a_{\h(i)}.
}
Hence if we let $\cM_\h$ be the mutation sequence changing the reduced word $$\cM_\h:(i_1,...,i_N)\to (\h(i_1),...,\h(i_N)),$$ then we naturally also want to identify $Q$ with $\cM_\h(Q)$.

With these observations, we made the following definition.

\begin{Def}\label{basicquiver} A \emph{basic quiver} $Q$ for $\cD_\g$ corresponding to the word $w_0=s_{i_1}...s_{i_N}$ is a quiver associated to the triangle $ABC$ such that 
\begin{itemize}
\item the amalgamation of $Q$ and $\til{Q}$ gives the $\cD_\g$-quiver,
\item $\cM(Q)$ is identical to the quiver $\til{Q}$, where the frozen nodes $\{f_i^0, f_i^{n_i}, e_i^0\} $ of $\cM(Q)$ is identified with the frozen nodes $\{f_i^{-n_i}, f_i^0, e_{\h(i)}^0\}$ of $\til{Q}$.
\item $Q$ is identical to the quiver $\cM_\h(Q)$, where the frozen nodes $\{f_i^0, f_i^{n_i},e_i^0\}$ of $Q$ is identified with the frozen nodes $\{f_{\h(i)}^0, f_{\h(i)}^{n_i}, e_{\h(i)}^0\}$ of $\cM_\h(Q)$.
\end{itemize}
\end{Def}
Note that when $\h=id$, the third condition is trivial.

\begin{Thm}For each $\g$ of simple Lie type, there exists a unique basic quiver $Q$.
\end{Thm}
\begin{proof}
We need to solve for the required relations among the nodes $\{e_i^0, f_i^0\}$ that satisfies the definition of a basic quiver. First of all, by the construction of the elementary quivers in Section \ref{sec:quiver:elementary}, we can naturally determine the arrows between $\{f_i^0\}$ already by reading $w_0$ from the left.

Furthermore, Theorem \ref{initial} implies that any quiver mutations preserve the relations below whenever the initial term $X_{e_i^0}=e(-2u_j^a)$  for $a>1$ in the quantum group embedding:
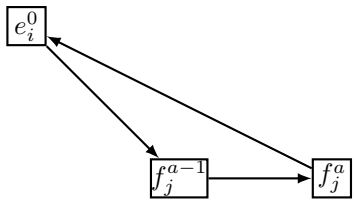
\begin{figure}[H]
\centering
\begin{tikzpicture}[every node/.style={inner sep=0, minimum size=0.5cm, thick, fill=white},x=2cm,y=2cm]
\node (1) at (0,1)[draw, thick] {$e_i^0$};
\node (2) at (1,0)[draw, thick] {$f_j^{a-1}$};
\node (3) at (2,0)[draw, thick] {$f_j^a$};
\drawpath{1,2,3,1}{thick}
\end{tikzpicture}
\end{figure}
Hence for the basic quiver, we see that in order for $\cM(Q)$ to be identical to $\til{Q}$, we must also have the above quiver for $a=1$, and this establishes the arrows between $\{e_i^0\}$ and $\{f_i^0\}$. Therefore it remains to determine the arrows among the nodes $\{e_i^0\}$.

Since quiver mutation is a bijection, it suffices to construct the basic quiver for some reduced word $\bi\in \fR$. Hence throughout the proof, we will use the same reduced word for $\bi$ in Section \ref{sec:embedding} for each type of $\g$ in the construction of the $\cD_\g$-quiver. 

Let $Q_+$ denote the subquiver of $\cD_\g$ containing the nodes $\{f_i^n\}_{n\geq 0, i\in I}$ and $\{e_i^0\}_{i\in I}$. The mutation sequences $\cM$ and $\cM_\h$ are obtained by recursively bringing the required index to the right of $\bi$ using the \emph{swapping} $s_is_j=s_js_i$, or the \emph{Coxeter moves}.

\textbf{Type $A_n$.} This is a very special case as the change of words 
$$\bi=(121321...n,(n-1),...1)\corr (123... n, 123... (n-1),... 123121)$$
does not require any Coxeter moves, but only swapping between commuting reflections. Therefore $\cM=id$, and the definition of basic quiver requires that $Q$ associated to $ABC$ is the same as $\til{Q}$ associated to $B'A'C'$, where the order of $\{e_i^0\}$ is reversed. In particular it says that the sides $\ve{AC}$ and $\ve{BC}$ of $Q$ is the same as the sides $\ve{B'C'}$ and $\ve{A'C'}$ of $\til{Q}$, which by definition is just the sides $\ve{CB}$ and $\ve{CA}$ of $Q$. This uniquely determines the arrows among the nodes $\{f_i^0\}$, and between the nodes $\{e_i^0\}$ and $\{f_i^0\}$.

The Dynkin involution is given by 
\Eq{
\h(i):=n+1-i.} By considering the mutation $\cM_\h$ of the Dynkin involution
$$(121321... n,(n-1),... 1) \to (n,(n-1),n,(n-2),(n-1),n,.... 123...n)$$
which in a sense is just flipping the diagram upside-down, we observe that the arrows between consecutive $\{e_i^0\}$ are mutated once whenever there is a change of word $(...121...)\corr (...212...)$. The arrows among $\{e_i^0\}$ are chosen such that $Q=\cM_\h(Q)$. The end result forces $Q$ to have a magical $\Z_3$ symmetry, and we recover the well-known basic quiver for type $A_n$ associated to $n$-triangulation first studied by \cite{FG1}.

More precisely, the basic quiver $Q$ is obtained by attaching to $Q_+$ the additional arrows:
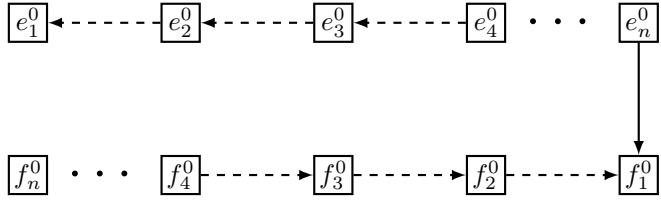
\begin{figure}[H]
\centering
\begin{tikzpicture}[every node/.style={inner sep=0, minimum size=0.5cm, thick, fill=white},x=2cm,y=2cm]
\node (1) at (0,1)[draw] {$e_1^0$};
\node (2) at (1,1)[draw] {$e_2^0$};
\node (3) at (2,1)[draw] {$e_3^0$};
\node (4) at (3,1)[draw] {$e_4^0$};
\node (en) at (4,1)[draw] {$e_n^0$};
\node (fn) at (0,0)[draw] {$f_n^0$};
\node (5) at (1,0)[draw] {$f_4^0$};
\node (6) at (2,0)[draw] {$f_3^0$};
\node (7) at (3,0)[draw] {$f_2^0$};
\node (8) at (4,0)[draw] {$f_1^0$};
\node at (0.5,0) {\huge$\cdots$};
\node at (3.5,1) {\huge$\cdots$};
\drawpath{4,3,2,1}{thick, dashed}
\drawpath{5,6,7,8}{thick, dashed}
\drawpath{en,8}{thick}
\end{tikzpicture}
\caption{Additional arrows attaching to $Q_+$ to give $Q$ in type $A_n$.}
\end{figure}

\begin{figure}[!htb]
\centering
\begin{tikzpicture}[every node/.style={inner sep=0, minimum size=0.5cm, thick, fill=white},x=1.732cm,y=1cm]
\draw (-0.3,0) node[anchor=south, fill=none]{$A$} 
-- (5.1,5.5) node[anchor=south]{$B$}
 -- (5.1,-5.5) node[anchor=north]{$C$}
 --cycle;
\node (1) at (1,1)[draw] {$e_n^0$};
\node (2) at (2,2)[draw] {$e_3^0$};
\node (3) at (3,3)[draw] {$e_2^0$};
\node (4) at (4,4)[draw] {$e_1^0$};
\node (5) at (1,-1)[draw] {$f_1^0$};
\node (6) at (2,0)[draw, circle] {};
\node (7) at (3,1)[draw, circle] {};
\node (8) at (4,2)[draw, circle] {};
\node (9) at (5,3)[draw] {$f_1^n$};
\node (10) at (2,-2)[draw] {$f_2^0$};
\node (11) at (3,-1)[draw,circle] {};
\node (12) at (4,0)[draw,circle] {};
\node (13) at (5,1)[draw] {$f_2^{n-1}$};
\node (14) at (3,-3)[draw] {$f_3^0$};
\node (15) at (4,-2)[draw,circle] {};
\node (16) at (5,-1)[draw] {$f_3^{n-2}$};
\node (17) at (4,-4)[draw] {$f_n^0$};
\node (18) at (5,-3)[draw] {$f_n^1$};
\drawpath{2,...,4}{thick, dashed}
\node[rotate=30, fill=none] at (1.5,1.5){\huge$\cdots$};
\drawpath{9,13,16}{thick, dashed}
\node[rotate=90, fill=none] at (5,-2){\huge$\cdots$};
\drawpath{14,10,5}{thick, dashed}
\node[rotate=-30, fill=none] at (3.5,-3.5){\huge$\cdots$};
\node[rotate=-30, fill=none] at (4.5,-2.5){\huge$\cdots$};
\node[rotate=90, fill=none] at (4,-3){\huge$\cdots$};
\drawpath{1,5}{thick}
\drawpath{2,6,10}{thick}
\drawpath{3,7,11,14}{thick}
\drawpath{4,8,12,15}{thick}
\drawpath{9,4}{thick}
\drawpath{13,8,3}{thick}
\drawpath{16,12,7,2}{thick}
\drawpath{15,11,6,1}{thick}
\drawpath{17,18}{thick}
\drawpath{14,15,16}{thick}
\drawpath{10,11,12,13}{thick}
\drawpath{5,6,7,8,9}{thick}
\end{tikzpicture}
\caption{Basic quiver in type $A_n$ with $\Z_3$ symmetry.}
\label{basicAn}
\end{figure}

\textbf{Type $B_n$ and $C_n$.} The change of words $\cM$ for
$$\bi=(1212\; 32123\;...n(n-1)...1...(n-1)n)$$
consists of $\frac23n(n-1)(n-2)$ simply-laced mutations, and $\frac12n(n-1)$ doubly-laced mutations. Recall from Section \ref{sec:mutation:doubly} that each doubly-laced mutation corresponds to 3 quiver mutations. Hence the change of words $\cM$ corresponds to $$\frac23n(n-1)(n-2)+\frac32n(n-1)=\frac16n(n-1)(4n+1)$$ quiver mutations.

By comparing $\cM(Q)$ with $\til{Q}$, we found that the basic quiver is obtained by adjoining $Q_+$ the following arrows in type $B_n$:
\begin{figure}[H]
\centering
\begin{tikzpicture}[every node/.style={inner sep=0, minimum size=0.5cm, thick, fill=white},x=2cm,y=2cm]
\node (1) at (0,1)[draw] {$e_1^0$};
\node (2) at (1,1)[draw] {$e_2^0$};
\node (3) at (2,1)[draw] {$e_3^0$};
\node (4) at (3,1)[draw] {$e_4^0$};
\node (en) at (4,1)[draw] {$e_n^0$};
\node (5) at (0,0)[draw] {$f_1^0$};
\node (6) at (1,0)[draw] {$f_2^0$};
\node (7) at (2,0)[draw] {$f_3^0$};
\node (8) at (3,0)[draw] {$f_4^0$};
\node (fn) at (4,0)[draw] {$f_n^0$};
\node at (3.5,0) {\huge$\cdots$};
\node at (3.5,1) {\huge$\cdots$};
\drawpath{1,2,3,4}{vthick, dashed}
\drawpath{8,7,6,5}{vthick, dashed}
\drawpath{1,5}{thin}
\end{tikzpicture}
\caption{Additional arrows attaching to $Q_+$ to give $Q$ in type $B_n$.}
\end{figure}
and the following arrows in type $C_n$:
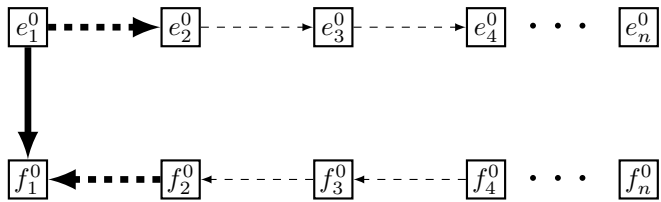
\begin{figure}[H]
\centering
\begin{tikzpicture}[every node/.style={inner sep=0, minimum size=0.5cm, thick, fill=white},x=2cm,y=2cm]
\node (1) at (0,1)[draw] {$e_1^0$};
\node (2) at (1,1)[draw] {$e_2^0$};
\node (3) at (2,1)[draw] {$e_3^0$};
\node (4) at (3,1)[draw] {$e_4^0$};
\node (en) at (4,1)[draw] {$e_n^0$};
\node (5) at (0,0)[draw] {$f_1^0$};
\node (6) at (1,0)[draw] {$f_2^0$};
\node (7) at (2,0)[draw] {$f_3^0$};
\node (8) at (3,0)[draw] {$f_4^0$};
\node (fn) at (4,0)[draw] {$f_n^0$};
\node at (3.5,0) {\huge$\cdots$};
\node at (3.5,1) {\huge$\cdots$};
\drawpath{1,2}{vthick, dashed}
\drawpath{2,3,4}{thin, dashed}
\drawpath{8,7,6}{thin, dashed}
\drawpath{6,5}{vthick, dashed}
\drawpath{1,5}{vthick}
\end{tikzpicture}
\caption{Additional arrows attaching to $Q_+$ to give $Q$ in type $C_n$.}
\end{figure}

\textbf{Type $D_n$.} The change of words $\cM$ for
$$\bi=(012012\; 320123\;...(n-1)...2012...(n-1))$$
consists of $\frac23n(n-1)(n-2)$ mutations. When $n$ is even, we have $\h=id$, and the condition $\cM(Q)=\til{Q}$ uniquely determines the arrows among the frozen nodes. Otherwise when $n$ is odd, the condition $\cM(Q)=\til{Q}$ uniquely determines the arrows among the frozen nodes except $e_0^0$ and $e_1^0$. In this case the Dynkin involution is given by 
\Eq{
\h(i)=\case{1-i&i=0,1,\\i&otherwise,}} hence we see that from our choice of $w_0$, $\cM_\h$ is trivial. This means that we cannot have arrows between $e_0^0$ and $e_1^0$.

The resulting basic quiver $Q$ is then obtained by taking $Q_+$ and adjoining the following arrows:
\begin{figure}[H]
\centering
\begin{tikzpicture}[every node/.style={inner sep=0, minimum size=0.5cm, thick, fill=white},x=3cm,y=1cm]
\node (0) at (0,4.7)[draw] {$e_0^0$};
\node (1) at (0,3.3)[draw] {$e_1^0$};
\node (2) at (1,4)[draw] {$e_2^0$};
\node (3) at (2,4)[draw] {$e_3^0$};
\node (4) at (3,4)[draw] {$e_4^0$};
\node (en) at (4,4)[draw] {$e_n^0$};
\node (5) at (0,1.7)[draw] {$f_0^0$};
\node (6) at (0,0.3)[draw] {$f_1^0$};
\node (7) at (1,1)[draw] {$f_2^0$};
\node (8) at (2,1)[draw] {$f_3^0$};
\node (9) at (3,1)[draw] {$f_4^0$};
\node (fn) at (4,1)[draw] {$f_n^0$};
\node at (3.5,4) {\huge$\cdots$};
\node at (3.5,1) {\huge$\cdots$};
\drawpath{1,2,3,4}{thick, dashed}
\drawpath{0,2}{thick, dashed}
\drawpath{9,8,7,6}{thick, dashed}
\drawpath{7,5}{thick, dashed}
\drawpath{0,5}{thick, bend right}
\drawpath{1,6}{thick, bend right}
\end{tikzpicture}
\caption{Additional arrows attaching to $Q_+$ to give $Q$ in type $D_n$.}
\end{figure}
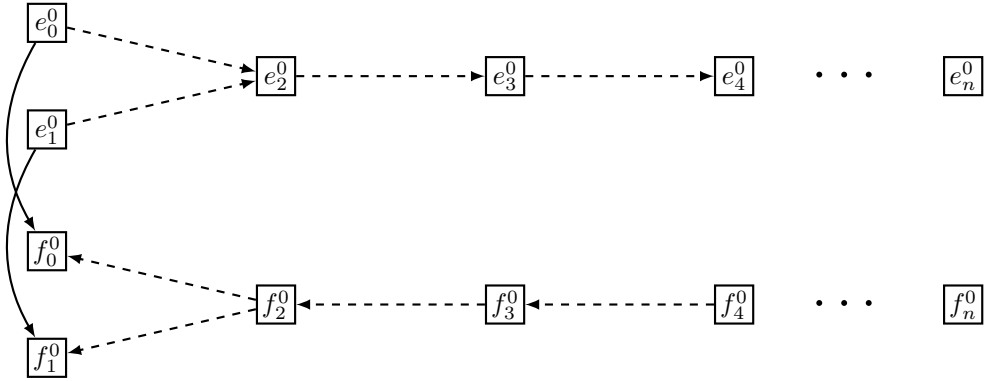

\textbf{Type $E_6$.}  The change of words $\cM$ for $$\bi=(3\;43\;034\; 230432\;12340321\;5432103243054321)$$ consists of 78 mutations. The Dynkin involution is given by 
\Eq{
\h(i)=\case{6-i&i>0,\\0&i=0,}} whence the change of words $\cM_\h$ consists of 42 mutations. After comparing the quiver, we found that the basic quiver $Q$ is obtained by taking $Q_+$ and adjoining the following arrows:
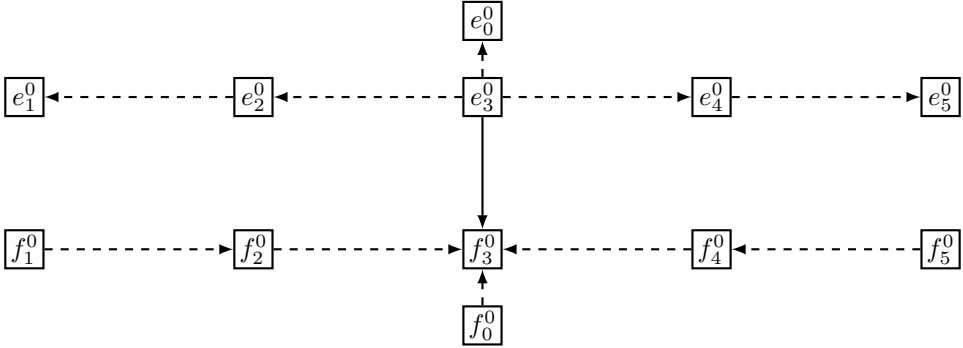
\begin{figure}[H]
\centering
\begin{tikzpicture}[every node/.style={inner sep=0, minimum size=0.5cm, thick, fill=white},x=3cm,y=1cm]
\node (0) at (2,3)[draw] {$e_0^0$};
\node (1) at (0,2)[draw] {$e_1^0$};
\node (2) at (1,2)[draw] {$e_2^0$};
\node (3) at (2,2)[draw] {$e_3^0$};
\node (4) at (3,2)[draw] {$e_4^0$};
\node (5) at (4,2)[draw] {$e_5^0$};
\node (6) at (2,-1)[draw] {$f_0^0$};
\node (7) at (0,0)[draw] {$f_1^0$};
\node (8) at (1,0)[draw] {$f_2^0$};
\node (9) at (2,0)[draw] {$f_3^0$};
\node (10) at (3,0)[draw] {$f_4^0$};
\node (11) at (4,0)[draw] {$f_5^0$};
\drawpath{7,8,9}{thick, dashed}
\drawpath{6,9}{thick, dashed}
\drawpath{11,10,9}{thick,dashed}
\drawpath{3,9}{thick}
\drawpath{3,2,1}{thick, dashed}
\drawpath{3,4,5}{thick, dashed}
\drawpath{3,0}{thick, dashed}
\end{tikzpicture}
\caption{Additional arrows attaching to $Q_+$ to give $Q$ in type $E_6$.}
\end{figure}

\textbf{Type $E_7$.}  The change of words $\cM$ for
$$\bi=(3\;43\;034\;230432\;12340321\;5432103243054321\;654320345612345034230123456)$$
consists of $336$ mutations. By comparing $\cM(Q)$ and $\til{Q}$, the basic quiver $Q$ is found to be obtained by taking $Q_+$ and adjoining the following arrows:
\begin{figure}[H]
\centering
\begin{tikzpicture}[every node/.style={inner sep=0, minimum size=0.5cm, thick, fill=white},x=2cm,y=1cm]
\node (0) at (2,3)[draw] {$e_0^0$};
\node (1) at (0,2)[draw] {$e_1^0$};
\node (2) at (1,2)[draw] {$e_2^0$};
\node (3) at (2,2)[draw] {$e_3^0$};
\node (4) at (3,2)[draw] {$e_4^0$};
\node (5) at (4,2)[draw] {$e_5^0$};
\node (6) at (5,2)[draw] {$e_6^0$};
\node (7) at (2,-1)[draw] {$f_0^0$};
\node (8) at (0,0)[draw] {$f_1^0$};
\node (9) at (1,0)[draw] {$f_2^0$};
\node (10) at (2,0)[draw] {$f_3^0$};
\node (11) at (3,0)[draw] {$f_4^0$};
\node (12) at (4,0)[draw] {$f_5^0$};
\node (13) at (5,0)[draw] {$f_6^0$};
\drawpath{8,9,10}{thick, dashed}
\drawpath{7,10}{thick, dashed}
\drawpath{13,12,11,10}{thick,dashed}
\drawpath{3,10}{thick}
\drawpath{3,2,1}{thick, dashed}
\drawpath{3,0}{thick, dashed}
\drawpath{3,4,5,6}{thick,dashed}
\end{tikzpicture}
\caption{Additional arrows attaching to $Q_+$ to give $Q$ in type $E_7$.}
\end{figure}
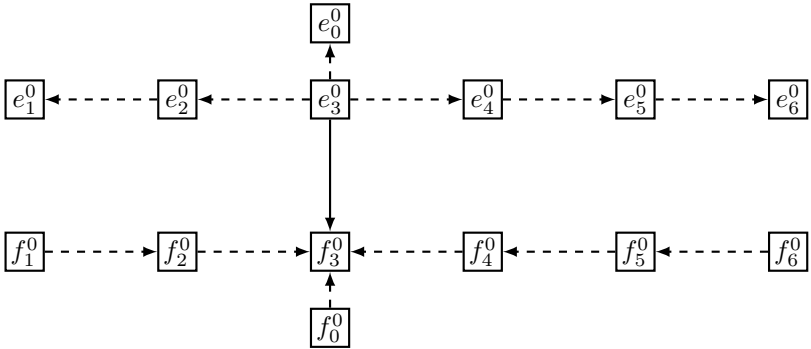
\textbf{Type $E_8$.}  The change of words $\cM$ for
\Eqn{
\bi=&(3\;43\;034\;230432\;12340321\;5432103243054321\;654320345612345034230123456\\
&765432103243546503423012345676543203456123450342301234567)
}
consists of $1120$ mutations. By comparing $\cM(Q)$ and $\til{Q}$, the basic quiver $Q$ is found to be obtained by taking $Q_+$ and adjoining the following arrows:
\begin{figure}[H]
\centering
\begin{tikzpicture}[every node/.style={inner sep=0, minimum size=0.5cm, thick, fill=white},x=2cm,y=1cm]
\node (0) at (2,3)[draw] {$e_0^0$};
\node (1) at (0,2)[draw] {$e_1^0$};
\node (2) at (1,2)[draw] {$e_2^0$};
\node (3) at (2,2)[draw] {$e_3^0$};
\node (4) at (3,2)[draw] {$e_4^0$};
\node (5) at (4,2)[draw] {$e_5^0$};
\node (6) at (5,2)[draw] {$e_6^0$};
\node (7) at (6,2)[draw] {$e_7^0$};
\node (8) at (2,-1)[draw] {$f_0^0$};
\node (9) at (0,0)[draw] {$f_1^0$};
\node (10) at (1,0)[draw] {$f_2^0$};
\node (11) at (2,0)[draw] {$f_3^0$};
\node (12) at (3,0)[draw] {$f_4^0$};
\node (13) at (4,0)[draw] {$f_5^0$};
\node (14) at (5,0)[draw] {$f_6^0$};
\node (15) at (6,0)[draw] {$f_7^0$};
\drawpath{9,10,11}{thick, dashed}
\drawpath{8,11}{thick, dashed}
\drawpath{15,14,13,12,11}{thick,dashed}
\drawpath{3,2,1}{thick, dashed}
\drawpath{3,4,5,6,7}{thick, dashed}
\drawpath{3,0}{thick, dashed}
\drawpath{3,11}{thick}
\end{tikzpicture}
\caption{Additional arrows attaching to $Q_+$ to give $Q$ in type $E_8$.}
\end{figure}
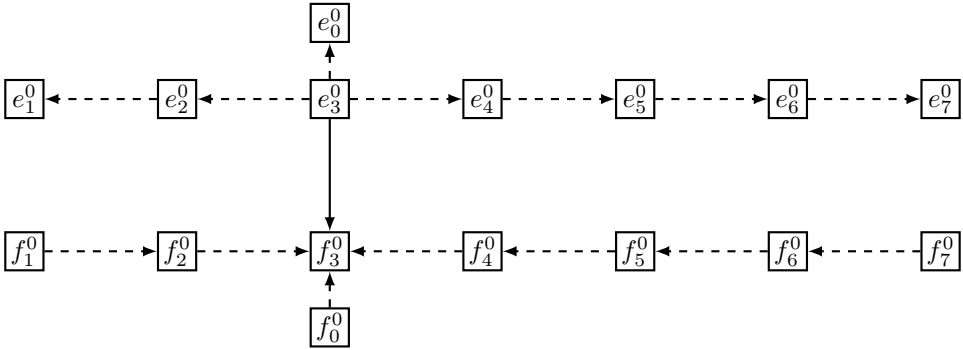
\textbf{Type $F_4$.} The change of words $\cM$ for
$$\bi=(3232\;12321\;432312343213234)$$
consists of 32 simply-laced mutations and 18 doubly-laced mutations, hence it corresponds to $32+3\x 18=86$ quiver mutations. The resulting basic quiver is obtained by adjoining to $Q_+$ the following arrows:
\begin{figure}[H]
\centering
\begin{tikzpicture}[every node/.style={inner sep=0, minimum size=0.5cm, thick, fill=white},x=2cm,y=2cm]
\node (1) at (0,1)[draw] {$e_1^0$};
\node (2) at (1,1)[draw] {$e_2^0$};
\node (3) at (2,1)[draw] {$e_3^0$};
\node (4) at (3,1)[draw] {$e_4^0$};
\node (5) at (0,0)[draw] {$f_1^0$};
\node (6) at (1,0)[draw] {$f_2^0$};
\node (7) at (2,0)[draw] {$f_3^0$};
\node (8) at (3,0)[draw] {$f_4^0$};
\drawpath{3,2,1}{vthick, dashed}
\drawpath{3,4}{thin, dashed}
\drawpath{5,6,7}{vthick, dashed}
\drawpath{8,7}{thin, dashed}
\drawpath{3,7}{thin}
\end{tikzpicture}
\caption{Additional arrows attaching to $Q_+$ to give $Q$ in type $F_4$.}
\end{figure}
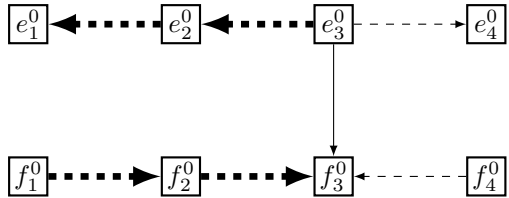

\textbf{Type $G_2$.}
Finally, the change of words $\cM$ for $$\bi=(2,1,2,1,2,1)\to (1,2,1,2,1,2)$$
is described in Section \ref{sec:mutation:G2}, which consists of 11 or 12 quiver mutations. The basic quiver $Q$ can be presented as follows for the two cases in Figure \ref{G2quiver}.
\begin{subfigures}
\begin{figure}[!htb]
\centering
\begin{tikzpicture}[every node/.style={inner sep=0, minimum size=0.5cm, thick, fill=white},x=2cm,y=2cm]
\node (1) at (0,1)[draw] {$f_1^0$};
\node (2) at (1,1)[draw] {$f_1^1$};
\node (3) at (2,1)[draw] {$f_1^2$};
\node (4) at (3,1)[draw] {$f_1^3$};
\node (5) at (0,0)[draw] {$f_2^0$};
\node (6) at (1,0)[draw] {$f_2^1$};
\node (7) at (2,0)[draw] {$f_2^2$};
\node (8) at (3,0)[draw] {$f_2^3$};
\node (9) at (2,2)[draw] {$e_1^0$};
\node (10) at (1,-1)[draw] {$e_2^0$};
\node (empty) at (4,0.45)[minimum size=0cm]{};
\drawpath{8,3,7,2,6,1,2,3,4,9,3}{vthick}
\drawpath{6,10,5,6,7,8}{thin}
\drawpath{4,8}{vthick, dashed}
\drawpath{1,5}{vthick, dashed}
\path (10) edge[-, vthick, dashed, out=0, in=270] (empty);
\path (empty) edge[->, vthick, dashed, out=90, in=0] (9);
\end{tikzpicture}
\caption{Basic quiver $Q$ for $w_0=s_2s_1s_2s_1s_2s_1$.}\label{G2basicquiver}
\end{figure}
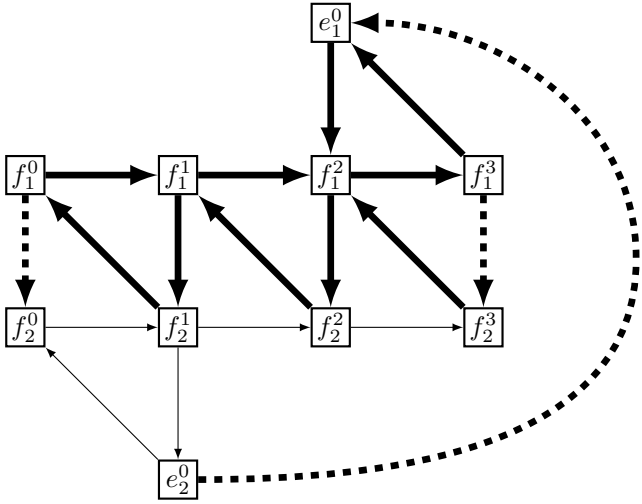
\begin{figure}[!htb]
\centering
\begin{tikzpicture}[every node/.style={inner sep=0, minimum size=0.5cm, thick, fill=white},x=2cm,y=2cm]
\node (1) at (0,1)[draw] {$f_2^0$};
\node (2) at (1,1)[draw] {$f_2^1$};
\node (3) at (2,1)[draw] {$f_2^2$};
\node (4) at (3,1)[draw] {$f_2^3$};
\node (5) at (0,0)[draw] {$f_1^0$};
\node (6) at (1,0)[draw] {$f_1^1$};
\node (7) at (2,0)[draw] {$f_1^2$};
\node (8) at (3,0)[draw] {$f_1^3$};
\node (9) at (2,2)[draw] {$e_2^0$};
\node (10) at (1,-1)[draw] {$e_1^0$};
\node (empty) at (4,0.45)[minimum size=0cm]{};
\drawpath{1,2,3,4,9,3}{thin}
\drawpath{6,10,5,6,7,8,3,7,2,6,1}{vthick}
\drawpath{4,8}{vthick, dashed}
\drawpath{1,5}{vthick, dashed}
\path (10) edge[-, vthick, dashed, out=0, in=270] (empty);
\path (empty) edge[->, vthick, dashed, out=90, in=0] (9);
\end{tikzpicture}
\caption{Basic quiver for $w_0=s_1s_2s_1s_2s_1s_2$, i.e. $\cM(Q)$.}
\end{figure}
\label{G2quiver}
\end{subfigures}
This is identical to the $G_2$ quiver found in \cite{Le2}. In particular, it demonstrates the Langland's duality of the change of short and long roots as a change of weights of the quivers in the diagram.　Also as mentioned before, $\cM(Q)$ is a mirror image of $Q$ with all arrows flipped as desired.

This completes the proof of the Theorem.
\end{proof}

\begin{Cor} The basic quiver is mutation equivalent to the $Conf_3 \cA_G$ quiver for $G$ of type $A_n,B_n,C_n,D_n,G_2$ described in \cite{Le1,Le2}.
\end{Cor}
\begin{proof} The $Conf_3 \cA_G$ quivers described in \cite{Le1, Le2} correspond to the words $w_0=w_c^\frac{h}{2}$ where $w_c$ is the Coxeter element and $h$ the Coxeter number. Hence one can check directly the mutation sequence from the change of words of $w_0$ and compare the quivers.
\end{proof}

\begin{Rem} As we can see, the arrows joining the nodes $\{e_i^0\}$ and the arrows joining the nodes $\{f_{\h(i)}^0\}$ turns out to be opposite to each other. This should reflect some internal symmetries of the moduli spaces $Conf_3 \cA_G$ of the configurations of triples of principal flags, and one should be able to find a more conceptual way to fix the basic quiver. We believe that this uniqueness theorem can be used to solve the series of conjectures regarding the uniqueness of the cluster structure proposed in Section 3 of \cite{Le2}.
\end{Rem}

The mutation $\cM$ corresponds to the transposition interchanging the sides $AC$ and $BC$ of the triangle where $\{f_i^0\}$ and $\{f_i^{n_i}\}$ are attached respectively (cf. Figure \ref{quiverQ}). On the other hand, in order to realize $S_3$ symmetry, we also want a mutation sequence corresponding to transposition of sides $AB$ and $AC$, where $\{e_{\h(i)}^0\}$ and $\{f_i^0\}$ are attached respectively. Also note that $f_i^{n_i}=:e_i^{-m_i}$ in the quantum group embedding. Hence such mutation should correspond to the longest Lusztig's transformation (see Definition \ref{LusIso}):
\Eqn{
T_{i_1}T_{i_2}...T_{i_N}(\be_i)&= q_i\bf_{\h(i)}K_{\h(i)}\inv\\
T_{i_1}T_{i_2}...T_{i_N}(\bf_i)&= q_i\be_{i}K_{i}\\
}
In \cite{Ip4}, we showed that these transformations $T_i$ are represented by certain unitary transformation given by conjugation of the Weyl elements. Hence we conjecture that 
\begin{Con}\label{EFflip} The Lusztig's isomorphisms $T_i$ are represented by a sequence of quiver mutations.
\end{Con}
This will also give a representation theoretic meaning of the mutation sequences found explicitly for type $A_n,B_n,C_n,D_n$ \cite{Le1} and $G_2$ \cite{Le2}, as well as proving the conjecture of $S_3$ symmetry regarding type $E_n$ and $F_4$ proposed in \cite{Le2}. 

\begin{Rem}In the product formula of $R$ described in the next section, the transformations $T_i$ generate the split-real version of the so-called quantum Weyl group introduced in \cite{LS1}, which is a byproduct of the representation theory of the quantized algebra of functions on $G$, and is based on a choice of ``good generators" for certain representations of the quantized enveloping algebra. Through this conjecture, it will be interesting to recast the concept of quantum Weyl group into the language of cluster transformations. We thank Yan Soibelman for the remarks.
\end{Rem}
%==============================================================================
\section{Factorization of the $R$-matrix}\label{sec:R}
%==============================================================================
In this section, we will prove a factorization formula for the universal $R$ matrix such that it is expressed in terms of products of quantum dilogarithms, with the arguments given by monomials of the quantum cluster variables. This generalizes the factorization in type $A_n$ found in \cite{SS2}, which in turn is a generalization of the factorization given in \cite{Fa2} for $\cU_q(\sl_2)$, where it has been used to construct new continuous braided tensor category of representations of $\cU_q(\sl(2,\R))$ \cite{BT, PT1, PT2}.

\subsection{Positive Lusztig's isomorphism}\label{sec:R:Lusztig}
First we recall the positive version of Lusztig's isomorphism giving the expression of non-simple root generators:
\begin{Def} \cite{Ip4}\label{LusIso} We define the ``positive version" of Lusztig's isomorphism on the simple generators by:
\Eqn{
T_i(K_j)&:=K_jK_i^{-a_{ij}},\\
T_i(\be_i)&:=q_i f_iK_i\inv,\\
T_i(\be_j)&:=(-1)^{a_{ij}}\left[\left[\be_i,...[\be_i,\be_j]_{q_i^{\frac{a_{ij}}{2}}}\right]_{q_i^{\frac{a_{ij}+2}{2}}}...\right]_{q_i^{\frac{-a_{ij}-2}{2}}}\prod_{k=1}^{-a_{ij}}(q_i^k-q_i^{-k})\inv,\\
T_i(\bf_i)&:=q_i \be_i K_i,\\
T_i(\bf_j)&:=(-1)^{a_{ij}}\left[\left[\bf_i,...[\bf_i,\bf_j]_{q_i^{\frac{a_{ij}}{2}}}\right]_{q_i^{\frac{a_{ij}+2}{2}}}...\right]_{q_i^{\frac{-a_{ij}-2}{2}}}\prod_{k=1}^{-a_{ij}}(q_i^k-q_i^{-k})\inv,
}
where
$$[X,Y]_q:=qXY-q\inv YX.$$
\end{Def}
\begin{Prop} \cite{Ip4}
Let $\bi=(i_1,...,i_N)\in\fR$ be a reduced word. Let $$\be_{\a_k}:=T_{i_1}T_{i_2}...T_{i_{k-1}}(\be_{i_k})$$ and similarly for $\bf_{\a_k}$. Then $\be_{\a_k}$ and $\bf_{\a_k}$ are  positive self-adjoint operators under the positive representation $\cP_\l$ for every $k=1,...,N$.
\end{Prop}
%==============================================================================

\subsection{Coproduct of $\fD_\g$ and the $\cZ_\g$-quiver}\label{sec:R:coproduct}
The coalgebra structure of $\cU_q(\g)$ can naturally be represented by amalgamation of two $\cD_\g$ quivers, associated to triangulations of a disk with two punctures and two marked points on the boundary:
\begin{figure}[H]
\centering
\begin{tikzpicture}[every node/.style={draw,thick,circle,inner sep=0pt}, x=0.6cm,y=0.6cm]
\draw [vthick] (0,0) circle (2);
\node (a) at (1,0) [minimum size=8]{};
\node (d) at (-1,0) [minimum size=8]{};
\node (b) at (0,2) [minimum size=8, fill=black]{};
\node (c) at (0,-2) [minimum size=8, fill=black]{};
\path (a) edge[-, very thick] (b);
\path (a) edge[-, very thick] (c);
\path (d) edge[-, very thick] (b);
\path (d) edge[-, very thick] (c);
\path (b) edge[-, very thick] (c);
\end{tikzpicture}
\caption{Triangulation of a disk with two punctures and two marked points.}
\end{figure}
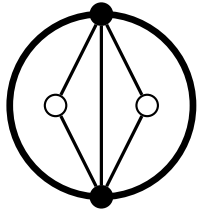

\begin{Def} The $\cZ_\g$ quiver is obtained by amalgamating two $\cD_\g$-quivers, where the frozen nodes $f_i^{n_i}$ of the first quiver is identified with $f_i^{-n_i}$ of the second quiver. For simplicity, we will denote the vertices of the second $\cD_\g$-quiver by $\{{f'}_i^{-n_i}... {f'}_i^{n_i}, {e'}_i^0\}_{i\in I}$ such that $f_i^{n_i}={f'}_i^{-n_i}$ in $Z_\g$. We will also denote by $\cZ_\g$ the corresponding quantum torus algebra.
\end{Def}
Then one can easily observe the following
\begin{Prop} We have an embedding
\Eq{
(\iota\ox\iota) \circ \D: \fD_\g\to \cZ_\g\subset \cD_\g\ox \cD_\g}
where the coproduct $\D(\be_i)$ (resp. $\D(\bf_i)$) can be represented in the $Z_\g$-quiver by concatenating the $E_i$-path (resp. $F_i$-path) of the two $\cD_\g$ quivers and ignoring the last vertex. The coproduct $\D(K_i)$ (resp. $\D(K_i')$) is given by the product of the monomials along the $E_i$-paths of $\D(\be_i)$ (resp. $F_i$-paths of $\D(\bf_i)$).

The iterated coproduct $\D^n(X), X\in \fD_\g$ can be obtained by amalgamating $n+1$ copies of $\cD_\g$ in the same way.
\end{Prop}
\begin{proof} We will consider $\D(\bf_i)$, where the other statements are similar. Recall that 
$$\D(\bf_i)=\bf_i\ox 1+K_i'\ox \bf_i.$$
Then the first half of the $F_i$-path in $\cZ_\g$ is the $F_i$-path in the first $\cD_\g$ quiver, which gives the polynomial $\bf_i\ox 1$. On the other hand, the second half of the $F_i$-path in $\cZ_\g$ is obtained by multiplying the $F_i$-path in the second copy of $\cD_\g$ quiver, and the product of the first half of the $F_i$-path, which by definition represents $K_i'$. Hence combining it gives $K_i'\ox \bf_i$, and hence the concatenation of the $F_i$-path represents $\D(\bf_i)$ as desired.
\end{proof}

%==============================================================================

\subsection{Standard description of the universal $R$-matrix}\label{sec:R:standard}
Recall that the universal $R$-matrix of the quantum group $\cU_q(\g)$ is an element in certain completion of the tensor square 
\Eq{
\cR\in \cU_q(\g)\what{\ox} \cU_q(\g)
}
and it gives the braiding relation:
\Eq{\label{braiding}
\cR\D(X)=\D^{op}(X)\cR,\tab X\in \cU_q(\g)
}
In \cite{Ip4}, a natural expression of $R$ in the split real case is constructed. Given a reduced word $\bi=(i_1,...,i_N)\in\fR$, We have the well-known decomposition
\Eq{
\cR=\cK\ov{R}
}
Here the \emph{Cartan part} is given by
\Eq{
\cK=\prod_{ij} q_i^{(A\inv)_{ij} H_i\ox H_j}
}
where $A$ is the Cartan matrix, and formally we write $K_i=:q_i^{H_i}$. The \emph{reduced $R$-matrix} is given by
\Eq{
\ov{\cR}=\prod_{k=1}^N{}^{op} g_{b_{i_k}}(\be_{\a_k}\ox \bf_{\a_k})\label{Rmat}
}
where $\be_{\a_k}=T_{i_1}T_{i_2}...T_{i_{k-1}}\be_{i_k}$ and similarly for $\bf_{\a_k}$. The product $\Pi^{op}$ is taken with $k=1$ from the right.

Note that (see Remark \ref{qd}) if we write $g_{b}(x)=Exp_{q^{-2}}(-\frac{x}{q-q\inv})$, then \eqref{Rmat} coincides with the well-known formula \cite{KR, KT, LS1, LS2}. Also $\cR$ naturally extends to $\fD_\g$ by replacing $H_i\ox H_j$ in $\cK$ with $-H_i\ox H_j'$ instead, where $K_j'=:q_j^{H_j'}$.

The action of the Cartan part on $\fD_\g$ is easy to describe (see Section \ref{sec:Thm}). In particular, it describes a monomial transformation on the quantum torus algebra $\cX_\bi$, where $X_{f_i^{-n_i}}$ and $X_{f_i^{n_i}}$ on both $\cD_\g$ components of the $Z_\g$-quiver is modified, and this does not change the underlying quiver. Hence we will focus on studying the reduced $R$-matrix, which corresponds to certain quiver mutations.
%==============================================================================

\subsection{First factorization of the reduced $R$-matrix}\label{sec:R:firstfac}
Now we can state our second main result of the paper. Under the embedding $\iota\ox \iota:\fD_\g\ox\fD_\g\to \cD_\g\ox \cD_\g$, we have the following factorization of the reduced $R$-matrix, which generalizes the case of $\cU_q(\sl_2)$ first described by Faddeev \cite{Fa2}, as well as the type $A_n$ case by \cite{SS2}.

\begin{Thm}\label{main2} Let $\bi=(i_1,...,i_N)\in\fR$ be a reduced word. Let us rewrite the embedding of $\bf_i$ from Proposition \ref{Fi} as
\Eqn{
\bf_i&=F_i^{n_i,-}+...+F_i^{1,-}+F_i^{1,+}+...F_i^{n_i,+}\\
&=\sum_{i_k=i} X_k^- + \sum_{i_k=i} X_k^+,\tab k=1,...,N,
}
where 
\Eq{
X_{v(i,k)}^\pm := F_i^{k,\pm}.
}

Then under the embedding $\iota\ox\iota$, the reduced $R$ matrix factorization is given by
\Eq{\label{Rfac}
\ov{R}=&g_{b_{i_N}}(\be_{i_N}\ox X_N^+)...g_{b_{i_2}}(\be_{i_2}\ox X_2^+)g_{b_{i_1}}(\be_{i_1}\ox X_1^+)\cdot\nonumber\\
&g_{b_{i_1}}(\be_{i_1}\ox X_1^-)g_{b_{i_2}}(\be_{i_2}\ox X_2^-)... g_{b_{i_N}}(\be_{i_N}\ox X_N^-).
}
\end{Thm}
We will prove the Theorem in Section \ref{sec:Thm}.

Since $\be_i=\Phi [u]e(-2p)\Phi^*$ for some unitary transformation by \eqref{unitrans}, and
$$[u]e(-2p)=e^{\pi b_i(u-2p)}+e^{\pi b_i(-u-2p)},$$
each $\be_i$ can also be split into
$$\be_i=\be_i^-+\be_i^+$$
such that $$\be_i^-\be_i^+=q_i^{-2} \be_i^+\be_i^-,$$ where $$\be_i^\pm := \Phi e^{\pi b_i(\mp u-2p)} \Phi^*.$$ Then we have
\begin{Cor}\label{CorR1234} Under the embedding $\iota\ox\iota$, the reduced $R$ matrix can also be factorized as
\Eq{\label{R1234}
\ov{R}=R_4\cdot R_3\cdot R_2\cdot R_1
}
where
\Eqn{
R_4=&g_{b_{i_N}}(\be_{i_N}^+\ox X_N^+)... g_{b_{i_2}}(\be_{i_2}^+\ox X_2^+)g_{b_{i_1}}(\be_{i_1}^+\ox X_1^+),\\
R_3=&g_{b_{i_N}}(\be_{i_N}^-\ox X_N^+)... g_{b_{i_2}}(\be_{i_2}^-\ox X_2^+)g_{b_{i_1}}(\be_{i_1}^-\ox X_1^+),\\
R_2=&g_{b_{i_1}}(\be_{i_1}^+\ox X_1^-)g_{b_{i_2}}(\be_{i_2}^+\ox X_2^-)... g_{b_{i_N}}(\be_{i_N}^+\ox X_N^-),\\
R_1=&g_{b_{i_1}}(\be_{i_1}^-\ox X_1^-)g_{b_{i_2}}(\be_{i_2}^-\ox X_2^-)... g_{b_{i_N}}(\be_{i_N}^-\ox X_N^-).
}
\end{Cor}
\begin{proof}
Note that from the remark above, 
$$g_{b_{i_n}}(\be_{i_n}\ox X_n) = g_{b_{i_n}}(\be_{i_n}^+\ox X_n)g_{b_{i_n}}(\be_{i_n}^-\ox X_n).$$
Hence it suffices to show that we can arrange all the $\be_i^+$ to the left hand side of $\be_j^-$ in $R_1$ and $R_2$ (similarly for $R_3$ and $R_4$ to the right). This is equivalent to the statement:
$$\left[\be_{i_n}^-\ox X_n^-, \be_{i_m}^+\ox X_m^+\right] = 0, \tab n>m.$$
T2his follows from Lemma \ref{FikLem} and
$$\be_i^+\be_j^- = q_i^{a_{ij}} \be_j^-\be_i^+$$
by conjugating it to the rank 2 case.
\end{proof}
This simplies the proof of \cite[Theorem 7.4]{SS2} as well as generalizing it to arbitrary type.

%==============================================================================
\subsection{Full factorization of the reduced $R$ matrix}\label{sec:R:fullfac}
In order to realize the $R$ matrix factorization as certain quiver mutation sequences, we have to decompose the terms $g_{b_i}(\be_i^\pm\ox X_N^\pm)$ in the decomposition in Corollary \ref{CorR1234}. In other words, we have to decompose $g_{b_i}(\be_i)$. Then for the monomial terms that we obtain after the decomposition, we compare it with Lemma \ref{useful} in order to obtain the mutation sequence.

\begin{Prop}\label{Endecomp} For every generators $\be_i\in \cU_q(\g)$, consider the explicit embedding given in Section \ref{sec:embedding} for the chosen reduced word $\bi$. Then we can decompose $g_{b_i}(\be_i^\pm)$ into products of the form 
\Eq{
g_{b_i}(\be_i^\pm) = \prod g_b(X_{...}),
} (in type $G_2$ we also need $g_b^*$) where each argument is given by certain cluster monomials $X_{...}$.
\end{Prop}
\begin{proof} It suffices to consider $g_{b_i}(\be_i^-)$, while the decomposition for $g_{b_i}(\be_i^+)$ is just a reflection.
For the generators 
\Eqn{
\case{\be_i& A_n,B_n, D_n,E_n,\\
\be_i,i\neq 1& C_n,\\
\be_4&F_4,\\
\be_2&G_2,
}
}
let us write the embedding of the generators of $\be_i^-$ as a sum of monomials in the form
$$\be_i^-=E_i^0+E_i^1+...+E_i^{m_i},$$
where $E_i^0=X_{f_i^{n_i}}$ and ends before the next term $E_i^{m_i+1} = X_{f_i^{n_i},..., e_i^0}$.

Note that $\be_4^-$ in type $F_4$ and $\be_2^-$ in type $G_2$ is just a monomial, hence the statement is trivial.

For all generators except $\be_1$ in type $B_n$ and $\be_6$ in type $E_8$, we have $E_i^kE_i^l = q^2 E_i^l E_i^k$ for $k>l$, hence we can apply \eqref{guv} to obtain
\Eq{
g_{b_i}(\be_i^-)=g_{b_i}(E_i^{m_i})...g_{b_i}(E_i^1)g_{b_i}(E_i^0).
}
For $\be_1$ in type $B_n$, since $E_1^{2n+1}E_1^{2n}=q^2 E_1^{2n}E_1^{2n+1}$, using \eqref{gb2} we obtain
\Eqn{
g_{b_s}(\be_1^-)=g_{b_s}(E_i^{m_i})...g_{b_s}(E_i^3)g_b(qE_i^2E_i^3)g_{b_s}(E_i^2)g_{b_s}(E_i^1)g_{b}(qE_i^0E_i^1)g_{b_s}(E_i^0).
}
For $\be_6$ in type $E_8$, the path comes in blocks as follows
\Eqn{
\be_6^-=&X_{f_6^9}+(X_{f_6^9,f_7^2}+X_{f_6^9,f_7^2,f_6^6}) + ... (X_{...f_4^{18}}+X_{...f_3^{23}})\\
&+(X_{...f_3^{24}}+X_{...,f_3^{24},f_0^{11}})+(X_{...,f_3^{24},f_2^{15}}+X_{...,f_3^{24},f_2^{15},f_0^{11}})\\
&++(X_{... f_3^{22}}+X_{...f_3^{23}})+... + (X_{...f_6^5}+X_{...f_6^6})+X_{...f_7^1}+X_{...f_6^3}.
}
One can check that each block $q^{-2}$-commutes with all the blocks to the right of it, and within each block the two terms also $q^{-2}$-commute with each other. Hence apply repeatedly \eqref{guv} we arrive at the decomposition of the same form as others.

For the long generators
\Eqn{
\case{\be_1&C_n,\\\be_1,\be_2&F_4,}
}
let us write the embedding of the generators $\be_i^-$ as
\Eqn{
\be_i^-&=E_i^0+E_i^1+...+E_i^k+[2]_{q_s}(q^2E_i^kE_i^{k+1})^{\half}+E_i^{k+1}+...\\
&=E_i^0+E_i^1+...+((E_i^k)^{1/2}+(E_i^{k+1})^{1/2})^2+...
}
whenever we have the double term $(...a,*b,...)$ appear in the $E_i$ path such that 
$$E_i^k:=X_{...,a}, \tab (qE_i^kE_i^{k+1})^{1/2}=X_{...a,b}, \tab E_i^{k+1}:=X_{...,a,b^2}.$$

Then each block $q^{-2}$-commutes with the terms on the right, and since $E_i^{k+1}E_i^k=q^4 E_i^kE_i^{k+1}$, by \eqref{gb2}, we have
$$g_b(\be_i^-)=...g_b(E_i^{k+1})g_{b_s}((qE_i^kE_i^{k+1})^{\half})g_b(E_i^k)...g_b(E_i^1)g_b(E_i^0)$$
whenever the double term appears in the $E_i$-path.

The remaining two special cases are as follows:
The generator $\be_3$ of type $F_4$ is given by:
\Eqn{
\be_3^-&=X_{f_3^9}+(X_{f_3^9,f_4^2}+X_{f_3^9,f_4^2,f_3^6})+(X_{...f_3^7}+X_{...f_2^6})\\
&+(X_{...f_3^5}+X_{...f_3^6})+((X_{...f_4^1}+X_{...f_3^3}))+((X_{...f_2^3}+X_{...f_3^2}))+X_{...f_2^1}+X_{...f_3^1},
}
where each blocks $q_s^{-2}$ commute with all the blocks to the right of it. Within each block, the terms inside the single brackets $q_s^{-2}$ commute, while the terms inside double brackets $q^{-2}$ commute. Hence by \eqref{guv} and \eqref{gb2} we can decompose $g_{b_s}(\be_3^-)$.

For the generator $\be_2$ in type $G_2$, we have to involve conjugations, which gives
\begin{tiny}
\Eqn{
&g_{b_s}(\be_2^-)\\
&=g_{b_s}(X_{f_2^2,f_1^2,(f_2^2)^2,f_1^1,f_2^1}+X_{f_2^2,f_1^2,(f_2^2)^2,f_1^1}+X_{f_2^3,f_1^2,(f_2^2)^2}+[2]_{q_s}X_{f_2^3,f_1^2,f_2^2}+X_{f_2^3,f_1^2}+X_{f_2^3})\\
&=g_{b_s}(X_{f_2^2,f_1^2,(f_2^2)^2,f_1^1,f_2^1})g_{b_s}(X_{f_2^2,f_1^2,(f_2^2)^2,f_1^1}+X_{f_2^3,f_1^2,(f_2^2)^2}+[2]_{q_s}X_{f_2^3,f_1^2,f_2^2}+X_{f_2^3,f_1^2}+X_{f_2^3})\\
&=g_{b_s}(X_{f_2^2,f_1^2,(f_2^2)^2,f_1^1,f_2^1})g_b^*(X_{f_1^1})g_{b_s}(X_{f_2^3,f_1^2,(f_2^2)^2}+[2]_{q_s}X_{f_2^3,f_1^2,f_2^2}+X_{f_2^3,f_1^2}+X_{f_2^3,f_1^2,f_1^1}+X_{f_2^3})g_b(X_{f_1^1})\\
&=g_{b_s}(X_{f_2^2,f_1^2,(f_2^2)^2,f_1^1,f_2^1})g_b^*(X_{f_1^1})g_b^*(X_{f_1^1,f_1^2})g_{b_s}(X_{f_2^3,f_1^2,(f_2^2)^2}+[2]_{q_s}X_{f_2^3,f_1^2,f_2^2}+X_{f_2^3,f_1^2}+X_{f_2^3})g_b(X_{f_1^1,f_1^2})g_b(X_{f_1^1})\\
&=g_{b_s}(X_{f_2^2,f_1^2,(f_2^2)^2,f_1^1,f_2^1})g_b^*(X_{f_1^1})g_b^*(X_{f_1^1,f_1^2})g_{b_s}^*(X_{f_2^2})g_{b_s}(X_{f_2^3,f_1^2}+X_{f_2^3}+X_{f_2^3,f_2^2})g_{b_s}(X_{f_2^2})g_b(X_{f_1^1,f_1^2})g_b(X_{f_1^1})\\
&=g_{b_s}(X_{f_2^2,f_1^2,(f_2^2)^2,f_1^1,f_2^1})g_b^*(X_{f_1^1})g_b^*(X_{f_1^1,f_1^2})g_{b_s}^*(X_{f_2^2})g_b^*(X_{f_1^2})g_{b_s}(X_{f_2^3}+X_{f_2^3,f_2^2})g_b(X_{f_1^2})g_{b_s}(X_{f_2^2})g_b(X_{f_1^1,f_1^2})g_b(X_{f_1^1})\\
&=g_{b_s}(X_{f_2^2,f_1^2,(f_2^2)^2,f_1^1,f_2^1})g_b^*(X_{f_1^1})g_b^*(X_{f_1^1,f_1^2})g_{b_s}^*(X_{f_2^2})g_b^*(X_{f_1^2})g_{b_s}(X_{f_2^3})g_{b_s}(X_{f_2^3,f_2^2})g_b(X_{f_1^2})g_{b_s}(X_{f_2^2})g_b(X_{f_1^1,f_1^2})g_b(X_{f_1^1}).
}\end{tiny}\end{proof}
%==============================================================================

\section{Universal $R$-matrix as half-Dehn twist}\label{sec:Dehn}

Finally we state the final main result of the paper. Consider the $\cZ_\g$-quiver associated to the triangulation of the disk with two marked points $A,C$ and two punctures $B,D$ as before, where the basic quiver $Q$ and its mirror image $\til{Q}$ are put onto the triangles according to Section \ref{sec:basicquiver}, and we label the nodes using the standard form. Let $P$ be the permutation 
\Eq{
P(X\ox Y):=Y\ox X.} 
Note that $P\circ Ad_\cR$ acts as identity on the coalgebra structure, hence it naturally corresponds to an automorphism of seed $\bi\to \bi$. 

\begin{Thm}\label{main3}We have
\Eq{
P\circ Ad_{\cR}=(\mu_{i_1}^q...\mu_{i_T}^q\circ \s^*)\inv = (\s^*)\inv\circ \mu_{i_T}^q...\mu_{i_1}^q
}
for some mutation sequence $\mu_{i_T}...\mu_{i_1}:\bi\to \bi'$ realizing the half-Dehn twist, and $\s:\bi'\simeq \bi$ is a permutation of the quiver returning to the original seed.

More precisely, recall from Lemma \ref{useful} that
$$\mu_{i_1}^q...\mu_{i_T}^q=\Phi_T \circ M_T.$$ 
Then we have
\Eq{
Ad_{\ov{R}}&=\Phi_T\inv,\\
P\circ Ad_{\cK}&=(\s^*)\inv\circ M_T\inv.
}
The factors $R_1,R_2,R_3,R_4$ in \eqref{R1234} correspond to the sequences of quiver mutations realizing the 4 flips of triangulations giving the half-Dehn twist as follows:

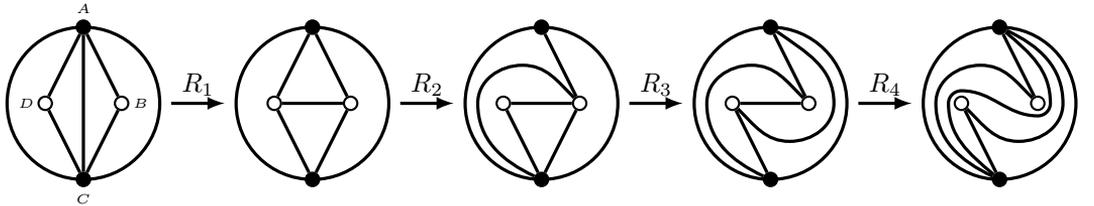
\begin{figure}[H]
\centering
\begin{tikzpicture}[every node/.style={thick,circle,inner sep=0pt}, x=0.5cm,y=0.5cm]
\foreach \y in {0,...,4}{
	\draw [very thick] (6*\y,0) circle (2);
	\node (a-\y) at (6*\y,2) [draw,minimum size=5, fill=black]{};
	\node (b-\y) at (6*\y-1,0) [draw,minimum size=5]{};
	\node (c-\y) at (6*\y+1,0) [draw,minimum size=5]{};
	\node (d-\y) at (6*\y,-2) [draw,minimum size=5, fill=black]{};
	\path (a-\y) edge[-,very thick] (c-\y);
	\path (b-\y) edge[-,very thick] (d-\y);
	\ifthenelse{\NOT \y=0}{
	\path (6*\y-3.7,0) edge[->, very thick] (6*\y-2.3, 0);
	\node at (6*\y-3, 0.5) {$R_\y$};
	}{};
}
\node at (0,2.5) {\tiny $A$};
\node at (1.5,0) {\tiny $B$};
\node at (0,-2.5) {\tiny $C$};
\node at (-1.5,0) {\tiny $D$};
\path (a-0) edge[-, very thick] (b-0);
\path (d-0) edge[-, very thick] (c-0);
\path (a-0) edge[-, very thick] (d-0);
\path (a-1) edge[-, very thick] (b-1);
\path (d-1) edge[-, very thick] (c-1);
\path (b-1) edge[-, very thick] (c-1);
\draw[very thick] (c-2) to [out=135, in=0] (11.5,1) to [out=180, in=90] (10.33,0) to [out=-90, in=155] (d-2);
\path (d-2) edge[-, very thick] (c-2);
\path (b-2) edge[-, very thick] (c-2);
\draw[very thick] (c-3) to [out=135, in=0] (17.5,1) to [out=180, in=90] (16.33,0) to [out=-90, in=155] (d-3);
\draw[very thick] (b-3) to [out=-45, in=180] (18.5,-1) to [out=0, in=-90] (19.66,0) to [out=90, in=-35] (a-3);
\path (b-3) edge[-, very thick] (c-3);
\draw[very thick] (c-4) to [out=135, in=0] (23.5,1) to [out=180, in=90] (22.33,0) to [out=-90, in=155] (d-4);
\draw[very thick] (b-4) to [out=-45, in=180] (24.5,-1) to [out=0, in=-90] (25.66,0) to [out=90, in=-25] (a-4);
\draw[very thick] (a-4) to [out=-40, in=90] (25.33,0) to [out=-90, in=0] (25,-0.33) to [out=180, in=0] (23,0.33) to [out=180, in=90] (22.66,0) to [out=-90, in=140] (d-4);
\end{tikzpicture}
\caption{Half-Dehn twist.}
\end{figure}
\end{Thm}
\begin{Rem} We note that the mutation sequence is not unique, for example, using \eqref{gvu} one can replace 2 $g_b$'s with 3 $g_b$'s, thus giving the same mutation (with different permutation index at the end) but with a longer sequence.
\end{Rem}

In terms of the quiver associated to the triangulations, the 4 flips are realized as follows. Let us write $\mu_{R_i}$ for the sequence of quiver mutations (starting from a standard form) corresponding to $R_i$, and $\s_i$ the permutation bringing the labeling of the basic quivers back to the standard form. Then we have the following configuration:

\begin{subfigures}
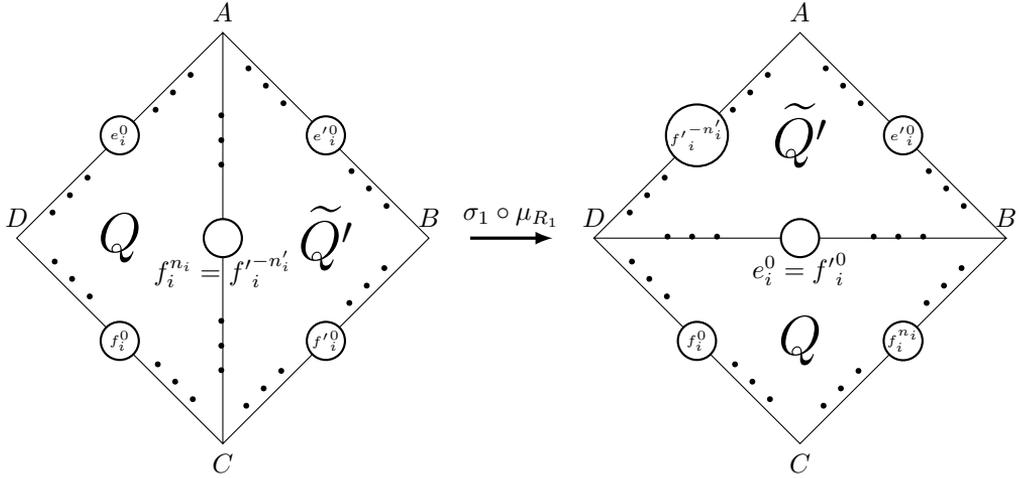
\begin{figure}[H]
\centering
\begin{tikzpicture}[every node/.style={inner sep=0, minimum size=0.5cm, thick},x=0.27cm,y=0.27cm]
\begin{scope}[shift={(-14,0)}]
\draw (0,10)node[anchor=south]{$A$}  
-- (-10,0) node[anchor=south]{$D$} 
-- (0,-10) node[anchor=north]{$C$} 
-- (10,0) node[anchor=south]{$B$}
--(0,10)--(0,-10);
\node at (0,0) [draw, circle, fill=white, thick] {};
\node at (0,-1.5) {$f_i^{n_i}={f'}_i^{-n_i'}$};
\node at (-5,5) [draw, circle, fill=white, thick] {\tiny$e_i^0$};
\node at (-5,-5) [draw, circle, fill=white, thick] {\tiny$f_i^0$};
\node at (5,5) [draw, circle, fill=white, thick] {\tiny${e'}_i^0$};
\node at (5,-5) [draw, circle, fill=white, thick] {\tiny${f'}_i^0$};
\node[rotate=45] at (-2.2,7.2) {\huge$\cdots$};
\node[rotate=45] at (-7.2, 2.2) {\huge$\cdots$};
\node[rotate=-45] at (-2.2,-7.2) {\huge$\cdots$};
\node[rotate=-45] at (-7.2,-2.2) {\huge$\cdots$};
\node[rotate=90] at (0,5) {\huge$\cdots$};
\node[rotate=90] at (0,-5) {\huge$\cdots$};
\node[rotate=45] at (2.2,-7.2) {\huge$\cdots$};
\node[rotate=45] at (7.2,-2.2) {\huge$\cdots$};
\node[rotate=-45]  at (2.2,7.2) {\huge$\cdots$};
\node[rotate=-45]  at (7.2,2.2) {\huge$\cdots$};
\node (Q) at (-5,0) [minimum size=2cm]{\huge $Q$};
\node (Q) at (5,0) [minimum size=2cm]{\huge $\til{Q'}$};
\end{scope}
\begin{scope}[shift={(14,0)}]
\draw (10,0) node[anchor=south]{$B$}
-- (0,-10) node[anchor=north]{$C$} 
-- (-10,0) node[anchor=south]{$D$} 
--(0,10) node[anchor=south]{$A$}  
--(10,0)--(-10, 0);
\node at (0,0) [draw, circle, fill=white, thick] {};
\node at (0,-1.5) {$e_i^0={f'}_i^0$};
\node at (-5,5) [draw, circle, fill=white, thick] {\tiny${f'}_i^{-n_i'}$};
\node at (-5,-5) [draw, circle, fill=white, thick] {\tiny$f_i^0$};
\node at (5,5) [draw, circle, fill=white, thick] {\tiny${e'}_i^0$};
\node at (5,-5) [draw, circle, fill=white, thick] {\tiny$f_i^{n_i}$};
\node[rotate=45] at (-2.2,7.2) {\huge$\cdots$};
\node[rotate=45] at (-7.2, 2.2) {\huge$\cdots$};
\node[rotate=-45] at (-2.2,-7.2) {\huge$\cdots$};
\node[rotate=-45] at (-7.2,-2.2) {\huge$\cdots$};
\node  at (5,0) {\huge$\cdots$};
\node  at (-5,0) {\huge$\cdots$};
\node[rotate=45] at (2.2,-7.2) {\huge$\cdots$};
\node[rotate=45] at (7.2,-2.2) {\huge$\cdots$};
\node[rotate=-45]  at (2.2,7.2) {\huge$\cdots$};
\node[rotate=-45]  at (7.2,2.2) {\huge$\cdots$};
\node (Q) at (0,-5) [minimum size=2cm]{\huge $Q$};
\node (Q) at (0,5) [minimum size=2cm]{\huge $\til{Q'}$};
\end{scope}
\path (-2,0) edge[->, very thick] (2,0);
\node at (0,1) {$\s_1\circ \mu_{R_1}$};
\end{tikzpicture}
\caption{Flip of triangulations corresponding to $R_1$.}
\label{mutateR1}
\end{figure}
\begin{figure}[H]
\centering
\begin{tikzpicture}[every node/.style={inner sep=0, minimum size=0.5cm, thick},x=0.27cm,y=0.27cm]
\begin{scope}[shift={(-14,0)}]
\draw (0,10)node[anchor=south]{$A$}  
-- (-10,0) node[anchor=south]{$C$} 
-- (0,-10) node[anchor=north]{$D$} 
-- (10,0) node[anchor=south]{$B$}
--(0,10)--(0,-10);
\node at (0,0) [draw, circle, fill=white, thick] {};
\node at (0,-1.5) {$e_i^0={f'}_i^{-n_i'}$};
\node at (-5,5) [draw, fill=white, thick] {\tiny$f_i^{-n_i}$};
\node at (-5,-5) [draw, circle, fill=white, thick] {\tiny$f_i^0$};
\node at (5,5) [draw, circle, fill=white, thick] {\tiny${e'}_i^0$};
\node at (5,-5) [draw, circle, fill=white, thick] {\tiny${f'}_i^0$};
\node[rotate=45] at (-2.2,7.2) {\huge$\cdots$};
\node[rotate=45] at (-7.2, 2.2) {\huge$\cdots$};
\node[rotate=-45] at (-2.2,-7.2) {\huge$\cdots$};
\node[rotate=-45] at (-7.2,-2.2) {\huge$\cdots$};
\node[rotate=90] at (0,5) {\huge$\cdots$};
\node[rotate=90] at (0,-5) {\huge$\cdots$};
\node[rotate=45] at (2.2,-7.2) {\huge$\cdots$};
\node[rotate=45] at (7.2,-2.2) {\huge$\cdots$};
\node[rotate=-45]  at (2.2,7.2) {\huge$\cdots$};
\node[rotate=-45]  at (7.2,2.2) {\huge$\cdots$};
\node (Q) at (-5,0) [minimum size=2cm]{\huge $\til{Q}$};
\node (Q) at (5,0) [minimum size=2cm]{\huge $\til{Q'}$};
\end{scope}
\begin{scope}[shift={(14,0)}]
\draw (10,0) node[anchor=south]{$B$}
-- (0,-10) node[anchor=north]{$D$} 
-- (-10,0) node[anchor=south]{$C$} 
--(0,10) node[anchor=south]{$A$}  
--(10,0)--(-10, 0);
\node at (0,0) [draw, circle, fill=white, thick] {};
\node at (0,-1.5) {$f_i^{-n_i}={f'}_i^0$};
\node at (-5,5) [draw, fill=white, thick] {\tiny${f'}_i^{-n_i'}$};
\node at (-5,-5) [draw, circle, fill=white, thick] {\tiny$f_i^0$};
\node at (5,5) [draw, circle, fill=white, thick] {\tiny${e'}_i^0$};
\node at (5,-5) [draw, circle, fill=white, thick] {\tiny$e_i^0$};
\node[rotate=45] at (-2.2,7.2) {\huge$\cdots$};
\node[rotate=45] at (-7.2, 2.2) {\huge$\cdots$};
\node[rotate=-45] at (-2.2,-7.2) {\huge$\cdots$};
\node[rotate=-45] at (-7.2,-2.2) {\huge$\cdots$};
\node  at (5,0) {\huge$\cdots$};
\node  at (-5,0) {\huge$\cdots$};
\node[rotate=45] at (2.2,-7.2) {\huge$\cdots$};
\node[rotate=45] at (7.2,-2.2) {\huge$\cdots$};
\node[rotate=-45]  at (2.2,7.2) {\huge$\cdots$};
\node[rotate=-45]  at (7.2,2.2) {\huge$\cdots$};
\node (Q) at (0,-5) [minimum size=2cm]{\huge $\til{Q}$};
\node (Q) at (0,5) [minimum size=2cm]{\huge $\til{Q'}$};
\end{scope}
\path (-2,0) edge[->, very thick] (2,0);
\node at (0,1) {$\s_2\circ \mu_{R_2}$};
\end{tikzpicture}
\caption{Flip of triangulations corresponding to $R_2$.}
\label{mutateR2}
\end{figure}
\begin{figure}[H]
\centering
\begin{tikzpicture}[every node/.style={inner sep=0, minimum size=0.5cm, thick},x=0.27cm,y=0.27cm]
\begin{scope}[shift={(-14,0)}]
\draw (0,10)node[anchor=south]{$B$}  
-- (-10,0) node[anchor=south]{$D$} 
-- (0,-10) node[anchor=north]{$C$} 
-- (10,0) node[anchor=south]{$A$}
--(0,10)--(0,-10);
\node at (0,0) [draw, circle, fill=white, thick] {};
\node at (0,-1.5) {$f_i^{n_i}={f'}_i^0$};
\node at (-5,5) [draw, circle, fill=white, thick] {\tiny$e_i^0$};
\node at (-5,-5) [draw, circle, fill=white, thick] {\tiny$f_i^0$};
\node at (5,5) [draw, circle, fill=white, thick] {\tiny${e'}_i^0$};
\node at (5,-5) [draw, fill=white, thick] {\tiny${f'}_i^{n_i'}$};
\node[rotate=45] at (-2.2,7.2) {\huge$\cdots$};
\node[rotate=45] at (-7.2, 2.2) {\huge$\cdots$};
\node[rotate=-45] at (-2.2,-7.2) {\huge$\cdots$};
\node[rotate=-45] at (-7.2,-2.2) {\huge$\cdots$};
\node[rotate=90] at (0,5) {\huge$\cdots$};
\node[rotate=90] at (0,-5) {\huge$\cdots$};
\node[rotate=45] at (2.2,-7.2) {\huge$\cdots$};
\node[rotate=45] at (7.2,-2.2) {\huge$\cdots$};
\node[rotate=-45]  at (2.2,7.2) {\huge$\cdots$};
\node[rotate=-45]  at (7.2,2.2) {\huge$\cdots$};
\node (Q) at (-5,0) [minimum size=2cm]{\huge $Q$};
\node (Q) at (5,0) [minimum size=2cm]{\huge $Q'$};
\end{scope}
\begin{scope}[shift={(14,0)}]
\draw (10,0) node[anchor=south]{$A$}
-- (0,-10) node[anchor=north]{$C$} 
-- (-10,0) node[anchor=south]{$D$} 
--(0,10) node[anchor=south]{$B$}  
--(10,0)--(-10, 0);
\node at (0,0) [draw, circle, fill=white, thick] {};
\node at (0,-1.5) {$e_i^0={f'}_i^{n_i'}$};
\node at (-5,5) [draw, circle, fill=white, thick] {\tiny${f'}_i^0$};
\node at (-5,-5) [draw, circle, fill=white, thick] {\tiny$f_i^0$};
\node at (5,5) [draw, circle, fill=white, thick] {\tiny${e'}_i^0$};
\node at (5,-5) [draw, fill=white, thick] {\tiny$f_i^{n_i}$};
\node[rotate=45] at (-2.2,7.2) {\huge$\cdots$};
\node[rotate=45] at (-7.2, 2.2) {\huge$\cdots$};
\node[rotate=-45] at (-2.2,-7.2) {\huge$\cdots$};
\node[rotate=-45] at (-7.2,-2.2) {\huge$\cdots$};
\node  at (5,0) {\huge$\cdots$};
\node  at (-5,0) {\huge$\cdots$};
\node[rotate=45] at (2.2,-7.2) {\huge$\cdots$};
\node[rotate=45] at (7.2,-2.2) {\huge$\cdots$};
\node[rotate=-45]  at (2.2,7.2) {\huge$\cdots$};
\node[rotate=-45]  at (7.2,2.2) {\huge$\cdots$};
\node (Q) at (0,-5) [minimum size=2cm]{\huge $Q$};
\node (Q) at (0,5) [minimum size=2cm]{\huge $Q'$};
\end{scope}
\path (-2,0) edge[->, very thick] (2,0);
\node at (0,1) {$\s_3\circ \mu_{R_3}$};
\end{tikzpicture}
\caption{Flip of triangulations corresponding to $R_3$.}
\label{mutateR3}
\end{figure}
\begin{figure}[H]
\centering
\begin{tikzpicture}[every node/.style={inner sep=0, minimum size=0.5cm, thick},x=0.27cm,y=0.27cm]
\begin{scope}[shift={(-14,0)}]
\draw (0,10)node[anchor=south]{$B$}  
-- (-10,0) node[anchor=south]{$C$} 
-- (0,-10) node[anchor=north]{$D$} 
-- (10,0) node[anchor=south]{$A$}
--(0,10)--(0,-10);
\node at (0,0) [draw, circle, fill=white, thick] {};
\node at (0,-1.5) {$e_i^0={f'}_i^0$};
\node at (-5,5) [draw, circle, fill=white, thick] {\tiny$f_i^{-n_i}$};
\node at (-5,-5) [draw, circle, fill=white, thick] {\tiny$f_i^0$};
\node at (5,5) [draw, circle, fill=white, thick] {\tiny${e'}_i^0$};
\node at (5,-5) [draw, circle, fill=white, thick] {\tiny${f'}_i^{n_i'}$};
\node[rotate=45] at (-2.2,7.2) {\huge$\cdots$};
\node[rotate=45] at (-7.2, 2.2) {\huge$\cdots$};
\node[rotate=-45] at (-2.2,-7.2) {\huge$\cdots$};
\node[rotate=-45] at (-7.2,-2.2) {\huge$\cdots$};
\node[rotate=90] at (0,5) {\huge$\cdots$};
\node[rotate=90] at (0,-5) {\huge$\cdots$};
\node[rotate=45] at (2.2,-7.2) {\huge$\cdots$};
\node[rotate=45] at (7.2,-2.2) {\huge$\cdots$};
\node[rotate=-45]  at (2.2,7.2) {\huge$\cdots$};
\node[rotate=-45]  at (7.2,2.2) {\huge$\cdots$};
\node (Q) at (-5,0) [minimum size=2cm]{\huge $\til{Q}$};
\node (Q) at (5,0) [minimum size=2cm]{\huge $Q'$};
\end{scope}
\begin{scope}[shift={(14,0)}]
\draw (10,0) node[anchor=south]{$A$}
-- (0,-10) node[anchor=north]{$D$} 
-- (-10,0) node[anchor=south]{$C$} 
--(0,10) node[anchor=south]{$B$}  
--(10,0)--(-10, 0);
\node at (0,0) [draw, circle, fill=white, thick] {};
\node at (0,-1.5) {$f_i^{-n_i}={f'}_i^{n_i'}$};
\node at (-5,5) [draw, circle, fill=white, thick] {\tiny${f'}_i^0$};
\node at (-5,-5) [draw, circle, fill=white, thick] {\tiny$f_i^0$};
\node at (5,5) [draw, circle, fill=white, thick] {\tiny${e'}_i^0$};
\node at (5,-5) [draw, circle, fill=white, thick] {\tiny$e_i^0$};
\node[rotate=45] at (-2.2,7.2) {\huge$\cdots$};
\node[rotate=45] at (-7.2, 2.2) {\huge$\cdots$};
\node[rotate=-45] at (-2.2,-7.2) {\huge$\cdots$};
\node[rotate=-45] at (-7.2,-2.2) {\huge$\cdots$};
\node  at (5,0) {\huge$\cdots$};
\node  at (-5,0) {\huge$\cdots$};
\node[rotate=45] at (2.2,-7.2) {\huge$\cdots$};
\node[rotate=45] at (7.2,-2.2) {\huge$\cdots$};
\node[rotate=-45]  at (2.2,7.2) {\huge$\cdots$};
\node[rotate=-45]  at (7.2,2.2) {\huge$\cdots$};
\node (Q) at (0,-5) [minimum size=2cm]{\huge $\til{Q}$};
\node (Q) at (0,5) [minimum size=2cm]{\huge $Q'$};
\end{scope}
\path (-2,0) edge[->, very thick] (2,0);
\node at (0,1) {$\s_4\circ \mu_{R_4}$};
\end{tikzpicture}
\caption{Flip of triangulations corresponding to $R_4$.}
\label{mutateR4}
\end{figure}
\end{subfigures}

After the 4 flips, the quiver comes back to the original configuration with $B\corr D$, $Q\corr Q', \til{Q}\corr \til{Q'}$, and we have 
\Eq{
\s=\s_4\circ \s_3\circ \s_2\circ \s_1.
}
\begin{figure}[!htb]
\centering
\begin{tikzpicture}[every node/.style={thick,circle,inner sep=0pt}, x=1cm,y=1cm]
\foreach \y in {0,1}{
	\draw [very thick] (6*\y,0) circle (2);
	\node (a-\y) at (6*\y,2) [draw,minimum size=5, fill=black]{};
	\node (b-\y) at (6*\y+1,0) [draw,minimum size=5]{};
	\node (c-\y) at (6*\y,-2) [draw,minimum size=5, fill=black]{};
	\node (d-\y) at (6*\y-1,0) [draw,minimum size=5]{};
	\path (a-\y) edge[-, very thick] (b-\y);
	\path (b-\y) edge[-, very thick] (c-\y);
	\path (c-\y) edge[-, very thick] (d-\y);
	\path (d-\y) edge[-, very thick] (a-\y);
	\path (a-\y) edge[-, very thick] (c-\y);					
}
\node at (0,2.3){\tiny$A$};
\node at (1,0.3){\tiny$B$};
\node at (0,-2.3){\tiny$C$};
\node at (-1,0.3){\tiny$D$};
\node at (-1.5,0){$\til{Q}$};
\node at (-0.5,0){$Q$};
\node at (0.5,0){$\til{Q'}$};
\node at (1.5,0){$Q'$};
\node at (6,2.3){\tiny$A$};
\node at (7,0.3){\tiny$D$};
\node at (6,-2.3){\tiny$C$};
\node at (5,0.3){\tiny$B$};
\node at (4.5,0){$\til{Q'}$};
\node at (5.5,0){$Q'$};
\node at (6.5,0){$\til{Q}$};
\node at (7.5,0){$Q$};
\path (2.5,0) edge[->, very thick] (3.5,0);
\end{tikzpicture}
\caption{Half-Dehn twist of the quiver $\cZ_\g$.}
\end{figure}
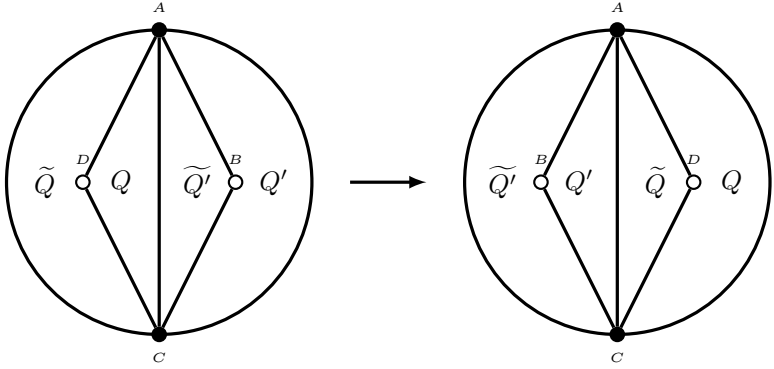
We observe that for each flip, one can think of the quiver mutation as rotating both basic quivers clockwise by 30 degree, and then stack the right quiver on top of the left one. In the next subsection, we will show how to obtain such mutation sequence.

%==============================================================================
\subsection{Explicit mutation sequence for the Half-Dehn twist}\label{sec:Dehn:explicit}
By the symmetry of the decomposition \eqref{R1234} as well as the mutation configurations, we can see that $R_2$ and $R_3$ commute, and the mutation sequence corresponding to $R_4$ and $R_3$ in some sense are just ``mirror images" to those of $R_1$ and $R_2$ respectively. Using the explicit decomposition from Proposition \ref{Endecomp}, we arrive at the following more precise description of the quiver mutation giving the half-Dehn twist:
\begin{Prop}\label{mirrorpalin} The mutation sequence is a mirrored palindrome:
\Eqn{
\mu_{R_4}&=(\rho\s_1)_* (\mu_{R_1}\inv),& \s_4&=\rho_*(\s_1\inv),\\
\mu_{R_3}&=(\rho\s_2)_* (\mu_{R_2}\inv),& \s_3&=\rho_*(\s_2\inv),
}
where $\rho$ is the permutation given by the reflection $f_i^k\corr f_i^{-k}$:
\Eq{
\rho: \{f_i^k, e_i^0, {f'}_i^k, {e'}_i^0\} \corr \{ f_i^{-k}, e_i^0, {f'}_i^{-k}, {e'}_i^0\},
}
and for a permutation $\pi$ we denote by $\pi_*(X):=\pi \circ X \circ \pi\inv$.
\end{Prop}
 
Hence below we will only study the mutation sequence corresponding to $R_1$ and $R_2$. 

To describe the mutation sequence, let us define the following notation:
\begin{Def}\label{concatseq}
Let $S=(s_0,...,s_n),T=(t_0,...,t_m)$ be two sequences (of the nodes of some quiver). If $s_n$ and $t_1$ denote the same node in the quiver, then we define a new sequence of length $n+m+1$:
$$\<S- T\>:=(s_0,...,s_n=t_0,...,t_m)$$
to be the concatenation of the two sequences, and it is indexed from $-n$ to $m$ such that 
$$\<S-T\>_0= s_n = t_0.$$
If $\cP$ is a sequence constructed in this way, then we define its flip to be
$$\cF(\cP):=\<T-S\>$$
whenever $t_m=s_0$ in some other quiver in which this sequence is indexing.
\end{Def}
\begin{Def}
If $$\mu_T:=\mu_{j_M}...\mu_{j_1}$$ is a mutation sequence, we alternatively write it as
$$\mu_T=:\{j_1\to j_2\to...\to j_M\}.$$
Then given a sequence $\cP$, we denote the \emph{$k$-shifted mutation subsequence} of length $m$ by
\Eqn{
\cP[k,m]:=\{\cP_{1-k}\to \cP_{2-k}\to...\to\cP_{m-k}\}.
}
\end{Def}
\begin{Def}
We define the sequences
\Eqn{
\cP_{E_i}^{Q}&:=(f_i^{n_i},..., e_i^0),\\
\cP_{E_i}^{\til{Q}}&:=(e_i^0,..., f_i^{-n_i}),\\
\cP_{F_i}^{Q}&:=(f_i^{n_i},..., f_i^0),\\
\cP_{F_i}^{\til{Q}}&:=(f_i^0,..., f_i^{-n_i})
}
to be the $E_i$ and (reverse of) $F_i$ paths of the quiver $Q$ and $\til{Q}$ respectively. Similarly we use ${}'$ to denote the corresponding paths in the second quiver $1\ox \cD_\g\subset\cZ_\g$.
\end{Def}

Finally, given a reduced word $\bi$, we denote by $\bi'$ the reversed word, and recall that (cf. Definition \ref{variable}) \Eq{v'(i,k)=m} if $i_m$ is the $k$-th appearance of the root index $i$ from the left of $\bi'$, i.e. right of $\bi$.
%==============================================================================

\subsubsection{Toy Example: Type $A_2$}\label{sec:Dehn:A2}

To demonstrate the procedure, let us first look at the toy example in type $A_2$ in detail using the notation of our paper. This has also been worked out in detail in \cite{SS2} with slightly different notations.
\begin{figure}[!htb]
\centering
\begin{tikzpicture}[every node/.style={inner sep=0, minimum size=0.5cm, thick}, x=1.5cm, y=1cm]
\node (1) at (-2,4)[draw]{\tiny $f_1^{-2}$};
\node (2) at (-1,3)[draw,circle]{\tiny $f_1^{-1}$};
\node (3) at (0,2)[draw,circle]{\tiny $f_1^0$};
\node (4) at (1,3)[draw,circle]{\tiny $f_1^1$};
\node (5) at (2,4)[draw, circle]{};
\node (6) at (-2,2)[draw]{\tiny $f_2^{-1}$};
\node (7) at (0,0)[draw,circle]{\tiny $f_2^0$};
\node (8) at (2,2)[draw, circle]{};
\node (9) at (0,6)[draw,circle]{\tiny $e_1^0$};
\node (10) at (0,4)[draw,circle]{\tiny $e_2^0$};
\node (1') at (2,3.5){\tiny $f_1^2={f'}_1^{-2}$};
\node (2') at (3,3)[draw,circle]{\tiny ${f'}_1^{-1}$};
\node (3') at (4,2)[draw,circle]{\tiny ${f'}_1^0$};
\node (4') at (5,3)[draw,circle]{\tiny ${f'}_1^1$};
\node (5') at (6,4)[draw]{\tiny ${f'}_1^2$};
\node (6') at (2,1.5){\tiny $f_2^1={f'}_2^{-1}$};
\node (7') at (4,0)[draw,circle]{\tiny ${f'}_2^0$};
\node (8') at (6,2)[draw]{\tiny ${f'}_2^1$};
\node (9') at (4,6)[draw,circle]{\tiny ${e'}_1^0$};
\node (10') at (4,4)[draw,circle]{\tiny ${e'}_2^0$};
\drawpath{1,...,5,9,1}{thick}
\drawpath{7,8,4,10,2,6,7,2,9,4,7}{thick}
\drawpath{5,2',3',4',5',9',5}{thick}
\drawpath{7',8',4',10',2',8,7',2',9',4',7'}{thick}
 \path(6) edge[->, dashed, thick] (1);
 \path(5') edge[->, dashed, thick] (8');
\end{tikzpicture}
\caption{$\cZ_{A_2}$-quiver.}
\end{figure}

Recall the embedding of type $A_2$ from Section \ref{sec:embedding:A2}. First of all, the $g_b(\be_i)$ can be easily decomposed using \eqref{guv} with $g_b(\be_i)=g_b(\be_i^+)g_b(\be_i^-)$ as
\Eqn{
g_b(\be_1^+)&=g_b(X_{f_1^2,e_1^0}),\\
g_b(\be_1^-)&=g_b(X_{f_1^2}),\\
g_b(\be_2^+)&=g_b(X_{f_2^1,f_1^1,e_2^0,f_1^{-1}})g_b(X_{f_2^1,f_1^1,e_2^0}),\\
g_b(\be_2^-)&=g_b(X_{f_2^1,f_1^1})g_b(X_{f_2^1}).
}
Hence by Corollary \ref{CorR1234}, the reduced $R$ matrix decomposed as:
\Eqn{
R_4&=g_b(X_{f_1^2,e_1^0}\ox X_{f_1^{-2},f_1^{-1},f_1^0,f_1^2})g_b(X_{f_2^1,f_1^2,e_2^0,f_1^{-1}}\ox X_{f_2^{-1},f_2^0})\\
&\tab g_b(X_{f_2^1,f_1^2,e_2^0}\ox X_{f_2^{-1},f_2^0})g_b(X_{f_1^2,e_1^0}\ox X_{f_1^{-2},f_1^{-1},f_1^0}).\\
R_3&=g_b(X_{f_1^2}\ox X_{f_1^{-2},f_1^{-1}})g_b(X_{f_2^1,f_1^1}\ox X_{f_2^{-1}})g_b(X_{f_2^1}\ox X_{f_2^{-1}})g_b(X_{f_1^2}\ox X_{f_1^{-2}}).\\
R_2&=g_b(X_{f_1^2,e_1^0}\ox X_{f_1^{-2},f_1^{-1}})g_b(X_{f_2^1,f_1^2,e_2^0,f_1^{-1}}\ox X_{f_2^{-1}})g_b(X_{f_2^1,f_1^2,e_2^0}\ox X_{f_2^{-1}})g_b(X_{f_1^2,e_1^0}\ox X_{f_1^{-2}}).\\
R_1&=g_b(X_{f_1^2}\ox X_{f_1^{-2},f_1^{-1}})g_b(X_{f_2^1,f_1^1}\ox X_{f_2^{-1}})g_b(X_{f_2^1}\ox X_{f_2^{-1}})g_b(X_{f_1^2}\ox X_{f_1^{-2}}).
}
Then we calculate term by term the corresponding mutation sequence (recall that $f_i^{n_i}$ is glued to ${f'}_i^{-n_i}$):
\Eqn{
X_{f_1^2}\ox X_{f_1^{-2}}&\sim \mu_{f_1^2}\\
X_{f_2^1}\ox X_{f_2^{-1}} = \mu'_{f_1^2}(X^{\mu}_{f_2^1}\ox X^{\mu}_{f_2^{-1}}))&\sim \mu_{f_2^1}\\
X_{f_2^1,f_1^1}\ox X_{f_2^{-1}}=\mu'_{f_1^2}\mu'_{f_{2^1}}(X^{\mu^2}_{f_1^1}\ox 1)&\sim \mu_{f_1^1}\\
X_{f_1^2}\ox X_{f_1^{-2},f_1^{-1}}=\mu'_{f_1^2}\mu'_{f_2^1}\mu'_{f_1^1}(1\ox X^{\mu^3}_{f_1^{-1}})&\sim \mu_{{f'}_1^{-1}}\\
\cdots&\sim\cdots
}
and so on, where we denoted by $X^{\mu^n}$ the corresponding mutated quantum cluster variables after $n$ mutations (but we do not change the labels). Then we obtain:
\Eqn{
\mu_{R_1}&=\mu_{{f'}_1^{-1}}\mu_{f_1^1}\mu_{{f'}_2^{-1}}\mu_{{f'}_1^{-2}}, &\s_1&=({f'}_2^0, {f'}_2^{-1}, f_1^1, e_2^0)(e_1^0, {f'}_1^{-2}, {f'}_1^{-1}, {f'}_1^0),\\
\mu_{R_2}&=\mu_{{f'}_1^{-1}}\mu_{f_1^{-1}}\mu_{{f'}_2^{-1}}\mu_{{f'}_1^{-2}}, &\s_2&=({f'}_2^0,{f'}_2^{-1}, {f}_1^{-1}, f_2^{-1})(f_1^{-2},{f'}_1^{-2},{f'}_1^{-1},{f'}_1^0),\\
\mu_{R_3}&=\mu_{{f'}_1^1}\mu_{f_1^1}\mu_{{f'}_2^0}\mu_{{f'}_1^0},&\s_3&=({f'}_2^1,{f'}_2^0,f_1^1,e_2^0)(e_1^0,{f'}_1^0,{f'}_1^1,{f'}_1^2),\\
\mu_{R_4}&=\mu_{{f'}_1^1}\mu_{f_1^{-1}}\mu_{{f'}_2^0}\mu_{{f'}_1^0},&\s_4&=({f'}_2^1,{f'}_2^0,f_1^{-1},f_2^{-1})(f_1^{-2},{f'}_1^0,{f'}_1^1,{f'}_1^2).
}
Note that $\s_i$ are given by shifting along the concatenation of the $F_i$ path in the right quiver and $E_i$ path in the left quiver, and that the mutation corresponding to $R_3$ and $R_4$ are the mirror reflections of $R_2$ and $R_1$ satisfying Proposition \ref{mirrorpalin}. We display the configurations in Figure \ref{muRi}, omitting $R_3$ and $R_4$. Also recall that $Q=\til{Q}$ in type $A_n$ due to the $S_3$ symmetry, hence in fact under this identification all 4 flips are identical.

\begin{subfigures}
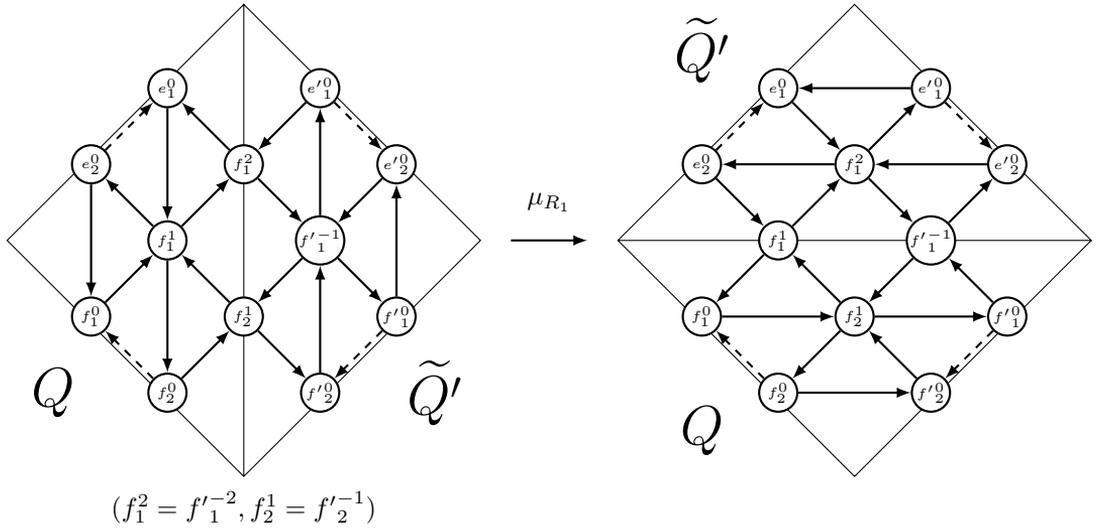
\begin{figure}[H]
\centering
\begin{tikzpicture}[every node/.style={circle, fill=white,inner sep=0, minimum size=0.5cm, thick}, x=1cm, y=1cm]
\begin{scope}[shift={(-4,0)}]
\draw (0,3.1)-- (-3.1,0)-- (0,-3.1)-- (3.1,0)--(0,3.1)--(0,-3.1);
\node (1) at (-2,1)[draw]{\tiny$e_2^0$};
\node (2) at (-2,-1)[draw]{\tiny$f_1^0$};
\node (3) at (-1,2)[draw]{\tiny$e_1^0$};
\node (4) at (-1,0)[draw]{\tiny$f_1^1$};
\node (5) at (-1,-2)[draw]{\tiny$f_2^0$};
\node (6) at (0,1)[draw]{\tiny$f_1^2$};
\node (7) at (0,-1)[draw]{\tiny $f_2^1$};
\node (8) at (1,2)[draw]{\tiny${e'}_1^0$};
\node (9) at (1,0)[draw]{\tiny${f'}_1^{-1}$};
\node (10) at (1,-2)[draw]{\tiny${f'}_2^0$};
\node (11) at (2,1)[draw]{\tiny${e'}_2^0$};
\node (12) at (2,-1)[draw]{\tiny${f'}_1^0$};
\drawpath{2,4,6,9,12,11,9,7,4,1,2}{thick}
\drawpath{6,3,4,5,7,10,9,8,6}{thick}
\drawpath{1,3}{dashed,thick}
\drawpath{5,2}{dashed,thick}
\drawpath{8,11}{dashed,thick}
\drawpath{12,10}{dashed,thick}
\node at (-2.5,-2) {\huge $Q$};
\node at (2.5,-2) {\huge $\til{Q'}$};
\node at (0,-3.5) [fill=none, rectangle]{$(f_1^2={f'}_1^{-2}, f_2^1={f'}_2^{-1})$};
\end{scope}
\begin{scope}[shift={(4,0)}]
\draw (3.1,0)-- (0,-3.1)-- (-3.1,0)--(0,3.1)--(3.1,0)--(-3.1, 0);
\node (1) at (-2,1)[draw]{\tiny$e_2^0$};
\node (2) at (-2,-1)[draw]{\tiny$f_1^0$};
\node (3) at (-1,2)[draw]{\tiny$e_1^0$};
\node (4) at (-1,0)[draw]{\tiny$f_1^1$};
\node (5) at (-1,-2)[draw]{\tiny$f_2^0$};
\node (6) at (0,1)[draw]{\tiny $f_1^2$};
\node (7) at (0,-1)[draw]{\tiny $f_2^1$};
\node (8) at (1,2)[draw]{\tiny${e'}_1^0$};
\node (9) at (1,0)[draw]{\tiny${f'}_1^{-1}$};
\node (10) at (1,-2)[draw]{\tiny${f'}_2^0$};
\node (11) at (2,1)[draw]{\tiny${e'}_2^0$};
\node (12) at (2,-1)[draw]{\tiny${f'}_1^0$};
\drawpath{8,3,6,9,7,5,10,7,4,6,8}{thick}
\drawpath{1,4,2,7,12,9,11,6,1}{thick}
\drawpath{1,3}{dashed,thick}
\drawpath{5,2}{dashed,thick}
\drawpath{8,11}{dashed,thick}
\drawpath{12,10}{dashed,thick}
\node at (-2,-2.5) {\huge $Q$};
\node at (-2,2.5) {\huge $\til{Q'}$};
\end{scope}
\path (-0.5,0) edge[->,thick](0.5,0);
\node at (0,0.5) {$\mu_{R_1}$};
\end{tikzpicture}
\caption{The flipping of triangle $\mu_{R_1}$ of the basic quivers, before changing the index back to standard form.}
\end{figure}
\begin{figure}[H]
\centering
\begin{tikzpicture}[every node/.style={circle, fill=white,inner sep=0, minimum size=0.5cm, thick}, x=1cm, y=1cm]
\begin{scope}[shift={(-4,0)}]
\draw (0,3.1)-- (-3.1,0)-- (0,-3.1)-- (3.1,0)--(0,3.1)--(0,-3.1);
\node (1) at (-2,1)[draw]{\tiny$f_2^{-1}$};
\node (2) at (-2,-1)[draw]{\tiny$f_2^0$};
\node (3) at (-1,2)[draw]{\tiny$f_1^{-2}$};
\node (4) at (-1,0)[draw]{\tiny$f_1^{-1}$};
\node (5) at (-1,-2)[draw]{\tiny$f_1^0$};
\node (6) at (0,1)[draw]{\tiny$e_1^0$};
\node (7) at (0,-1)[draw]{\tiny $e_2^0$};
\node (8) at (1,2)[draw]{\tiny${e'}_1^0$};
\node (9) at (1,0)[draw]{\tiny${f'}_1^{-1}$};
\node (10) at (1,-2)[draw]{\tiny${f'}_2^0$};
\node (11) at (2,1)[draw]{\tiny${e'}_2^0$};
\node (12) at (2,-1)[draw]{\tiny${f'}_1^0$};
\drawpath{2,4,6,9,12,11,9,7,4,1,2}{thick}
\drawpath{6,3,4,5,7,10,9,8,6}{thick}
\drawpath{1,3}{dashed,thick}
\drawpath{5,2}{dashed,thick}
\drawpath{8,11}{dashed,thick}
\drawpath{12,10}{dashed,thick}
\node at (-2.5,-2) {\huge $\til{Q}$};
\node at (2.5,-2) {\huge $\til{Q'}$};
\node at (0,-3.5) [fill=none, rectangle]{$(e_1^0={f'}_1^{-2}, e_2^0={f'}_2^{-1})$};
\end{scope}
\begin{scope}[shift={(4,0)}]
\draw (3.1,0)-- (0,-3.1)-- (-3.1,0)--(0,3.1)--(3.1,0)--(-3.1, 0);
\node (1) at (-2,1)[draw]{\tiny$e_2^0$};
\node (2) at (-2,-1)[draw]{\tiny$f_1^0$};
\node (3) at (-1,2)[draw]{\tiny$e_1^0$};
\node (4) at (-1,0)[draw]{\tiny$f_1^1$};
\node (5) at (-1,-2)[draw]{\tiny$f_2^0$};
\node (6) at (0,1)[draw]{\tiny $f_1^2$};
\node (7) at (0,-1)[draw]{\tiny $f_2^1$};
\node (8) at (1,2)[draw]{\tiny${e'}_1^0$};
\node (9) at (1,0)[draw]{\tiny${f'}_1^{-1}$};
\node (10) at (1,-2)[draw]{\tiny${f'}_2^0$};
\node (11) at (2,1)[draw]{\tiny${e'}_2^0$};
\node (12) at (2,-1)[draw]{\tiny${f'}_1^0$};
\drawpath{8,3,6,9,7,5,10,7,4,6,8}{thick}
\drawpath{1,4,2,7,12,9,11,6,1}{thick}
\drawpath{1,3}{dashed,thick}
\drawpath{5,2}{dashed,thick}
\drawpath{8,11}{dashed,thick}
\drawpath{12,10}{dashed,thick}
\node at (-2,-2.5) {\huge $\til{Q}$};
\node at (-2,2.5) {\huge $\til{Q'}$};
\end{scope}
\path (-0.5,0) edge[->,thick](0.5,0);
\node at (0,0.5) {$\mu_{R_2}$};
\end{tikzpicture}
\caption{The flipping of triangle $\mu_{R_2}$ of the basic quivers, before changing the index back to standard form.}
\end{figure}\label{muRi}
\end{subfigures}
%==============================================================================
\subsubsection{Type $A_n$}\label{sec:Dehn:An}
Let $\bi=(1213214321...n...1)$ be the usual reduced word. To study $R_1$ and $R_2$, let
\Eq{
\cP_i^{Q\til{Q}'}:= \<\cP_{F_i}^{\til{Q}'}-\cP_{E_i}^Q\>,\tab \cP_i^{\til{Q}\til{Q}'}:=\<\cP_{F_i}^{\til{Q}'}-\cP_{E_i}^{\til{Q}}\>
} be the concatenation of the $F_i, E_i$ path of the right and left quiver respectively. Then the mutation sequences $\mu_{R_j}$, $j=1,2$, are given by
$$\mu_{R_j}:=\{\cP_1^j\to \cP_2^j\to...\cP_N^j\}$$
where for $v'(i,k)=m$, $\cP_m^j$ are the $k$-shifted subsequences
\Eqn{
\cP_m^1=\cP_i^{Q\til{Q}'}[k,i],\tab \cP_m^2=\cP_i^{\til{Q}\til{Q}'}[k,i].
}

Let
\Eq{
\cP_i^{{}_Q^{\til{Q}'}}:=  \cF(\cP_i^{Q\til{Q}'})=\<\cP_{E_i}^Q-\cP_{F_i}^{\til{Q}'}\>,\tab \cP_i^{{}_{\til{Q}}^{\til{Q}'}}:=\cF(\cP_i^{\til{Q}\til{Q}'})=\<\cP_{E_i}^{\til{Q}}-\cP_{F_i}^{\til{Q}'}\>
}
be the concatenation of the $E_i, F_i$ path of the bottom and top quiver respectively after the flip of triangulation. The permutations $\s_1,\s_2$ are then defined by renaming the corresponding sequence:
\Eqn{
\s_1:\cP_i^{Q\til{Q}'}\mapsto \cP_i^{{}_Q^{\til{Q}'}},\tab \s_2:\cP_i^{\til{Q}\til{Q}'}\mapsto \cP_i^{{}_{\til{Q}}^{\til{Q}'}}
}

For example, in type $A_3$, we have $\bi=(1,2,1,3,2,1)$ and
\Eqn{
\cP_1^{Q\til{Q}'}&: ({f'}_1^0,{f'}_1^{-1},{f'}_1^{-2},{f'}_1^{-3}=f_1^3,e_1^0),\\
\cP_2^{Q\til{Q}'}&: ({f'}_2^0,{f'}_2^{-1},{f'}_2^{-2}=f_2^2,f_1^2,e_2^0),\\
\cP_3^{Q\til{Q}'}&: ({f'}_3^0,{f'}_3^{-1}=f_3^1,f_2^1,f_1^1,e_1^0),
}
and hence the mutation sequence giving the first flip of triangulations is
\Eqn{
\mu_{R_1}&=\{f_1^3\to f_1^2\to f_2^2\to f_1^1\to f_2^1\to f_3^1\to{f'}_1^{-2}\to{f'}_2^{-1}\to f_2^2\to{f'}_1^{-1}\}.
}
%==============================================================================
\subsubsection{Type $B_n$ and $C_n$}\label{sec:Dehn:Bn}
Let the reduced word $\bi$ be as in \eqref{Bni}. It turns out that type $B_n$ and type $C_n$ have identical mutation sequences.  Define the sequence
\Eqn{
\cS:&=(f_1^n,f_1^{n-1},f_1^{n-2},...,f_1^1,e_1^0),\\
\cS':&=(e_1^0,f_1^1,f_1^2,...,f_1^n),
}
and let
\Eqn{
\cP_i^{Q\til{Q}'}&:=\case{\<\cP_{F_i}^{\til{Q}'}-\cP_{E_i}^Q\>& i\neq 1,\\
(\cP_{F_1}^{\til{Q}'}-\cS\>& i=1,}\\
\cP_i^{\til{Q}\til{Q}'}&:= \case{\<\cP_{F_i}^{\til{Q}'}-\cP_{E_i}^{\til{Q}}\>&i\neq 1,\\
\<\cP_{F_1}^{\til{Q}'}-\cS'\>&i=1.}
}
Then the mutation sequences $\mu_{R_j}$, $j=1,2$, are given by
$$\mu_{R_j}:=\{\cP_1^j\to \cP_2^j\to...\cP_N^j\}$$

For $i\neq 1$ and $v'(i,k)=m$, $\cP_m^j$ are the $k$-shifted subsequences
\Eqn{
\cP_m^1=\cP_i^{Q\til{Q}'}[k,m_i],\tab \cP_m^2=\cP_i^{\til{Q}\til{Q}'}[k,m_i],
}where $m_i=2(n-i)+1$ is the length of the $E_i$ path.

For $i=1$ and $v'(1,k)=m$, let $$f_1^n\to f_2^{2n-3}\to f_1^{n-1}\to f_2^{2n-5}\to f_1^{n-2}\to...\to f_1^2\to f_2^1\to f_1^1$$ be the $E_i$ path of $\be_1$ in $Q$ (ignore the double count in type $C_n$). Then 
\Eqn{
\cP_m^1=&\cP_1^{Q\til{Q}'}(1-k)\to f_2^{2n-3}\to \cP_1^{Q\til{Q}'}(1-k)\to\cP_1^{Q\til{Q}'}(2-k)\to f_2^{2n-5}\to\cP_1^{Q\til{Q}'}(2-k) ...\\
&...\to\cP_1^{Q\til{Q}'}(n-k-1)\to f_2^1\to\cP_1^{Q\til{Q}'}(n-k-1)\to \cP_1^{Q\til{Q}'}(n-k).
}
Let $$e_1^0\to f_1^{-1}\to f_2^{-1}\to f_1^{-2}\to...\to f_1^{-(n-1)}\to f_2^{-(2n-3)}\to f_1^{-n}$$ be the $E_i$ path of $\be_1$ in $\til{Q}$. Then 
\Eqn{
\cP_m^2=&\cP_1^{\til{Q}\til{Q}'}(1-k)\to\cP_1^{\til{Q}\til{Q}'}(2-k)\to f_2^{-1}\to \cP_1^{\til{Q}\til{Q}'}(2-k)\to\cP_1^{\til{Q}\til{Q}'}(3-k)\to f_2^{-3}\to\cP_1^{\til{Q}\til{Q}'}(3-k) ...\\
&...\to\cP_1^{\til{Q}\til{Q}'}(n-k)\to f_2^{-(2n-3)}\to\cP_1^{\til{Q}\til{Q}'}(n-k).
}

Let
\Eqn{
\cP_i^{{}_Q^{\til{Q}'}}:= \cF(\cP_i^{Q\til{Q}'}),\tab \cP_i^{{}_{\til{Q}}^{\til{Q}'}}&:= \cF(\cP_i^{\til{Q}\til{Q}'})
} Then the permutations $\s_1,\s_2$ are again defined by renaming the corresponding sequence:
\Eqn{
\s_1:\cP_i^{Q\til{Q}'}\mapsto \cP_i^{{}_Q^{\til{Q}'}},\tab \s_2:\cP_i^{\til{Q}\til{Q}'}\mapsto \cP_i^{{}_{\til{Q}}^{\til{Q}'}}
}

For example, in type $B_3$, we have $\bi=(1,2,1,2,3,2,1,2,3)$ and
\Eqn{
\cP_1^{Q\til{Q}'}&: ({f'}_1^0,{f'}_1^{-1},{f'}_1^{-2},{f'}_1^{-3}=f_1^3,f_1^2,f_1^1,e_1^0),\\
\cP_2^{Q\til{Q}'}&: ({f'}_2^0,{f'}_2^{-1},{f'}_2^{-2},{f'}_2^{-3},{f'}_2^{-4}=f_2^4,f_3^1,f_2^2,e_2^0),\\
\cP_3^{Q\til{Q}'}&: ({f'}_3^0,{f'}_3^{-1},{f'}_3^{-2}=f_3^2,e_3^0),
}
and hence the mutation sequence giving the first flip of triangulations is (spacing according to $\bi'$):

\Eqn{
\mu_{R_1}=\{&f_3^2\to\\
&f_2^4\to f_3^1\to f_2^2\to \\
&f_1^3 \to f_2^3\to f_1^3\to f_1^2\to f_2^1\to f_1^2\to f_1^1\to \\
&{f'}_2^{-3}\to f_2^4\to f_3^1\to\\
& {f'}_3^{-1}\to\\
& {f'}_2^{-2}\to {f'}_2^{-3}\to f_2^4\to\\
& {f'}_1^{-2}\to f_2^3\to {f'}_1^{-2}\to f_1^3\to f_2^1\to f_1^3\to f_1^2\to\\
& {f'}_2^{-1}\to {f'}_2^{-2}\to{f'}_2^{-3}\to\\
& {f'}_1^{-1}\to f_2^3\to {f'}_1^{-1}\to {f'}_1^{-2}\to f_2^1\to {f'}_1^{-2}\to {f'}_1^{-3}\}.
}
%==============================================================================
\subsubsection{Type $D_n$}\label{sec:Dehn:Dn}
The description of the $D_n$ mutation sequences is a lot more complicated. Let the reduced word $\bi$ be as in \eqref{Dni}. For $i\neq 0,1$, define as before
$$\cP_i^{Q\til{Q}'}=\<\cP_{F_i}^{\til{Q}'}-\cP_{E_i}^{Q}\>.$$
Let
$$\ov{n}:=n\mbox{ (mod 3)}\in\{0,1,2\}$$
and define the following sequences, which are constructed by repeating in blocks of 4:
\Eqn{
\cS_1&=(X_1, ...,{f'}_0^{-3k+\ov{n}-2}, {f'}_0^{-3k-\ov{n}-1},{f'}_1^{-3k-\ov{n}-1}, {f'}_1^{-3k-\ov{n}},...,{f'}_1^{-n+1}),\\
\cS_2&=(X_2, ...,{f'}_0^{-3k-\ov{n}-1}, {f'}_0^{-3k-\ov{n}},{f'}_1^{-3k-\ov{n}}, {f'}_1^{-3k-\ov{n}+1},...,{f'}_0^{-n+1}),\\
\cS_0&=(X_0, ...,{f'}_0^{-3k-\ov{n}}, {f'}_0^{-3k-\ov{n}+1},{f'}_1^{-3k-\ov{n}+1}, {f'}_1^{-3k-\ov{n}+2},...,{f'}_0^{-n+1}),
}
where the starting terms are given by
\Eqn{
X_i=\case{{f'}_1^0&\ov{n}=i,\\{f'}_0^0&\mbox{otherwise}.}
}
Let \Eqn{
\cT_0=(f_0^{n-1}, f_1^{n-1},f_1^{n-2},f_0^{n-2},f_0^{n-3},f_1^{n-3},..., f_\e^1,f_{1-\e}^1,e_1^0)
}where $\e:=n\mbox{ (mod 2)}\in\{0,1\}$. Let $\cT_1=\cP_{E_1^\#}^{Q}$ denote the $E_1$ path in $Q$, but with the last term $e_1^0$ replaced by $e_0^0$, and let $\cT_2=\cP_{E_0}^{Q}$.

Finally, we define
\Eqn{
\cU_j:=\<\cS_j-\cT_j\>, \tab j=0,1,2.
}
Then the mutation sequence for $R_1$ is given by
\Eqn{
\mu_{R_1}:=\{\cP_1\to \cP_2\to...\to \cP_N\}.
}
For $i\neq 0,1$ and $v'(i,k)=m$, we have as before
\Eqn{
\cP_m=\cP_i^{Q\til{Q}'}[k,m_i]
}
where $m_i=2(n-i)-1$ is the length of the $E_i$ path.

For $i=0,1$ and $v'(i,k)=m$, we have
\Eqn{
\cP_m&=\cU_{\ov{k-1+i}}[K_k^i, 2n-3]
}
where $K_k^0=(0,2,3,4,6,7,8,10,11,12...)$ and $K_k^1=(0,1,2,4,5,6,8,9,10...)$.

Then the permutation is given by
$$\s_1: \cP_i^{Q\til{Q}'}\mapsto \cP_i^{{}_Q^{\til{Q}'}},\tab i\neq 0,1$$
and
$$\s_1: \cU_j\mapsto \<\cT_{\ov{j+1-n}}-\cS_{\ov{j+1}}\>,\tab j=0,1,2.$$

The second flip $R_2$ is described similarly, where all the sequences $\cT_j$ are reversed and the root indexes $0\corr 1$ interchanged, and $\cS_j$ are replaced by $\cS_{\ov{j-1}}$.

For example, in type $D_4$, we have $\bi=(0,1,2,0,1,2,3,2,0,1,2,3)$, and
\Eqn{
\cP_2^{Q\til{Q}'}&=({f'}_2^0,{f'}_2^{-1},{f'}_2^{-2},{f'}_2^{-3}, {f'}_2^{-4}=f_2^4,f_3^1,f_2^2,e_2^0),\\
\cP_3^{Q\til{Q}'}&=({f'}_3^0,{f'}_3^{-1},{f'}_3^{-2}=f_3^2, e_3^0),\\
\cU_1&=({f'}_1^0,{f'}_0^{-1},{f'}_0^{-2},{f'}_1^{-2},{f'}_1^{-3}=f_1^3,f_2^3,f_0^2,f_2^1,f_1^1,e_0^0),\\
\cU_2&=({f'}_0^0,{f'}_0^{-1},{f'}_1^{-1},{f'}_1^{-2},{f'}_0^{-3}=f_0^3,f_2^3,f_1^2,f_2^1,f_0^1,e_0^0),\\
\cU_0&=({f'}_0^0,{f'}_1^0,{f'}_1^{-1},{f'}_0^{-2},{f'}_0^{-3}=f_0^3,f_1^3,f_1^2,f_0^2,f_0^1,f_1^1,e_1^0),
}
and hence the mutation sequence giving the first flip of triangulations is (spacing according to $\bi'$):
\Eqn{
\mu_{R_1}=\{& f_3^2\to\\
&f_2^4\to f_3^1\to f_2^2\to\\
&f_1^3\to f_2^3\to f_0^2\to f_2^1\to f_1^1\to\\
&f_0^3\to f_1^3\to f_1^2\to f_0^2\to f_0^1\to\\
&{f'}_2^{-3}\to f_2^4\to f_3^1\to\\
&{f'}_3^{-1}\to\\
&{f'}_2^{-2}\to {f'}_2^{-3}\to f_2^4\to\\
&{f'}_1^{-2}\to f_0^3\to f_2^3\to f_1^2\to f_2^1\to\\
&{f'}_0^{-2}\to {f'}_1^{-2}\to f_1^3\to f_2^3\to f_0^2\to\\
&{f'}_2^{-1}\to {f'}_2^{-2}\to{f'}_2^{-3}\to\\
&{f'}_1^{-1}\to {f'}_0^{-2}\to f_0^3\to f_1^3\to f_1^2\to\\
&{f'}_0^{-1}\to {f'}_1^{-1}\to {f'}_1^{-2}\to f_0^3\to f_2^3
\}.
}
%==============================================================================
\subsubsection{Exceptional types}\label{sec:Dehn:Ex}
The mutation sequences can be worked out in the exception type, but there are no apparent patterns, so we will not present here. We know that the reduced $R$ matrix corresponds to 
$$T=4\prod_{i=1}^n n_i \cE_i$$
mutations, where $\cE_i$ is the number of factors in the $g_b(\be_i)$ decomposition as in Proposition \ref{Endecomp}. One check explicitly that indeed the mutation sequences give the half-Dehn twist. Combining with the classical types, we have
\begin{Prop} The half-Dehn twist can be represented by $T$ quiver mutations, where
$$T=\case{
\frac{2}{3}n(n+1)(n+2)&\mbox{Type $A_n$}\\
\frac{4}{3}n(4n^2-1)&\mbox{Type $B_n$ and $C_n$}\\
\frac{4}{3}n(n-1)(4n-5)&\mbox{Type $D_n$}\\
1196&\mbox{Type $E_6$}\\
3464&\mbox{Type $E_7$}\\
12064&\mbox{Type $E_8$}\\
976&\mbox{Type $F_4$}\\
144&\mbox{Type $G_2$}}$$
and each flip of triangulations are given by $\frac{T}{4}$ quiver mutations.
\end{Prop}

In type $G_2$, from the factorization of $g_{b_s}(\be_2)$, we see that it involves the factor $g_{b_s}^*(X_{...})$. We use the fact that
\Eqn{
\mu_k^q&=Ad_{g_b^*(X_k)}\circ \mu_k',\\
&=Ad_{g_b(X_k\inv)}\circ \mu_k'',
}
where $\mu_k''$ is the same as $\mu_k$ but with $b_{ij}\to b_{ji}$ inverted in the formula. With slight modification of Lemma \ref{useful}, we obtain a mutation sequence $\mu_{R_1}$ of length 36 given by
\begin{tiny}
\Eqn{
\mu_{R_1}=\{&{f'}_1^{-3}\to f_1^1\to f_1^2\to f_2^2\to f_1^1\to {f'}_2^{-3}\to f_2^2\to f_1^1\to f_2^3\to f_1^2\to f_1^1\to f_2^1\to\\
&{f'}_1^{-2}\to f_1^1\to f_1^2\to f_2^3\to f_1^1\to {f'}_2^{-2}\to f_2^3\to f_1^1\to {f'}_2^{-2}\to f_1^2\to f_1^1\to f_2^2\to\\
&{f'}_1^{-1}\to f_1^1\to f_1^2\to {f'}_2^2\to f_1^1\to {f'}_2^{-1}\to{f'}_2^{-2}\to f_1^1\to {f'}_2^{-1}\to f_1^2\to f_1^1\to f_2^3\},
}
\end{tiny}
where $f_i^{n_i}$ and ${f'}_i^{-n_i}$ are identified.

The basic quiver (cf. Figure \ref{G2basicquiver}) can be attached to a triangle as in Figure \ref{G2basicquiver2}.
Then the mutation $\mu_{R_1}$ appears as in Figure \ref{G2mutation}, and we can determine $\s_1$ to be:
\Eqn{
\s_1:\<\cS_i-\cT_i\>\mapsto \<\cT_i-\cS_i\>,\tab i=1,2,
}
where
\Eqn{
\cS_1=\cP_{F_1}^{\til{Q}'}&=({f'}_1^0,{f'}_1^{-1},{f'}_1^{-2},{f'}_1^{-3}),& \cT_1&=(f_1^3,e_1^0),\\
\cS_2=\cP_{F_2}^{\til{Q}'}&=({f'}_2^0,{f'}_2^{-1},{f'}_2^{-2},{f'}_2^{-3}),&\cT_2&=(f_2^3,f_2^2,f_2^1,e_2^0).
}

Similarly, the description for $\mu_{R_2}$ is given by
\begin{tiny}
\Eqn{
\mu_{R_2}=\{&{f'}_1^{-3}\to {f'}_2^{-3}\to f_1^{-2}\to f_1^{-1}\to f_2^{-2}\to f_1^{-2}\to f_2^{-1}\to f_2^{-2}\to f_1^{-2}\to f_2^{-1}\to f_1^{-1}\to f_1^{-2}\to\\
&{f'}_1^{-2}\to {f'}_2^{-2}\to f_1^{-2}\to f_1^{-1}\to f_2^{-1}\to f_1^{-2}\to e_2^0\to f_2^{-1}\to f_1^{-2}\to e_2^0\to f_1^{-1}\to f_1^{-2}\to\\
&{f'}_1^{-1}\to {f'}_2^{-1}\to f_1^{-2}\to f_1^{-1}\to e_2^0\to f_1^{-2}\to {f'}_2^{-2}\to e_2^0\to f_1^{-2}\to {f'}_2^{-2}\to f_1^{-1}\to f_1^{-2}\},
}
\end{tiny}
where $e_i^0$ and ${f'}_i^{-n_i}$ are identified. The permutation is then given by
\Eqn{
\s_2:\<\cS_i-\cT'_i\>\mapsto \<\cT'_i-\cS_i\>,\tab i=1,2,
}
where $\cS_i$ is the same as before, while
\Eqn{
\cT'_1&=(e_1^0,f_1^{-3}),\tab\cT'_2=(e_2^0,f_2^{-1},f_2^{-2},f_2^{-3}).
}
\begin{figure}[!htb]
\centering
\begin{tikzpicture}[every node/.style={inner sep=0, minimum size=0.5cm, thick, fill=white},x=1.3cm,y=0.75cm]
\draw (-0.3,0) node[anchor=south, fill=none]{} 
-- (6.1,6.5) node[anchor=south]{}
 -- (6.1,-6.5) node[anchor=north]{}
 --cycle;
\node (1) at (2,-2)[draw] {$f_1^0$};
\node (2) at (3.5,1.3)[draw, circle] {$f_1^1$};
\node (3) at (4.5,1.3)[draw, circle] {$f_1^2$};
\node (4) at (6,2)[draw] {$f_1^3$};
\node (5) at (4,-4)[draw] {$f_2^0$};
\node (6) at (3.5,-0.5)[draw, circle] {$f_2^1$};
\node (7) at (4.5,-0.5)[draw, circle] {$f_2^2$};
\node (8) at (6,-2)[draw] {$f_2^3$};
\node (9) at (4,4)[draw] {$e_1^0$};
\node (10) at (2,2)[draw] {$e_2^0$};
\drawpath{1,2,3,4}{vthick}
\drawpath{5,6,7,8}{thin}
\drawpath{6,10,5}{thin}
\drawpath{10,9}{vthick, dashed}
\drawpath{4,8}{vthick, dashed}
\drawpath{1,5}{vthick, dashed}
\drawpath{8,3,7,2,6,1}{vthick}
\drawpath{4,9,3}{vthick}
\end{tikzpicture}
\caption{Basic quiver in type $G_2$ attached to a triangle.}\label{G2basicquiver2}
\end{figure}
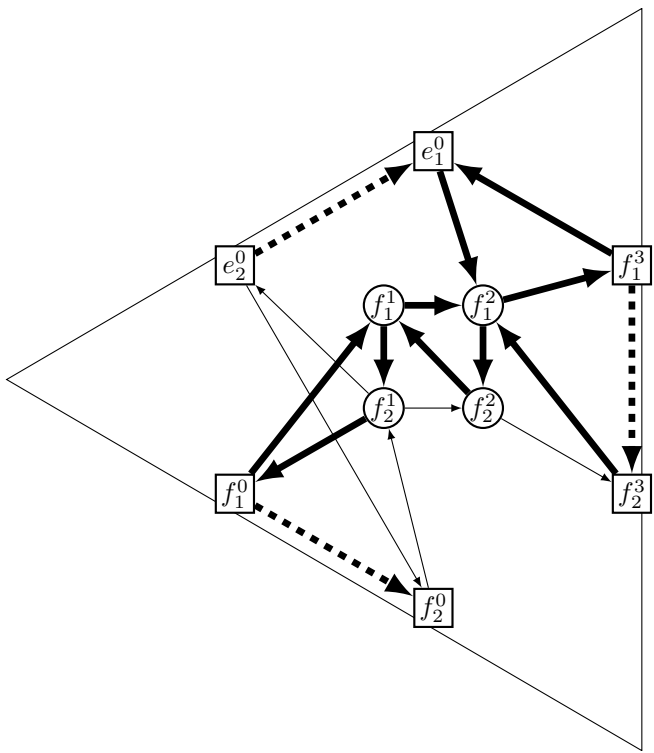
\begin{landscape}
\begin{subfigures}
\begin{figure}[H]
\centering
\begin{tikzpicture}[every node/.style={circle, fill=white,inner sep=0, minimum size=0.5cm, thick}, x=0.75cm, y=0.75cm]
\begin{scope}[shift={(-7,0)}]
\draw (0,6.1)node[anchor=south, fill=none]{$A$}
--(6.1,0)node[anchor=south, fill=none]{$B$}
--(0,-6.1)node[anchor=north, fill=none]{$C$}
--(-6.1,0)node[anchor=south, fill=none]{$D$}
--(0,6.1)--(0,-6.1);
\node (1) at (-4,-2)[draw] {\tiny$f_1^0$};
\node (2) at (-2.5,1)[draw] {\tiny$f_1^1$};
\node (3) at (-1,1)[draw] {\tiny$f_1^2$};
\node (4) at (0,2)[draw] {\tiny$f_1^3$};
\node (5) at (-2,-4)[draw] {\tiny$f_2^0$};
\node (6) at (-2.5,-0.5)[draw] {\tiny$f_2^1$};
\node (7) at (-1,-0.5)[draw] {\tiny$f_2^2$};
\node (8) at (0,-2)[draw] {\tiny$f_2^3$};
\node (9) at (-2,4)[draw] {\tiny$e_1^0$};
\node (10) at (-4,2)[draw] {\tiny$e_2^0$};
\node (11) at (4,-2)[draw] {\tiny${f'}_1^0$};
\node (12) at (2.5,1)[draw] {\tiny${f'}_1^{-1}$};
\node (13) at (1,1)[draw] {\tiny${f'}_1^{-2}$};
\node (15) at (2,-4)[draw] {\tiny${f'}_2^0$};
\node (16) at (2.5,-0.5)[draw] {\tiny${f'}_2^{-1}$};
\node (17) at (1,-0.5)[draw] {\tiny${f'}_2^{-2}$};
\node (19) at (2,4)[draw] {\tiny${e'}_1^0$};
\node (20) at (4,2)[draw] {\tiny${e'}_2^0$};
\drawpath{1,2,3,4}{vthick}
\drawpath{5,6,7,8}{thin}
\drawpath{6,10,5}{thin}
\drawpath{10,9}{vthick, dashed}
\drawpath{1,5}{vthick, dashed}
\drawpath{8,3,7,2,6,1}{vthick}
\drawpath{4,9,3}{vthick}
\drawpath{4,13,12,11}{vthick}
\drawpath{8,17,16,15}{thin}
\drawpath{15,20,16}{thin}
\drawpath{19,20}{vthick, dashed}
\drawpath{15,11}{vthick, dashed}
\drawpath{11,16,12,17,13,8}{vthick}
\drawpath{13,19,4}{vthick}
\node at (-5,-4) {\huge $Q$};
\node at (5,-4) {\huge $\til{Q'}$};
\node at (0,-7.5) [fill=none, rectangle]{$(f_1^3={f'}_1^{-3}, f_2^3={f'}_2^{-3})$};
\end{scope}
\begin{scope}[shift={(7,0)}]
\draw
(6.1,0)node[anchor=south, fill=none]{$B$}
--(0,-6.1)node[anchor=north, fill=none]{$C$}
--(-6.1,0)node[anchor=south, fill=none]{$D$} 
--(0,6.1)node[anchor=south, fill=none]{$A$}
--(6.1,0)--(-6.1,0);
\node (1) at (-4,-2)[draw] {\tiny$f_1^0$};
\node (2) at (0,-1)[draw] {\tiny$f_1^1$};
\node (3) at (1,-2)[draw] {\tiny$f_1^2$};
\node (4) at (4,-2)[draw] {\tiny${f'}_1^0$};
\node (5) at (-2,-4)[draw] {\tiny$f_2^0$};
\node (6) at (-1,-2)[draw] {\tiny${f'}_2^{-2}$};
\node (7) at (0,-3)[draw] {\tiny${f'}_2^{-1}$};
\node (8) at (2,-4)[draw] {\tiny${f'}_2^0$};
\node (9) at (2,0)[draw] {\tiny${f'}_1^{-1}$};
\node (10) at (-2,0)[draw] {\tiny$f_2^3$};
\node (12) at (1,2)[draw] {\tiny${f'}_1^{-2}$};
\node (13) at (0,3)[draw] {\tiny$f_1^3$};
\node (14) at (-2,4)[draw] {\tiny$e_1^0$};
\node (16) at (0,1)[draw] {\tiny$f_2^2$};
\node (17) at (-1,2)[draw] {\tiny$f_2^1$};
\node (18) at (-4,2)[draw] {\tiny$e_2^0$};
\node (19) at (2,4)[draw] {\tiny${e'}_1^0$};
\node (20) at (4,2)[draw] {\tiny${e'}_2^0$};
\drawpath{1,2,3,4}{vthick}
\drawpath{5,6,7,8}{thin}
\drawpath{6,10,5}{thin}
\drawpath{10,9}{vthick}
\drawpath{1,5}{vthick, dashed}
\drawpath{4,8}{vthick, dashed}
\drawpath{8,3,7,2,6,1}{vthick}
\drawpath{4,9,3}{vthick}
\drawpath{14,13,12,9}{vthick}
\drawpath{18,17,16,10}{thin}
\drawpath{10,20,16}{thin}
\drawpath{19,20}{vthick, dashed}
\drawpath{18,14}{vthick, dashed}
\drawpath{9,16,12,17,13,18}{vthick}
\drawpath{13,19,14}{vthick}
\node at (-4,-5) {\huge $Q$};
\node at (-4,5) {\huge $\til{Q'}$};
\end{scope}
\path (-0.5,0) edge[->,thick](0.5,0);
\node at (0,0.5) {$\mu_{R_1}$};
\end{tikzpicture}
\caption{The flipping of triangle $\mu_{R_1}$ of the basic quivers in type $G_2$, before changing the index back to standard form. The basic quivers are stacked according to Figure \ref{mutateR1}.}
\label{G2mutation}
\end{figure}
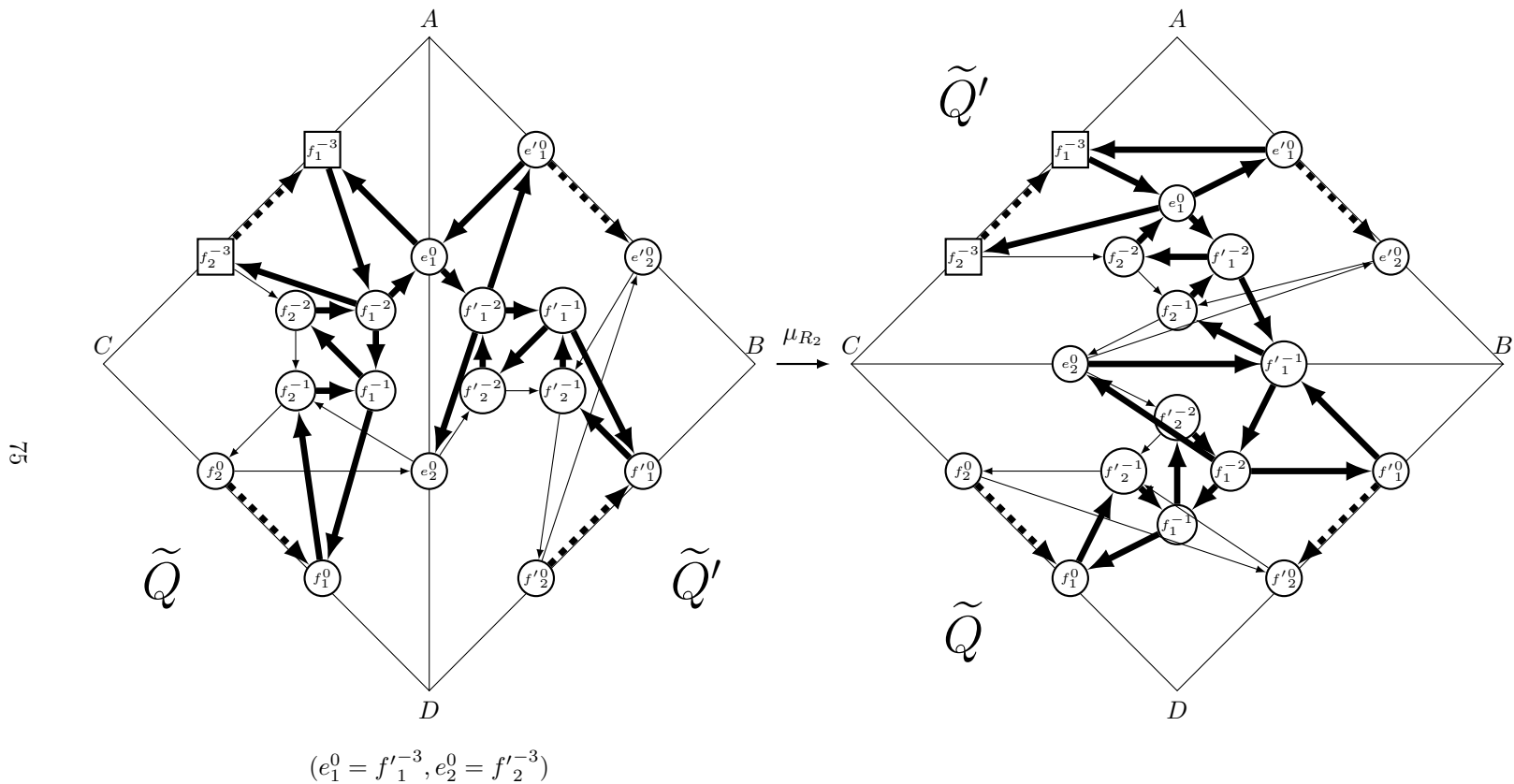
\begin{figure}[H]
\centering
\begin{tikzpicture}[every node/.style={circle, fill=white,inner sep=0, minimum size=0.5cm, thick}, x=0.75cm, y=0.75cm]
\begin{scope}[shift={(-7,0)}]
\draw (0,6.1)node[anchor=south, fill=none]{$A$}
--(6.1,0)node[anchor=south, fill=none]{$B$}
--(0,-6.1)node[anchor=north, fill=none]{$D$}
--(-6.1,0)node[anchor=south, fill=none]{$C$}
--(0,6.1)--(0,-6.1);
\node (1) at (-2,-4)[draw] {\tiny$f_1^0$};
\node (2) at (-1,-0.5)[draw] {\tiny$f_1^{-1}$};
\node (3) at (-1,1)[draw] {\tiny$f_1^{-2}$};
\node (4) at (-2,4)[draw, rectangle] {\tiny$f_1^{-3}$};
\node (5) at (-4,-2)[draw] {\tiny$f_2^0$};
\node (6) at (-2.5,-0.5)[draw] {\tiny$f_2^{-1}$};
\node (7) at (-2.5,1)[draw] {\tiny$f_2^{-2}$};
\node (8) at (-4,2)[draw, rectangle] {\tiny$f_2^{-3}$};
\node (9) at (0,2)[draw] {\tiny$e_1^0$};
\node (10) at (0,-2)[draw] {\tiny$e_2^0$};
\node (11) at (4,-2)[draw] {\tiny${f'}_1^0$};
\node (12) at (2.5,1)[draw] {\tiny${f'}_1^{-1}$};
\node (13) at (1,1)[draw] {\tiny${f'}_1^{-2}$};
\node (15) at (2,-4)[draw] {\tiny${f'}_2^0$};
\node (16) at (2.5,-0.5)[draw] {\tiny${f'}_2^{-1}$};
\node (17) at (1,-0.5)[draw] {\tiny${f'}_2^{-2}$};
\node (19) at (2,4)[draw] {\tiny${e'}_1^0$};
\node (20) at (4,2)[draw] {\tiny${e'}_2^0$};
\node at (-5,-4) {\huge $\til{Q}$};
\node at (5,-4) {\huge $\til{Q}'$};
\drawpath{4,3,2,1}{vthick}
\drawpath{8,7,6,5}{thin}
\drawpath{5,10,6}{thin}
\drawpath{8,4}{vthick, dashed}
\drawpath{5,1}{vthick, dashed}
\drawpath{1,6,2,7,3,8}{vthick}
\drawpath{3,9,4}{vthick}
\drawpath{9,13,12,11}{vthick}
\drawpath{10,17,16,15}{thin}
\drawpath{15,20,16}{thin}
\drawpath{19,20}{vthick, dashed}
\drawpath{15,11}{vthick, dashed}
\drawpath{11,16,12,17,13,10}{vthick}
\drawpath{13,19,9}{vthick}
\node at (0,-7.5) [fill=none, rectangle]{$(e_1^0={f'}_1^{-3}, e_2^0={f'}_2^{-3})$};
\end{scope}
\begin{scope}[shift={(7,0)}]
\draw
(6.1,0)node[anchor=south, fill=none]{$B$}
--(0,-6.1)node[anchor=north, fill=none]{$D$}
--(-6.1,0)node[anchor=south, fill=none]{$C$} 
--(0,6.1)node[anchor=south, fill=none]{$A$}
--(6.1,0)--(-6.1,0);
\node (1) at (-2,-4)[draw] {\tiny$f_1^0$};
\node (2) at (0,-3)[draw] {\tiny$f_1^{-1}$};
\node (3) at (1,-2)[draw] {\tiny$f_1^{-2}$};
\node (4) at (2,0)[draw] {\tiny${f'}_1^{-1}$};
\node (5) at (-4,-2)[draw] {\tiny$f_2^0$};
\node (6) at (-1,-2)[draw] {\tiny${f'}_2^{-1}$};
\node (7) at (0,-1)[draw] {\tiny${f'}_2^{-2}$};
\node (8) at (-2,0)[draw] {\tiny$e_2^0$};
\node (9) at (4,-2)[draw] {\tiny${f'}_1^0$};
\node (10) at (2,-4)[draw] {\tiny${f'}_2^0$};
\node (12) at (1,2)[draw] {\tiny${f'}_1^{-2}$};
\node (13) at (0,3)[draw] {\tiny$e_1^0$};
\node (14) at (-2,4)[draw, rectangle] {\tiny$f_1^{-3}$};
\node (16) at (0,1)[draw] {\tiny$f_2^{-1}$};
\node (17) at (-1,2)[draw] {\tiny$f_2^{-2}$};
\node (18) at (-4,2)[draw, rectangle] {\tiny$f_2^{-3}$};
\node (19) at (2,4)[draw] {\tiny${e'}_1^0$};
\node (20) at (4,2)[draw] {\tiny${e'}_2^0$};
\drawpath{4,3,2,1}{vthick}
\drawpath{8,7,6,5}{thin}
\drawpath{5,10,6}{thin}
\drawpath{9,10}{vthick, dashed}
\drawpath{5,1}{vthick, dashed}
\drawpath{1,6,2,7,3,8}{vthick}
\drawpath{3,9,4}{vthick}

\drawpath{14,13,12,4}{vthick}
\drawpath{18,17,16,8}{thin}
\drawpath{8,20,16}{thin}
\drawpath{19,20}{vthick, dashed}
\drawpath{18,14}{vthick, dashed}
\drawpath{8,4}{vthick}
\drawpath{4,16,12,17,13,18}{vthick}
\drawpath{13,19,14}{vthick}
\node at (-4,-5) {\huge $\til{Q}$};
\node at (-4,5) {\huge $\til{Q}'$};
\end{scope}
\path (-0.5,0) edge[->,thick](0.5,0);
\node at (0,0.5) {$\mu_{R_2}$};
\end{tikzpicture}
\caption{The flipping of triangle $\mu_{R_2}$ of the basic quivers in type $G_2$, before changing the index back to standard form. The basic quivers are stacked according to Figure \ref{mutateR2}.}
\end{figure}
\end{subfigures}
\end{landscape}
%==============================================================================
\subsection{Alternative factorization of the reduced $R$ matrix}\label{sec:Dehn:alternative}
From Remark \ref{iw}, one can use the Cartan involution and replace the first factor $\be_i\ox 1$ in the reduced $R$ matrix with the embedding by the $F_i$ paths. Then the embedding $\iota^w\ox \iota$ induces a very simple factorization of the reduced $R$ matrix, where
\Eqn{
g_b(\be_i^-) &= g_b(F_i^{1,-})...g_b(F_i^{n_i,-}),\\
g_b(\be_i^+) &= g_b(F_i^{n_i,+})...g_b(F_i^{1,+}),
}
and hence by Corollary \ref{CorR1234}, 
\begin{Cor}Under the embedding $\iota^w\ox \iota$, the reduced $R$ matrix factorizes as
\Eqn{
\ov{R}=R_4\cdot R_3\cdot R_2\cdot R_1,
}
where
\Eqn{
R_1&=\prod_{k=1}^{N}{}^{op}\prod_{j=1}^{n_{i_k}}{}^{op} g_b(F_{i_k}^{j,+}\ox X_k^+),\\
R_2&=\prod_{k=1}^{N}{}^{op}\prod_{j=1}^{n_{i_k}} g_b(F_{i_k}^{j,-}\ox X_k^+),\\
R_3&=\prod_{k=1}^{N}\prod_{j=1}^{n_{i_k}}{}^{op} g_b(F_{i_k}^{j,+}\ox X_k^-),\\
R_4&=\prod_{k=1}^{N}\prod_{j=1}^{n_{i_k}} g_b(F_{i_k}^{j,-}\ox X_k^-).
}
and recall that $\Pi^{op}$ means multiplying from the right.
\end{Cor}

The embedding $\iota^w\ox \iota$ corresponds to a new quiver $\til{\cZ_\g}$, which is another amalgamation of the two quivers $\cD_\g$, where the nodes $\{f_{i}^{n_i}\}$ of the first quiver are glued to $\{{f'}_i^{n_i}\}$ of the second quiver instead (see Figure \ref{newQQ} below). Then one can describe for every type of $\g$ the mutation sequence giving the flip of triangulations on $\til{\cZ_\g}$ easily:
\begin{Prop} Let 
$$\cP_i^{\til{Q}\til{Q}'}:=\<\cP_{F_i}^{\til{Q}'}-\cP_{F_i}^{\til{Q}}\>$$
be the concatenation of the $F_i$-paths in the corresponding subquivers of $\til{\cZ_\g}$. Then the mutation sequence giving the flip of triangulation is
$$\mu_{R_1}=\{\cP_1\to\cP_2\to...\cP_N\},$$
where as before if $i_m$ is the $k$-th appearance of the root index $i$ from the right of $\bi$, then
$$\cP_m=\cP_i^{\til{Q}\til{Q}'}[k,n_i].$$
\end{Prop}

When $\bi$ corresponds to the Coxeter element of the Weyl group, $w_0=w_c^{h/2}$, this coincides with the mutation sequence of the flip of triangulations (where two quivers mirrored to each other are glued) described in \cite{Le1} in the classical type. Hence this construction generalizes those of \cite{Le1}, and at the same time provides a representation theoretic meaning of the sequences giving the flip of triangulations described there.
\begin{figure}[!htb]
\centering
\begin{tikzpicture}[every node/.style={thick,circle,inner sep=0pt}, x=1cm,y=1cm]
\foreach \y in {0,1}{
	\draw [very thick] (6*\y,0) circle (2);
	\node (a-\y) at (6*\y,2) [draw,minimum size=5, fill=black]{};
	\node (b-\y) at (6*\y+1,0) [draw,minimum size=5]{};
	\node (c-\y) at (6*\y,-2) [draw,minimum size=5, fill=black]{};
	\node (d-\y) at (6*\y-1,0) [draw,minimum size=5]{};
	\path (a-\y) edge[-, very thick] (b-\y);
	\path (b-\y) edge[-, very thick] (c-\y);
	\path (c-\y) edge[-, very thick] (d-\y);
	\path (d-\y) edge[-, very thick] (a-\y);
	\path (a-\y) edge[-, very thick] (c-\y);					
}
\node at (0,2.3){\tiny$A$};
\node at (1,0.3){\tiny$B$};
\node at (0,-2.3){\tiny$C$};
\node at (-1,0.3){\tiny$D$};
\node at (-1.5,0){$Q$};
\node at (-0.5,0){$\til{Q}$};
\node at (0.5,0){$\til{Q'}$};
\node at (1.5,0){$Q'$};
\node at (6,2.3){\tiny$A$};
\node at (7,0.3){\tiny$D$};
\node at (6,-2.3){\tiny$C$};
\node at (5,0.3){\tiny$B$};
\node at (4.5,0){$\til{Q'}$};
\node at (5.5,0){$Q'$};
\node at (6.5,0){$Q$};
\node at (7.5,0){$\til{Q}$};
\path (2.5,0) edge[->, very thick] (3.5,0);
\end{tikzpicture}
\caption{Half-Dehn twist of the quiver $\til{\cZ_\g}$.}\label{newQQ}
\end{figure}
Although the description of the $R$ matrix factorization is very nice, we see that after 4 flips it does not return to the original quiver, but rather a mirror image with all the arrows flipped. A full Dehn twist, however, return us to the original configuration. If Conjecture \ref{EFflip} is true, which gives a quiver mutation equivalence between $\iota$ and $\iota^w$ (with Dynkin involution), this will relate such nice presentation of the $R$ matrix factorization to the canonical one found in the main theorem.
%==============================================================================
\section{Proof of Theorem \ref{main2}}\label{sec:Thm}
Let $\til{R}$ denote the right hand side of \eqref{Rfac}. The strategy is to show that $\cK\til{R}$ also gives the braiding relations \eqref{braiding} as well. 
First of all, we have
\Eq{
Ad_{\cK}(1\ox \be_i+\be_i\ox K_i') &= K_i\ox \be_i+\be_i\ox 1,\\
Ad_{\cK}(\bf_i\ox 1+K_i\ox \bf_i) &= \bf_i\ox K_i'+1\ox \bf_i,\\
Ad_{\cK} \D(K_i) &=K_i\ox K_i.
}
Hence in order to prove the braiding relations, it suffices to show
\Eq{
\til{R}\D(\be_i)&=(1\ox \be_i+\be_i\ox K_i'),\\
\til{R}\D(\bf_i)&=(\bf_i\ox 1+K_i\ox \bf_i),\\
\til{R}\D(K_i)&=K_i\ox K_i,
}
where the last one is trivial. We begin with several Lemmas:
\begin{Lem}\label{sl2com} For any $\sl_2$ triple $(\be, \bf, K, K')$, and any self-adjoint element $X$, we have
\Eq{
Ad_{g_b(\be\ox X)} (\bf\ox 1+K'\ox X) = \bf\ox 1+K\ox X.
}
\end{Lem}
\begin{proof}
This is a well-known result by considering the formal power series expansion of $g_b$ (recall that we restrict ourselves to the compact case, but it holds for the non-compact case as well).
Recall 
$$g_b(u)=Exp_{q^{-2}}(-\frac{u}{q-q\inv})=\sum_{n\geq 0}\frac{(-1)^nq^{\frac12n(n-1)}u^n}{(q^n-q^{-n})...(q-q^{-1})},$$
and that we have
$$\be^n\bf - \bf\be^n =(q^{n}-q^{-n})(q^{n-1}K'-q^{1-n}K)\be^{n-1}..$$
Hence we can work out
\Eqn{
&g_b(\be\ox X)(\bf\ox 1+K'\ox X)- (\bf\ox 1+K\ox X)g_b(\be\ox X)\\
=&(\sum_{n\geq 0}\frac{(-1)^nq^{\frac12n(n-1)}\be^n}{(q^n-q^{-n})...(q-q^{-1})}\ox X^n) (\bf\ox 1+K'\ox X)\\
&-(\bf\ox 1+K\ox X)\sum_{n\geq 0}\frac{(-1)^nq^{\frac12n(n-1)}\be^n}{(q^n-q^{-n})...(q-q^{-1})}\ox X^n\\
=&(\bf\ox 1)\sum_{n\geq 0}\frac{(-1)^nq^{\frac12n(n-1)}\be^n}{(q^n-q^{-n})...(q-q^{-1})}\ox X^n \\
&+(K'\ox X)\sum_{n\geq 0}\frac{(-1)^nq^{\frac12n(3+n)}\be^n}{(q^n-q^{-n})...(q-q^{-1})}\ox X^n \\
&-\sum_{n\geq 0}(q^{n}K'-q^{-n}K)\frac{(-1)^nq^{\frac12n(n+1)}\be^{n}}{(q^{n}-q^{-n})...(q-q^{-1})}\ox X^{n+1} \\
&-(\bf\ox 1+K\ox X)\sum_{n\geq 0}\frac{(-1)^nq^{\frac12n(n-1)}\be^n}{(q^n-q^{-n})...(q-q^{-1})}\ox X^n\\
=&0.
}
\end{proof} 

For simplicity, let us define 
\Eq{
Y_i^k:=\case{F_i^{n_i+1-k,-}& k\leq n_i,\\ F_i^{k-n_i,+}&k>n_i,}
}
such that $\bf_i = Y_i^1+Y_i^2+...+Y_i^{2n_i}$.
\begin{Lem}\label{Fbraid}
We have for $1\leq k\leq 2n_i$,
\Eqn{
&Ad_{g_b(\be_i \ox  Y_i^k)}(\bf_i\ox 1+K_i\ox \sum_{l=1}^{k-1} Y_i^l+K_i'\ox \sum_{l=k}^{2n_i} Y_i^l)\\
&=\bf_i\ox 1+K_i\ox \sum_{l=1}^{k} Y_i^l+K_i'\ox \sum_{l=k+1}^{2n_i} Y_i^l
}
and invariant under $Ad_{g_b(\be_j\ox Y_j^l)}$ for $j\neq i$ if $Y_j^l$ comes after $Y_i^{k-1}$ and before $Y_i^k$ in the decomposition \eqref{Rfac}.
\end{Lem}
\begin{proof} We observe that by Lemma \ref{FikLem},
\Eqn{
Ad_{g_b(\be_i\ox Y_i^k)}(K_i\ox Y_i^l)&=K_i\ox Y_i^l\tab l<k,\\
Ad_{g_b(\be_i\ox Y_i^k)}(K_i'\ox Y_i^l)&=K_i'\ox Y_i^l\tab l>k.
}
Hence we only care about the term $(\bf_i\ox 1 + K_i'\ox Y_i^k)$. By Lemma \ref{sl2com}
$$Ad_{g_b(\be_i\ox Y_i^k)}(\bf_i\ox 1+K_i'\ox Y_i^k)= \bf_i\ox 1+K_i\ox Y_i^k$$
and we are done.

Finally, again by Lemma \ref{FikLem} it is easy to check that $\be_j\ox Y_j^l$ commute with $$K_i\ox \sum_{l=1}^{k} Y_i^l+K_i'\ox \sum_{l=k+1}^{2n_i} Y_i^l$$ whenever $Y_j^l$ comes after $Y_i^{k-1}$ and before $Y_i^k$ in the decomposition \eqref{Rfac}.
\end{proof}

\begin{Lem}\label{Ebraid} For the reduced word $\bi=(i_1,...,i_N)\in\fR$, if $i_N=i$, then
$$\cK\til{R}\D(\be_i) = \D^{op}(\be_i)\cK\til{R}.$$
\end{Lem}
\begin{proof}
Note that if $i_N=i$, then $X_N^\pm=X_{f_i^{\pm n_i}}$, $\be_i=X_{f_i^{n_i}}+X_{f_i^{n_i},e_i^0}$ and $K_i=X_{f_i^{n_i},e_i^0,f_i^{-n_i}}$. We have
\Eqn{
X_{f_i^{- n_i}}X_{e_i^0} = q_i^2 X_{e_i^0}X_{f_i^{- n_i}}
}
Hence
\Eqn{
&Ad_{g_b(\be_i\ox X_N^-)} (1\ox \be_i+\be_i\ox K_i)\\
&=Ad_{g_b(\be_i\ox X_{f_i^{-n_i}})} (1\ox X_{f_i^{n_i}}+1\ox X_{f_i^{n_i},e_i^0}+e_i\ox K_i)\\
&=1\ox X_{f_i^{n_i}}+(1\ox X_{f_i^{n_i},e_i^0}+\be_i\ox K_i)(1+q_i\be_i\ox X_{f_i^{-n_i}}X_{e_i^0})\inv\\
&=1\ox X_{f_i^{n_i}}+1\ox X_{f_i^{n_i},e_i^0}(1+q_ie\b_i\ox X_{f_i^{-n_i}})(1+q_i\be_i\ox X_{f_i^{-n_i}}X_{e_i^0})\inv\\
&=1\ox \be_i.
}
One then check directly that $1\ox \be_i$ commutes with all the factors $\be_j\ox X_k^\pm$ for every $j,k$, except the last term $\be_i\ox X_N^+$, where we have the reverse of the above:
\Eqn{
Ad_{g_b(\be_i\ox X_N^+)}(1\ox \be_i) = 1\ox \be_i+\be_i\ox K_i',
}
and hence $$\cK\til{R}\D(\be_i)=\cK(1\ox \be_i+\be_i\ox K_i')\til{R} = \D^{op}(\be_i)\cK\til{R}$$ as required.
\end{proof}

In general, we use the fact that the decomposition of $\til{R}$ is invariant under the change of words $\cM$. 
Let $\what{F}_i^k$ denote the representation of $\bf_i$ using the mutated cluster variables $\what{X_i}:=\cM(X_i)$ under the change of words $\cM$ (cf. Section \ref{sec:mutation})

\begin{Lem}\label{Rchangewords}　\begin{itemize}
\item[(1)]For the change of words $\cM: (... iji...)\corr (...jij...)$ we have
$$g_b(\be_1\ox F_1^{k+1,\pm})g_b(\be_2\ox F_2^{l,\pm})g_b(\be_1\ox F_1^{k,\pm})=g_b(\be_2\ox \what{F}_2^{l+1,\pm})g_b(\be_1\ox \what{F}_1^{k,\pm})g_b(\be_2\ox \what{F}_2^{l,\pm})$$
for $v(i,k)<v(j,l)<v(i,k+1)$.
\item[(2)]
For the change of words $\cM: (... ijij...)\corr (...jiji...)$ where $i$ is short and $j$ is long, we have
\Eqn{
&g_{b_s}(\be_i\ox F_i^{k+1,\pm})g_b(\be_j\ox F_j^{l+1,\pm})g_{b_s}(\be_i\ox F_i^{k,\pm})g_b(\be_j\ox F_j^{l,\pm})\\
&=g_b(\be_j\ox \what{F}_j^{l+1,\pm})g_{b_s}(\be_i\ox \what{F}_i^{k+1,\pm})g_b(\be_j\ox \what{F}_j^{l,\pm})g_{b_s}(\be_i\ox F_i^{k,\pm})
}
for $v(j,l)<v(i,k)<v(j,l+1)<v(i,k+1)$.
\end{itemize}
\end{Lem}
\begin{proof} We will prove the $+$ case, while the $-$ case is similar. 

\textbf{Proof of (1).} In the simply-laced case, recall that we have
$$F_i^{k+1,+}F_i^{k,+}=q^2 F_i^{k,+}F_i^{k+1,+},$$
$$F_i^{k,+}F_j^{l,+}=q\inv F_j^{l,+}F_i^{k,+}.$$
Hence
\Eqn{
\frac{[\be_j\ox F_j^{l,+}, \be_i\ox F_i^{k,+}]}{q-q\inv} &= \be_j\be_i\ox F_j^{l,+}F_i^{k,+} -\be_i\be_j\ox F_i^{k,+}F_j^{l,+}\\
&=\frac{\be_j\be_i\ox  -q^{-1}\be_j\be_i}{q-q\inv}\ox F_j^{l,+}F_i^{k,+}\\
&=\be_{ij}\ox q^{-1/2}F_j^{l,+}F_i^{k,+}
}
where $\be_{ij}=T_i(\be_j)$ is given by the Lusztig's isomorphism.

Hence using \eqref{g12}, we have
\Eqn{&g_b(\be_i\ox F_i^{k+1,+})g_b(\be_j\ox F_j^{l,+})g_b(\be_i\ox F_i^{k,+})\\
&=g_b(\be_i\ox F_i^{k+1,+})g_b(\be_i\ox F_i^{k,+})g_b(\be_{ij}\ox q^{-1/2}F_j^{l,+}F_i^{k,+})g_b(\be_j\ox F_j^{l,+})\\
&=g_b(\be_i\ox (F_i^{k+1,+}+F_i^{k,+}))g_b(\be_{ij}\ox q^{-1/2}F_j^{l,+}F_i^{k,+})g_b(\be_j\ox F_j^{l,+}).
}
Similarly, we have
\Eqn{
&g_b(\be_j\ox \what{F}_j^{l+1,+})g_b(\be_i\ox \what{F}_i^{k,+})g_b(\be_j\ox \what{F}_j^{l,+})\\
&=g_b(\be_i\ox \what{F}_i^{k,+})g_b(\be_{ij}\ox q^{-1/2}\what{F}_j^{l+1,+}\what{F}_i^{k,+})g_b(\be_j\ox \what{F}_j^{l+1,+})g_b(\be_j\ox \what{F}_j^{l,+})\\
&=g_b(\be_i\ox \what{F}_i^{k,+})g_b(\be_{ij}\ox q^{-1/2}\what{F}_j^{l+1,+}\what{F}_i^{k,+})g_b(\be_j\ox (\what{F}_j^{l+1,+}+\what{F}_j^{l,+})).
}
If we write down the quantum cluster variables as
\Eqn{
F_i^k &= X_1, &F_i^{k+1}&=X_{1,2}, &F_j^l&=X_3,\\
\what{F}_i^k &= \what{X_1}, &\what{F_j}^l&=\what{X_3}, &\what{F_j}^{l+1}&=\what{X_{3,4}},
}
then we have 
\Eqn{
\what{X_1}&=X_1(1+qX_2),\\
\what{X_3}&=X_3(1+qX_2\inv)\inv,\\
\what{X_4}&=X_2\inv,
}
and one can see that 
\Eqn{
F_i^{k+1,+}+F_i^{k,+}&=\what{F}_i^{k,+},\\
F_j^{l,+}F_i^{k,+}&=\what{F}_j^{l+1,+}\what{F}_i^{k,+},\\
F_j^{l,+}&=\what{F}_j^{l+1,+}+\what{F}_j^{l,+}
}
as required.

\textbf{Proof of (2).} We have $F_i^{k,+}F_j^{l,+}=q\inv F_j^{l,+}F_i^{k,+}$ whenever $v(j,l)<v(i,k)$. Let
\Eqn{
v&=\be_i\ox F_i^{k,+},\\
u&=\be_j\ox F_j^{l,+},\\
\frac{c}{[2]_{q_s}}&=\frac{[u,v]}{q-q\inv}=\frac{q^{1/2}\be_j\be_i-q^{-1/2}\be_i\be_j}{q-q\inv}\ox q^{-1/2}F_j^{l,+}F_i^{k,+}=\be_Y\ox q^{-1/2}F_j^{l,+}F_i^{k,+},\\
d&=\frac{q_s\inv cv-q_s vc}{q-q\inv}=\frac{\be_Y\be_i-\be_i\be_Y}{q_s-q_s\inv}\ox q\inv F_j^{l,+}(F_i^{k,+})^2=\be_X\ox q\inv F_j^{l,+}(F_i^{k,+})^2,
}
where  $\be_X:=T_i(\be_j)$ and $\be_Y:=T_iT_j(\be_i)$ are given by the Lusztig's isomorphism. We have
\Eqn{
\be_Y\be_X&=q\be_X\be_Y,\\
\be_X\be_i&=q\be_i\be_X,\\
\be_j\be_Y&=q\be_Y\be_j, \\
\frac{[\be_j,\be_X]}{q-q\inv}&=\be_Y^2,
}
and hence $u,c,d,v$ satisfies the condition for \eqref{g1212}. Applying \eqref{g1212} repeatedly and rearranging, we have (we underline the terms to be transformed):
\begingroup
\allowdisplaybreaks
\Eqn{
&\underline{g_{b_s}(\be_i\ox F_i^{k+1,+})g_b(\be_j\ox F_j^{l+1,+})}\;\underline{g_{b_s}(\be_i\ox F_i^{k,+})g_b(\be_j\ox F_j^{l,+})}\\
=_\eqref{g1212}&g_b(\be_j\ox F_j^{l+1,+})g_{b_s}(\be_Y\ox q^{-1/2}F_j^{l+1,+}F_i^{k+1,+})g_b(\be_X\ox q\inv F_j^{l+1,+}(F_i^{k+1,+})^2)\underline{g_{b_s}(\be_i\ox F_i^{k+1,+})}\\
&\underline{g_b(\be_j\ox F_j^{l,+})}g_{b_s}(\be_Y\ox q^{-1/2}F_j^{l,+}F_i^{k,+})g_b(\be_X\ox q\inv F_j^{l,+}(F_i^{k,+})^2)g_{b_s}(\be_i\ox F_i^{k,+})\\
=_\eqref{g1212}&g_b(\be_j\ox F_j^{l+1,+})g_{b_s}(\be_Y\ox q^{-1/2}F_j^{l+1,+}F_i^{k+1,+})\underline{g_b(\be_X\ox q\inv F_j^{l+1,+}(F_i^{k+1,+})^2)g_b(\be_j\ox F_j^{l,+})}\\
&g_{b_s}(\be_Y\ox q^{-1/2}F_j^{l,+}F_i^{k+1,+})g_b(\be_X\ox q\inv F_j^{l,+}(F_i^{k+1,+})^2)\\
&\underline{g_{b_s}(\be_i\ox F_i^{k+1,+})g_{b_s}(\be_Y\ox q^{-1/2}F_j^{l,+}F_i^{k,+})}g_b(\be_X\ox q\inv F_j^{l,+}(F_i^{k,+})^2)g_{b_s}(\be_i\ox F_i^{k,+})\\
=_\eqref{gvu}&\underline{g_b(\be_j\ox F_j^{l+1,+})}g_{b_s}(\be_Y\ox q^{-1/2}F_j^{l+1,+}F_i^{k+1,+})\underline{g_b(\be_j\ox F_j^{l,+})}g_b(\be_Y^2\ox F_j^{l+1,+}F_j^{l,+}(F_i^{k+1,+})^2)\\
&g_b(\be_X\ox q\inv F_j^{l+1,+}(F_i^{k+1,+})^2)g_{b_s}(\be_Y\ox q^{-1/2}F_j^{l,+}F_i^{k+1,+})g_b(\be_X\ox q\inv F_j^{l,+}(F_i^{k+1,+})^2)\\
&g_{b_s}(\be_Y\ox q^{-1/2}F_j^{l,+}F_i^{k,+})g_{b_s}(\be_X\ox F_j^{l,+}F_i^{k,+}F_i^{k+1,+})\\
&\underline{g_{b_s}(\be_i\ox F_i^{k+1,+})}g_b(\be_X\ox q\inv F_j^{l,+}(F_i^{k,+})^2)\underline{g_{b_s}(\be_i\ox F_i^{k,+})}\\
=_\eqref{guv}&g_b(\be_j\ox (F_j^{l+1,+}+F_j^{l,+}))\\
&\underline{g_{b_s}(\be_Y\ox q^{-1/2}F_j^{l+1,+}F_i^{k+1,+})g_b(\be_Y^2\ox F_j^{l+1,+}F_j^{l,+}(F_i^{k+1,+})^2)g_{b_s}(\be_Y\ox q^{-1/2}F_j^{l,+}F_i^{k+1,+})}\\
&g_{b_s}(\be_Y\ox q^{-1/2}F_j^{l,+}F_i^{k,+})g_b(\be_X\ox q\inv F_j^{l+1,+}(F_i^{k+1,+})^2)\\
&\underline{g_b(\be_X\ox q\inv F_j^{l,+}(F_i^{k+1,+})^2)g_{b_s}(\be_X\ox F_j^{l,+}F_i^{k,+}F_i^{k+1,+})g_b(\be_X\ox q\inv F_j^{l,+}(F_i^{k,+})^2)}\\
&g_{b_s}(\be_i\ox(F_i^{k+1,+}+F_i^{k,+}))\\
=_\eqref{gb2}&g_b(\be_j\ox (F_j^{l+1,+}+F_j^{l,+}))\underline{g_{b_s}(\be_Y\ox q^{-1/2}(F_j^{l+1,+}+F_j^{l,+})F_i^{k+1,+})g_{b_s}(\be_Y\ox q^{-1/2}F_j^{l,+}F_i^{k,+})}\\
&\underline{g_b(\be_X\ox q\inv F_j^{l+1,+}(F_i^{k+1,+})^2)g_b(\be_X\ox q\inv F_j^{l,+}(F_i^{k+1,+}+F_i^{k,+})^2)}g_{b_s}(\be_i\ox(F_i^{k+1,+}+F_i^{k,+}))\\
=_\eqref{guv}&g_b(\be_j\ox (F_j^{l+1,+}+F_j^{l,+}))g_{b_s}(\be_Y\ox q^{-1/2}(F_j^{l+1,+}+F_j^{l,+})F_i^{k+1,+}+q^{-1/2}F_j^{l,+}F_i^{k,+})\\
&g_b(\be_X\ox q\inv  F_j^{l+1,+}(F_i^{k+1,+})^2+q\inv F_j^{l,+}(F_i^{k+1,+}+F_i^{k,+})^2)g_{b_s}(\be_i\ox(F_i^{k+1,+}+F_i^{k,+})),
}
where in the last line, we observe that the terms $q^2$ commute, hence we can apply \eqref{gb2}. 
\endgroup

On the other hand, by applying $\eqref{g1212}$ once, we have
\Eqn{
&g_b(\be_j\ox \what{F}_2^{l+1,+})g_{b_s}(\be_i\ox \what{F}_1^{k+1,+})g_b(\be_j\ox \what{F}_2^{l,+})g_{b_s}(\be_i\ox \what{F}_1^{k,+})\\
=&g_b(\be_j\ox (\what{F}_2^{l+1,+}+\what{F}_2^{l,+}))g_{b_s}(\be_Y\ox q^{-1/2}\what{F}_2^{l,+}\what{F}_1^{k+1,+})\\
&g_b(\be_X\ox q\inv \what{F}_2^{l,+}(\what{F}_1^{k+1,+})^2)g_{b_s}(\be_i\ox (\what{F}_1^{k+1,+}+\what{F}_1^{k,+})).
}
To compare, again we write out the quantum cluster variables as
\Eqn{
F_i^{k,+} &= X_1, &F_i^{k+1,+} &= X_{1,2},&F_j^{l,+} &= X_{3},&F_j^{l+1,+} &= X_{3,4},\\
\what{F}_i^{k,+} &= \what{X}_1, &\what{F}_i^{k+1,+} &= \what{X}_{1,2},&\what{F}_j^{l,+} &= \what{X}_{3},&\what{F}_j^{l+1,+} &= \what{X}_{3,4}.
}
Recall that we need to do mutation three times according to Section \ref{sec:mutation:doubly}, which gives at the end
\Eqn{
\what{X}_1&=D_2\inv X_{1,2,4},\\
\what{X}_2&=X_{2,4}\inv D_1,\\
\what{X}_3&=D_1\inv X_3 D_3,\\
\what{X}_4&=D_3\inv X_4,
}
where
\Eqn{
D_1&=(1+q_s X_2)(1+q_s^3 X_2)+q X_{2^2,4},\\
D_2&=(1+q_s X_2+q_s X_{2,4}),\\
D_3&=(1+q_s X_2+q_s X_{2,4})(1+q_s^3X_2+q_s^3X_{2,4}).
}
Now we can check directly that
\Eqn{
F_j^{l+1,+}+F_j^{l,+}&=\what{F}_2^{l+1,+}+\what{F}_2^{l,+},\\
F_j^{l+1,+}+F_j^{l,+})F_i^{k+1,+}+q^{-1/2}F_j^{l,+}F_i^{k,+}&=\what{F}_2^{l,+}\what{F}_1^{k+1,+},\\
F_j^{l+1,+}(F_i^{k+1,+})^2+ F_j^{l,+}(F_i^{k+1,+}+F_i^{k,+})^2&=\what{F}_2^{l,+}(\what{F}_1^{k+1,+})^2,\\
F_i^{k+1,+}+F_i^{k,+}&=\what{F}_1^{k+1,+}+\what{F}_1^{k,+}
}
and this completes the proof.
\end{proof}
\begin{Rem}
In type $G_2$, using the mutation sequence that gives the half-Dehn twist from Section \ref{sec:R:Ex}, one can conjugate the representation of $\D(\be_2)$ by \eqref{Rfac} and check the braiding relation directly. Using the fact that the standard form of the universal $R$ matrix is invariant under the change of words, we conclude that the analogue of Lemma \ref{Rchangewords} also holds in type $G_2$.
\end{Rem}
\begin{proof}[Proof of Theorem \ref{main2}] First it is obvious that $\cK$ and $\til{R}$ commute with both $\D(K_i)$ and $\D(K_i')$ by direct calculation.

As a consequence of Lemma \ref{Fbraid}, we have
\Eq{
\cK\til{R}\D(\bf_i) = \cK(\bf_i\ox 1+K'\ox \bf_i)\til{R} = \D^{op}(\bf_i)\cK\til{R}
}
as required.

As a consequence of Lemma \ref{Rchangewords}, we can choose freely the reduced word $\bi$ with any choice of index on the right of $\bi$, and by Lemma \ref{Ebraid}, we obtain
\Eqn{
\cK\til{R}\D(\be_i) = \D^{op}(\be_i)\cK\til{R}
} for every root index $i$, thus completing the proof of the braiding relations. 

Finally, recall that by the construction of the positive representations $\cP_\l$, one can choose appropriate discrete parameters $\l$ and restrict it to give any irreducible highest weight finite dimensional representations of $\cU_q(\g)$ \cite{Ip2}. Then $\cK\til{R}$ satisfies the braiding \eqref{braiding} on every finite dimensional representations of $\cU_q(\g)$, and as a formal power series it has constant term equals 1, hence we conclude that $\cK\til{R}$ equals the universal $R$ matrix.
\end{proof}

\section*{Acknowledgment}
\addcontentsline{toc}{section}{Acknowledgments}
I would like to thank Gus Schrader and Alexander Shapiro for stimulating discussions who inspired the constructions carried out in this work. I would also like to thank Masahito Yamazaki and Rei Inoue for valuable comments. This work is supported by JSPS KAKENHI Grant Numbers JP16K17571 and Top Global University Project, MEXT, Japan.
%==============================================================================
\appendix
\section{Quantum dilogarithm identities}\label{sec:dilog} 
The compact quantum dilogarithm function is defined to be the infinite product
\Eq{\Psi^q(x)=\prod_{r=0}^\oo (1+q^{2r+1}x) \inv,}
which is well defined for $0<q<1$. In the split real case, where $q=e^{\pi i b^2}$ with $0<b<1$, the infinite product is not so well-behaved. To treat this case, the non-compact quantum dilogarithm $g_b(x)$ is composed of two commuting copies, associated to the so-called \emph{Faddeev's modular double}, of the compact quantum dilogarithm $\Psi^q(x)$ \cite{Fa2, FKa}. It is a meromorphic function that can be represented as an integral expression:
\Eq{
g_b(x):=\exp\left(\frac{1}{4}\int_{\R+i0} \frac{x^{\frac{t}{ib}}}{\sinh(\pi bt)\sinh(\pi b\inv t)}\frac{dt}{t}\right),
}
such that by functional calculus, it is a unitary operator when $x$ is positive self-adjoint, and there is a $b$-duality:
\Eq{\label{bduality}
g_b(x)=g_{b\inv}(x^{\frac{1}{b^2}}).
}

In this paper however, we are only interested in the formal algebraic calculation, hence one may consider only the compact part and think about the correspondence in terms of formal power series
\Eq{
g_b(x)\sim \Psi^q(x)\inv=\prod_{r=0}^\oo (1+q^{2r+1}x) = Exp_{q^{-2}}\left(-\frac{u}{q-q\inv}\right),
}
where 
\Eq{
Exp_q(x)&:=\sum_{k\geq 0} \frac{x^k}{(k)_q!},\\
(k)_q&:=\frac{1-q^k}{1-q}.
}

In particular, we can rewrite the identities of $Exp_q(x)$ derived in \cite{KT} for the quantum dilogarithm function $g_b(x)$ that are needed in this paper. In particular, by writing in this way, the argument of $g_b(x)$ are all manifestly positive self-adjoint so that the identities are well-defined in the split real setting. 

We will be interested in two types of identities: the pentagon equation (PE) and  the quantum exponential relation (QE), 

\textbf{Simply-laced case.}
Let $u,v$ be self-adjoint variables. If $uv=q^2 vu$, then we have the pentagon equation and the quantum exponential relation:
\Eq{
(PE): \tab g_b(v)g_b(u)&=g_b(u)g_b(q\inv uv)g_b(v)\label{gvu},\\
(QE): \tab g_b(u+v)&=g_b(u)g_b(v)\label{guv}.
}
Let again $u,v$ be self-adjoint and $$c:=\frac{[u,v]}{q-q\inv},$$ such that $$uc=q^2cu,\tab cv=q^2vc.$$ Then we have the generalized pentagon equation:
\Eq{
(PE):\tab g_b(v)g_b(u)=g_b(u)g_b(c)g_b(v)\label{g12}.
}
in which \eqref{gvu} is a special case.

\textbf{Doubly-laced case.} In the doubly-laced case we have $q_s=q^{1/2}$. Let $u,v$ be self-adjoint variables, and let 
$$c:=\frac{[u,v]}{q_s-q_s\inv},\tab d:=\frac{q_s\inv cv-q_s vc}{q-q\inv},$$ such that 
$$uc=q^2 cu, \tab cd=q^2 dc, \tab dv=q^2 vd,\tab \frac{q\inv ud-qdu}{q-q\inv}=\frac{c^2}{[2]_{q_s}^2}.$$
We have
\Eq{\label{g1212}
(PE):\tab g_{b_s}(v)g_b(u)&=g_b(u)g_{b_s}(\frac{c}{[2]_{q_s}})g_b(d)g_{b_s}(v),\\
(QE):\tab g_{b_s}(c+v)&=g_{b_s}(c)g_b([2]_{q_s}d)g_{b_s}(v).
}
In particular if $uv=q^2vu$ and substitute $u\mapsto quv\inv/[2]_{q_s}$, we have:
\Eq{\label{gb2}
g_{b_s}(u+v)&=g_{b_s}(u)g_b(q\inv uv)g_{b_s}(v),\\
g_b((u+v)^2)&=g_b(u^2)g_{b_s}(q^{-1/2}uv)g_b(v^2).
}
These two relations are related by the $b$-duality \eqref{bduality}.

\textbf{Triply-laced case.} For completeness we also translate the type $G_2$ identity of \cite{KT} to $g_b(x)$, which becomes more natural looking. 

Let $q_s=q^{1/3}$, and let $u,v$ be self-adjoint. Define
\Eqn{
c&:=\frac{q_s^{-1}uv-q_svu}{q_s^2-q_s^{-2}},\\
d&:=\frac{q_s^{-2}cv-q_s^2vc}{q_s-q_s^{-1}},\\
d'&:=\frac{q_s^{-2}uc-q_s^2cu}{q_s-q_s^{-1}},
}
such that these relations are satisfied:
\Eqn{
ud'&=q^2d'u, & d'c&=q^2cd',& cd&=q^2dc,&dv&=q^2 vd,& \\
c^2&=\frac{q^{-1}ud-qdu}{q-q\inv},&c^2&=\frac{q^{-1}d'v-qvd'}{q-q\inv}, &c^3&=\frac{q^{-2}d'd-q^2 dd'}{q-q^{-1}}.
}

Then we have
\Eq{
(QE):\tab g_{b_s}(u+v)=&g_{b_s}(u)g_b(d')g_{b_s}(c)g_b(d)g_{b_s}(v).
}
In particular if $uv=q^2vu=q_s^6vu$, we have
\Eq{
g_{b_s}(u+v)&=g_{b_s}(u)g_b(q^{-2}u^2v)g_{b_s}(q\inv uv)g_b(q^{-2}uv^2)g_{b_s}(v),\\
g_b((u+v)^3)&=g_b(u^3)g_{b_s}(q^{-2}u^2v)g_b(q^{-3}u^3v^3)g_{b_s}(q^{-2}uv^2)g_b(v^3),
}
which are related by the $b$-duality \eqref{bduality}.

On the other hand, let $\be_1,\be_2$ be the generators of $\cU_q(\g_{G_2})$ with $\be_1$ long and $\be_2$ short, and $\ze_1\ze_2=q\inv \ze_2\ze_1$. Let the non-simple root generators be
\Eqn{
\be_W:=T_1(\be_2)&=\frac{[\be_2,\be_1]_{q_s^{3/2}}}{q_s^3-q_s^{-3}},\\
\be_X:=T_1T_2(\be_1)&=\frac{[\be_Y,\be_W]_{q_s^{-1/2}}}{q_s-q_s\inv},\\
\be_Y:=T_1T_2T_1(\be_2)&=\frac{[\be_2,\be_W]_{q_s^{1/2}}}{q_s^2-q_s^{-2}},\\
\be_Z:=T_1T_2T_1T_2(\be_1)&=\frac{[\be_2,\be_Y]_{q_s^{-1/2}}}{q_s-q_s\inv}.
}
Then we have (PE):
\Eq{
&g_{b_s}(\be_2\ox \ze_2)g_b(\be_1\ox \ze_1)\\
&=g_b(\be_1\ox \ze_1)g_{b_s}(\be_W\ox q^{1/2}\ze_1\ze_2)g_b(\be_X\ox q^3\ze_1^2 \ze_2^3)g_{b_s}(\be_Y\ox q\ze_1\ze_2^2)g_b(\be_Z\ox q^{3/2}\ze_1\ze_2^3)g_{b_s}(\be_2\ox \ze_2).\nonumber
}
%==============================================================================

\end{document}